\documentclass[reqno,11pt]{amsart}
\usepackage[T1]{fontenc}
\usepackage{amsmath,amsthm,amsfonts,amssymb,color,enumerate,bbm,todonotes, cancel}

\usepackage[all]{xy}
\usepackage[left=1in, right=1in, top=1.1in,bottom=1.1in]{geometry}

\usepackage{hyperref}
\hypersetup{
colorlinks   = true,
citecolor    = blue,
linkcolor=blue
}

\let\namerefOriginal\nameref

\renewcommand{\nameref}[1]{%
  {\hypersetup{linkcolor=black}
   \namerefOriginal{#1}}%
}

\newcommand{\namerefredbox}[1]{%
  \hyperref[{#1}]{%
    \begingroup
      \hypersetup{colorlinks=false, linkbordercolor=red}%
      \nameref*{#1}%
    \endgroup
  }%
}

\allowdisplaybreaks

\theoremstyle{plain}
\newtheorem{theorem}{Theorem}[section]
\newtheorem{corollary}[theorem]{Corollary}
\newtheorem{proposition}[theorem]{Proposition}
\newtheorem{lemma}[theorem]{Lemma}

\theoremstyle{definition}
\newtheorem{remark}[theorem]{Remark}
\newtheorem{definition}[theorem]{Definition}
\newtheorem{example}[theorem]{Example}

\numberwithin{equation}{section}

\newtheorem*{assumptions*}{\assumptionnumber}
\providecommand{\assumptionnumber}{}
\makeatletter
\newenvironment{assumptions}[1]
 {%
	 \renewcommand{\assumptionnumber}{Assumption #1}%
\begin{assumptions*}%
  \protected@edef\@currentlabel{#1}%
 }
 {%
  \end{assumptions*}
 }
\makeatother

\def\cD{\mathcal{D}}

\def\cF{\mathcal{F}}
\def\cH{\mathcal{H}}
\def\cL{\mathcal{L}}

\def\e{\varepsilon}
\def\E{\mathbb E}

\def\P{\mathbb P}

\newcommand{\R}{\mathbb{R}}
\newcommand{\N}{\mathbb{N}}
\newcommand{\W}{\dot{W}}
\newcommand{\ud}{\ensuremath{ \mathrm{d}} }

\newcommand{\Tr}{\mathrm{Tr}}

\newcommand{\FoxH}[5]{H_{#2}^{#1}\left(#3\:\middle\vert\: \begin{subarray}{l}#4\\[0.4em] #5\end{subarray}\right)}

\begin{document}

\title[]{Stochastic fractional diffusion equations with spatial Gaussian noise in the Stratonovich regime}

\begin{abstract}
For a class of stochastic fractional diffusion equations, we prove the existence of Stratonovich solutions and obtain the precise values of the Lyapunov exponents governing the large $t$-asymptotics of their $p$-th moments for all real $p\in\{1\}\cup[2,\infty)$, as well as the large $p$-asymptotics of the $p$-th moments.
\end{abstract}

\author[Y. Guo]{Yuhui Guo}
\address{Guangdong Provincial Key Laboratory of IRADS, and Department of Mathematical Sciences, Faculty of Science and Technology, Beijing Normal-Hong Kong Baptist University, Zhuhai 519087, China.}
\email{guoyuhui@mail.sdu.edu.cn}

\author[J. Song]{Jian Song}\address{Research Center for Mathematics and Interdisciplinary Sciences; Frontiers Science Center for Nonlinear Expectations, Ministry of Education, Shandong University, Qingdao 266237, China}
\email{txjsong@sdu.edu.cn}

\author[G. Wu]{Guanyu Wu}\address{School of Mathematical Sciences, Dalian University of
Technology, Dalian, Liaoning 116024, China}
\email{gywu@dlut.edu.cn}



\date{}

\maketitle

\hypersetup{linkcolor=black}   
\tableofcontents
\hypersetup{linkcolor=blue}   

\section{Introduction}

In this paper, we consider the following stochastic fractional diffusion equation (SFDE) in Stratonovich regime:
\begin{equation}\label{e:sfde}
	\begin{cases}
		\left(\partial^{b}+\dfrac{\nu}{2}\left(-\Delta\right)^{a / 2}\right) u(t, x)= \: I^{r}\left[\sqrt{\theta} u(t, x) \dot{W}(x) \right] , & t>0, x \in \mathbb{R}^d,   \\
		u(0, \cdot)=1,                                                                                                                                                           & \text { if } b\in(0,1], \\
		u(0, \cdot)=1, \quad \dfrac{\partial}{\partial t} u(0, \cdot)=0,                                                                                                      & \text { if } b \in (1,2],
	\end{cases}
\end{equation}
where $a\in(0,2]$, $b\in(0,2]$, $r\ge 0$, $\theta> 0$ and $\nu>0$. 
In \eqref{e:sfde}, $(-\Delta)^{a/2}$ is fractional Laplacian. $\partial^b$ denotes the Caputo fractional differential operator of order $b$ in time variable:
\begin{equation}\nonumber
  \partial^{b} f(t):=
\begin{cases}
  { \displaystyle   \dfrac{1}{\Gamma(n-b)}\int_{0}^{t}\frac{f^{(n)}(\tau)}{(t-\tau)^{b+1-n}}\ud\tau}, & \mbox{if } b \neq n, \\[1em] 
  \dfrac{\ud^n}{\ud t^n}f(t),   & \mbox{if }  b = n,
\end{cases}
\end{equation}
where $n=\lceil b \rceil$ is the smallest integer not smaller than $b$ and
 $\Gamma(x) = \int_{0}^{\infty}e^{-t}t^{x-1}\ud t$ is the Gamma function. We denote by
$I^r$  the Riemann-Liouville integral of order $r>0$ in the time variable:
\begin{equation*}
  (I^{r}f)(t):=\frac{1}{\Gamma(r)} \int_{0}^{t}f(r)(t-r)^{r-1} \ud r,
\end{equation*}
with the convention that $I^{0}$ is the identity operator when $r=0$.

The SFDE \eqref{e:sfde} driven by various types of Gaussian noise have been extensively studied in the literature, and we refer to \cite{Chen2017Nonlinear,Chen2023Interpolating,Chen2024Moments,Chen2019Nonlinear,Guo2024Stochastic,Guo2025Sample} for  recent developments. There are two main motivations for studying stochastic FDEs. The first motivation arises from the study of diffusions in
the viscoelastic media which exhibit both viscous and elastic
characteristics when undergoing deformation, such as honey and rubber, see \cite{Mainardi2010Fractional} for more details.  The second motivation is that such equations provide a unified framework encompassing two famous special cases. Specifically,  equation \eqref{e:sfde} reduces to the stochastic heat equation (SHE):
\begin{equation}\label{e:SHE}
\begin{cases}
   \displaystyle \left(\frac{\partial}{\partial t}-\frac{\nu}2 \frac{\partial^2}{\partial x^2}\right) u(t,x) = \sqrt{\theta} u(t,x) \dot{W}(x),\quad t>0, x \in \mathbb{R}^d,\\
    u(0,x) = 1,
\end{cases}
\end{equation}
when $a=2$, $b=1$ and $r=0$; and to  the stochastic wave equation (SWE):
\begin{equation}\label{e:SWE}
\begin{cases}
    \displaystyle \left(\frac{\partial^2}{\partial t^2}-\frac{\nu}2 \frac{\partial^2}{\partial x^2}\right) u(t,x) =  \sqrt{\theta} u(t,x) \dot{W}(x),\quad t>0, x \in \mathbb{R}^d,\\
    u(0,x)=0,\quad \dfrac{\partial}{\partial t}u(0,x)=1,
\end{cases}
\end{equation}
when $a=2$, $b=2$ and $r=0$.

In this paper, we assume that the noise $\W$ in \eqref{e:sfde} is a spatial Gaussian noise with covariance function
\begin{equation}\label{e:noise-covariance}
  \E\left[ \dot{W}(x)\dot{W}(y)\right]= \gamma(x-y).
\end{equation}
where $\gamma(\cdot)$ is a nonnegative and nonnegative-definite function, i.e.,  $\gamma(\cdot)$ has the representation
              \begin{equation*}
	\gamma(x)=\int_{\mathbb{R}^{d}} e^{\iota \xi \cdot x} \mu(\ud \xi), 
\end{equation*} 
where we use $\iota$ to denote the imaginary unit. For any function $\varphi,\psi\in L^1(\R^d)$, the Fourier transform of $\varphi(x)$ in the space variable is given by $\widehat{\varphi}(\xi)$ or $\cF \varphi(\xi):=\int_{\R^d} e^{-\iota x\cdot\xi}\varphi(x)\ud x$,
and the inverse Fourier transform is given by $    \cF^{-1}\psi(x):=\frac{1}{(2\pi)^d} \int_{\R^d} e^{\iota x\cdot\xi}\psi(\xi)\ud \xi$.

We impose the following assumptions on the covariance function. 

\begin{assumptions}{A}\label{A:Main}
The spectral measure $\mu$ satisfies Dalang's condition
	 \begin{equation}\label{dalang}
	\int_{\mathbb{R}^{d}} \frac{1}{1+|\xi|^a}\mu(\ud \xi)<\infty.
	\end{equation}
\end{assumptions}

\begin{assumptions}{B}\label{B:Main}
  \

	\begin{enumerate}[(i)]
		\item There exists $\alpha \in (0,d)$ such that
			\begin{equation}\label{E:scalef}
				\gamma(cx)=c^{-\alpha}\gamma(x), \quad \text{for any $c>0$, $x\in \R^d$}.
			\end{equation}
		\item There exists a nonnegative function $K$ on $\R^d$ such that $\gamma= K*K$,
			where ``$*$'' denotes the  convolution in space.
	\end{enumerate}
\end{assumptions}

\begin{remark}\nonumber
The scaling property \eqref{E:scalef}  of $\gamma(\cdot)$ implies that the spectral measure $\mu(\ud\xi)=\varphi(\xi)\ud\xi$ has the following scaling property:
	\begin{align}\nonumber
        \mu(cA)=c^{\alpha}\mu(A)
	\end{align}
  for all $c>0$ and $A\in \mathcal{B}(\R^d)$. Moreover,  Dalang's condition~\eqref{dalang} becomes $\alpha<a$, and hence $\alpha < \min \{a, d\}$. 
\end{remark}

The question of whether Dalang's condition \eqref{dalang} is both sufficient and necessary has consistently attracted substantial interest. For the SHE \eqref{e:SHE} driven by time-dependent noise in the Skorohod sense, Dalang's condition \eqref{dalang} is both necessary and sufficient for global solvability. A key feature of this condition is its independence of the temporal correlation of the noise, whether white-in-time \cite{dalang1999} or fractionally correlated \cite{bc2014}. In contrast, the Stratonovich SHE requires a stronger integrability condition whose exponent depends on the temporal regularity \cite{song2017}. 

For the SWE \eqref{e:SWE}, \cite{cdst2021} established necessary and sufficient conditions for the Skorohod solution under general temporally correlated noise. Regarding the Stratonovich solution, Dalang's condition \eqref{dalang} is necessary and sufficient when the noise is time-independent \cite{ch2024}, while the temporally correlated case where a stronger Danlang's condition was obtianed in \cite{chen2025}.

Recently, for SFDE \eqref{e:sfde}, when the noise is space-time white noise, it was shown in \cite{Chen2024Moments} that a unique Skorohod solution exists under Dalang's condition. In the time-independent noise case, a sufficient condition was derived in \cite{Chen2023Interpolating} under a scaling condition like Assumption \ref{B:Main}. 
To the best of our knowledge, existing studies on SFDE mainly focus on the Skorohod setting. In the present work, we extend these results to the Stratonovich setting.

As shown in \cite[Theorem 2.8]{Chen2024Moments}, the fundamental solution $G(t,x)$ for~\eqref{e:sfde} is given by
\begin{align}\label{E:Yab}
    G(t,x) = \pi^{-\frac d2} |x|^{-d}t^{b+r-1}
     \FoxH{2,1}{2,3}{\frac{|x|^a}{2^{a-1}\nu t^b}}
     {(1,1),\:(b+r,b)}{(d/2,a/2),\:(1,1),\:(1,a/2)}, \quad \text{for } b\in(0,2),
\end{align}
where $\FoxH{2,1}{2,3}{\cdots}{\cdots}{\cdots}$ refers to the Fox $H$-function (see Kilbas \cite{Kilbas2004H}). For $b=2$, the fundamental solution 
$G(t,x)$  takes the following explicit form when $a=2$ and $\gamma=0$,
\begin{equation*}
    G(t,x) = 
    \begin{cases}
         \dfrac{1}{2} \mathbf{1}_{\{|x|<t\}}, &\text{if } d=1,\vspace{0.1cm}\\
         \dfrac{1}{2\pi} \dfrac{1}{\sqrt{t^2-|x|^2}} \mathbf{1}_{\{|x|<t\}}, &\text{if } d=2,\\
         \dfrac{1}{4\pi t} \sigma_t, &\text{if } d=3,
    \end{cases}
\end{equation*}
where $\sigma_t$ is the surface measure on the sphere $\{x\in\R^3;|x|=t\}$, and $G(t,x)$ is a genuine distribution with compact support on $\R^d$ if $d\ge3$.

However, the Fourier transform of $G$ in space variable admits the following unified form
\begin{align}\label{E:FZ}
 \cF G(t,\cdot)(\xi) & = t^{b+r-1} E_{b,b+r} \left(-\tfrac12 \nu t^b |\xi|^a\right), \quad \text{for all } b\in(0,2],
\end{align}
where $E_{\varrho_1, \varrho_2}(z)$ is the two-parameter Mittag-Leffler function defined by
\begin{equation}\label{e:mlf}
  E_{\varrho_1, \varrho_2}(z):=\sum_{k=0}^{\infty}\frac{z^k}{\Gamma(\varrho_1 k+\varrho_2)},
\end{equation}
for all $\varrho_1>0$, $\varrho_2\in\mathbb{C}$ and $z\in\mathbb{C}$ (see, e.g., \cite[Sect.1.2]{Podlubny1999Fractional}). An application of \eqref{E:FZ}  leads to the following scaling property 
\begin{equation}\label{e:scal-G}
G(ct,x)=c^{b+r-1-db/a} G(t,c^{-b/a}x), \quad \text{for all } b\in(0,2].
\end{equation}

Throughout this paper, we impose the following non-negativity assumption on $G$.
\begin{assumptions}{G}\label{A:G}
One of the following two conditions holds:
\begin{enumerate}[(i)]
    \item the fundamental solution $G$ is  nonnegative and $b\in(0,2)$;
    \item $a=b=2$ and $r=0$.
\end{enumerate}
\end{assumptions}

\begin{remark}
    The fundamental solution $G$ is nonnegative if one of the following conditions is satisfied (see, e.g., \cite[Remark 1.2]{Chen2023Interpolating}):
 \begin{equation}\label{E:Pos}
\left\{
\begin{array}{r@{\quad}l@{\quad}l@{\quad}l}
(1) & a\in(0,2],\; b\in(0,1], & r\ge 0,       & \text{ if } d\geq1;       \\[2pt]
(2) & 1<b<a\leq2,            & r>0,          & \text{ if } 1\leq d\leq3; \\[2pt]
(3) & 1<a=b<2,              & r>\dfrac{d+3}{2}-b, & \text{ if } 1\leq d\leq3; \\[2pt]
(4) & a=b=2,                & r=0,          & \text{ if } 1\leq d\leq3.
\end{array}
\right.
\end{equation}
\end{remark}

We now present our first main result, which provides a sufficient condition for the existence of solutions; see Section~\ref{proof:exist} for the proof.
\begin{theorem}[Sufficient condition]\label{th:exist}
    Let Assumption \ref{A:G} hold. The SFDE \eqref{e:sfde} admits a square-integrable mild Stratonovich solution $u(t,x)$ in the sense of Definition~\ref{D: mild solution}, provided that 
    \begin{enumerate}[(i)]
        \item Assumption \ref{A:Main} holds;
        \item there exists $\delta\in\left(0,\frac{2(b+r)-1}{b}\right)$ such that 
        \begin{equation}\label{e:dalang-plus}
        \int_{\R^d}\left(\frac{1}{1+|\xi|^a}\right)^{\delta}\mu(\ud\xi)<\infty.
\end{equation}
    \end{enumerate}
    
Moreover, assuming the scaling property~\eqref{E:scalef}, the conditions (i) and (ii) are equivalent to 
\begin{equation}\label{e:strict-scaling}
0<\alpha<
\min\left\{
a,d, \frac{a}{b}\bigl(2(b+r)-1\bigr)
\right\},
\end{equation}
and then there exists a positive constant $C=C(\theta)$ such that
\begin{equation}\nonumber
\mathbb{E}u^2(t,x)
\leq
\exp\left(Ct^{\frac{2a(b+r)-b\alpha}{2a(b+r)-b\alpha-a}}\right).
\end{equation}
At the critical value
\begin{equation}
\alpha=\frac{a}{b}\bigl(2(b+r)-1\bigr)<\min\{a,d\}, \label{eq:scal-dalang-new}
\end{equation}
there exists $t_0=t_0(\theta)>0$ such that a square-integrable mild Stratonovich solution exists on $[0,t_0]$.
\end{theorem}

\begin{remark}\nonumber
The conditions (i) and (ii) in Theorem \ref{th:exist} can be expressed as
\begin{equation*}
    \int_{\R^d}\left(\frac{1}{1+|\xi|^a}\right)^{\min\left(\delta,1\right)}\mu(\ud\xi)<\infty,\quad \text{for some } \delta\in\left(0,\frac{2(b+r)-1}{b}\right).
\end{equation*}
Dalang's condition~\eqref{dalang} is sufficient for the existence of  a global  $L^2$-solution when $b+2r\geq1$.  If
$b+2r<1$, the alternative Dalang's condition~\eqref{e:dalang-plus} is sufficient. Note that the condition \eqref{e:dalang-plus} holds only if $2(b+r)>1$ which reduces to $b>1/2$ when $r=0$.
\end{remark}

\begin{remark}[On local $L^2$-solutions]
Under the scaling condition \eqref{E:scalef}, the condition \eqref{e:strict-scaling} reduces to
\begin{equation*}
    \begin{cases}
        \displaystyle 0<\alpha<\min\left\{\frac{a}{b}\bigl(2(b+r)-1\bigr),d \right\}, &\text{when }b+2r<1;\\
        0<\alpha<\min\{a,d\}, &\text{when }b+2r\ge 1.
    \end{cases}
\end{equation*}
When $b+2r<1$, Theorem~\ref{th:exist} guarantees that, at the critical value
$\alpha=\frac{a}{b}\bigl(2(b+r)-1\bigr)$,
there exists a local $L^2$-solution on $(0,t_0)$ for some $t_0=t_0(\theta)>0$,  provided that Dalang's condtion $\alpha<d$ holds. 
When $b+2r\ge 1$, at the critical value $\alpha=a$,  Theorem~\ref{th:necessity-dalang} below indicates that an $L^2$-solution exists only if Dalang's condition $\alpha<\min\{a,d\}$ is satisfied, and thus there is no (local) $L^2$-solution.
\end{remark}

The following result shows that  Dalang's condition~\eqref{dalang} is necessary for a (local) $L^2$-solution, see Section \ref{proof:nece} for the proof.

\begin{theorem}[Necessary condition]\label{th:necessity-dalang}
 Let Assumption \ref{A:G} hold.   If equation \eqref{e:sfde} admits a square-integrable mild Stratonovich solution with expansion \eqref{e:expan-utx} on some finite time interval, then necessarily Dalang's condition~\eqref{dalang} holds.
\end{theorem}

The following corollary is a direct consequence of Theorems \ref{th:exist} and \ref{th:necessity-dalang}.

\begin{corollary}
\nonumber
 Let Assumption \ref{A:G} hold. If $b+2r\geq1,$ then Dalang's condition \eqref{dalang} is necessary and sufficient for the existence of a square-integrable mild Stratonovich solution. 
\end{corollary}

\begin{remark}\nonumber
It remains to consider the range $b+2r<1$,
where necessary and sufficient conditions do not presently
coincide. Indeed, 
Theorem~\ref{th:necessity-dalang} shows that Dalang's condition is 
necessary in this range, while the sufficient condition is \eqref{e:dalang-plus}. In particular, under the scaling condition~\eqref{E:scalef}, the sufficient condition \eqref{e:dalang-plus} is equivalent to $0<\alpha<\min\left\{\frac{a}{b}\bigl(2(b+r)-1\bigr),d \right\}$, whereas the necessary condition is $0<\alpha<\min\{a,d\}$.
\end{remark}

Next, we will study the intermittency property of solution. 
The concept of intermittency property arises from physics. To be specific, a SHE system \eqref{e:SHE} is called intermittent if, after sufficiently long time, it develops a few large and steep peaks concentrated on small isolated islands, while the rest of the domain remains comparatively calm.
Zel'dovich et al. \cite{Zel1987Intermittency} provided a mathematically rigorous definition of intermittency in terms of upper and lower Lyapunov exponents
of $u(t,x)$ which are defined, respectively, by
\begin{equation*}
    \liminf_{t\to\infty}\frac{1}{\sigma(t)} \log \E[|u(t,x)|^p] \quad \text{and} \quad 
    \limsup_{t\to\infty}\frac{1}{\sigma(t)} \log \E[|u(t,x)|^p],
\end{equation*}
where $\lim_{t\to\infty}\sigma(t)=\infty$.  So far, a significant amount of research has been dedicated to the study of intermittency of solutions to SPDEs. We refer to \cite{bcc22,Chen2017Nonlinear,Chen2023Interpolating,Chen2024Moments,Chen2017Moment,Chen2018Temporal,Conus2013Intermittency,Dalang2009Intermittency,Hu2024Matching,Huang2017Large} for related results, which by no means is complete.

In this work we are interested in the precise value of the Lyapunov exponents.
For the study of SHE \eqref{e:SHE}, the Feynman–Kac formula serves as a particularly effective tool. By means of this formula,   \cite{Chen2015Precise} obtained the exact value of the Lyapunov exponent for the case of space–time white noise, and later extended the result to the setting of space–time fractional noise in \cite{Chen2017Moment}. \cite{Huang2017Large} derived the Feynman–Kac formula and computed Lyapunov exponents to SHE \eqref{e:SHE} driven by a time-white and space-colored noise, and  \cite{Huang2017Large2} studied the space-time correlated noise. For more applications of the Feynman–Kac formula to intermittency in heat equations, we refer the reader to \cite{Chen2019Parabolic,Chen2020Parabolic,Chen2017Spatial,Chen2018Temporal} and the references therein.

In the study of exact Lyapunov exponents for SWE \eqref{e:SWE}, because of the lack of Feynman–Kac formula, the Laplace transform method are frequently adopted. Initially, \cite{bs19} employed this technique to derive the second-order Lyapunov exponent for the wave equation driven by noise that is white in time and colored in space. Subsequently, \cite{bcc22} investigated the case of time-independent noise and obtained the precise value of the Lyapunov exponents using Laplace transform. Recently, \cite{ch2024} considered the SWE with time-independent noise under the Stratonovich regime, and established the exact Lyapunov exponent.

For the SFDE \eqref{e:sfde}, \cite{Chen2023Interpolating} investigated the exact Lyapunov exponents for time-independent noise using Laplace transform; subsequently, \cite{Chen2024Moments} employed fractional calculus to obtain the second-order Lyapunov exponent under space-time white noise.

Define the following notations:
 \begin{equation}\label{e:de-beta}
     \beta:=\frac{2 a(b+r)-b \alpha}{2 a(b+r)-b \alpha-a},
 \end{equation}
 and
$\mathbf M(a,b,r,\theta, \nu, \alpha)$ is given by
		\begin{equation}\label{e:bM}
			\begin{aligned}
				\mathbf M(a,b,r,\theta, \nu, \alpha)
				&:= \frac{1}{\beta-1} \left(\frac{2 a}{2 a(b+r)-b \alpha}\right)^\beta  \left(\theta\left(\frac{\nu}{2}\right)^{-\frac{\alpha}{a}} \mathcal{M}_a^{\frac{2 a-\alpha}{a}}\right)^{\beta-1},
			\end{aligned}
		\end{equation}
        where 
 \begin{equation}\label{e:de-Ma}
  \mathcal{M}_a:=   \sup _{g \in \mathcal{F}_{a,d}}\left\{\left(\int_{\mathbb{R}^d \times \mathbb{R}^d} \gamma(x-y) g^2(x) g^2(y) \ud x \ud y\right)^{1 / 2}-\int_{\mathbb{R}^d}|\xi|^a|\widehat{g}(\xi)|^2 \ud \xi\right\},
 \end{equation} 
 with
 \begin{equation}\label{e:cFad}
     \mathcal{F}_{a, d}:=\left\{g( x): \int_{\mathbb{R}^d} g^2( x) \ud x=1, \ \text { and } \int_{\mathbb{R}^d}|\xi|^a|\widehat{g}(\xi)|^2 \ud \xi <\infty\right\}.
 \end{equation}

We now provide our main result on precise asymptotics for $p$-moments with exact Lyapunov exponents. 

\begin{theorem}\label{th:excat}
Suppose Assumptions \ref{A:Main}, \ref{B:Main} and \ref{A:G} hold. Let $u(t,x)$ be the mild Stratonovich solution to \eqref{e:sfde}. Then the following asymptotic properties hold.

\begin{enumerate}[(i)]
    \item The first moment satisfies
    \[
    \lim _{t \rightarrow \infty} t^{-\beta} \log \mathbb{E}\left[u(t, x)\right]
    = \frac{1}{2} \mathbf M(a,b,r,\theta, \nu, \alpha).
    \]

    \item For every real $p\ge2$, the $p$-th moment satisfies
    \[
    \lim _{t \rightarrow \infty} t^{-\beta} \log \mathbb{E}\left[|u(t, x)|^p\right]
    = \frac{p^\beta}{2} \mathbf M(a,b,r,\theta, \nu, \alpha).
    \]

    \item For every fixed $t>0$, the moment growth in $p$ satisfies
    \[
    \lim _{p \rightarrow \infty} p^{-\beta} \log \mathbb{E}\left[|u(t, x)|^p\right]
    = \frac{t^\beta}{2} \mathbf M(a,b,r,\theta, \nu, \alpha).
    \]
\end{enumerate}
\end{theorem}

\begin{proof}
 We refer to the \nameref{proof:1} of Theorem \ref{th:excat} (i) in Section~\ref{se:upper}. For the proofs of (ii) and (iii), the upper and lower bounds follow from Theorem~\ref{T:upperbound} and Theorem~\ref{T:lowerbound1}, respectively.
\end{proof}

\begin{remark}\label{re:gap}
 In the proof of in \cite[Lemma 5.1]{ch2024}, a Lagrange multiplier argument is used to estimate the sum over $l_1,\dots,l_p$ and $k_1,\dots,k_p$. Specifically, it is claimed that the maximum of
\[
\prod_{j=1}^p \left(\frac{x_j+y_j}{x_j}\right)^{x_j}
\left(\frac{x_j+y_j}{y_j}\right)^{y_j}
\theta_j^{x_j+y_j}
\]
subject to $\sum x_j=n$, $\sum y_j=\frac{2-\alpha}{2}n$, $x_j,y_j>0$, is attained at $x_j=\theta_j n$, $y_j=\frac{2-\alpha}{2}\theta_j n$.  However, the optimization problem is not convex, and the claimed maximum point \((x_j,y_j)=(\theta_j n,\frac{2-\alpha}{2}\theta_j n)\) does not necessarily attain the global maximum over the domain. Consequently, the subsequent estimate in eq.~(5.11) and the lower bound in eq.~(5.22) therein still require justification.  This leaves a gap in the proof of the lower bound for the large-$t$ asymptotics in eq.~(1.9) therein.

In this work, we adopt a different approach to obtain the lower bound, which avoids the aforementioned issue and not only provides a proof of the lower bound but also extends the intermittency asymptotics from integer $p$ to all real $p\ge 2$; see also Remark~\ref{rem:lower}.
\end{remark}


Next, we illustrate the main results of Theorems \ref{th:exist}, \ref{th:necessity-dalang}, and \ref{th:excat} through concrete examples of  the SHE and SWE.
\begin{example}[SHE]
    When $a=2$, $b=1$ and $r=0$, SFDE \eqref{e:sfde}  reduces to SHE \eqref{e:SHE}. In this case, the mild Stratonovich solution $u(t,x)$ exists if and only if Dalang's condition \eqref{dalang} hold. 
    
    Under Assumptions \ref{A:Main} and \ref{B:Main}, we have the following asymptotics. Among them,  items (i) and (ii) coincide with \cite[Theorem 1.1]{Chen2018Temporal}, while item (iii)  appears to be new to the best of our knowledge.

    \begin{enumerate}[(i)]
      \item The first moment satisfies
    \[
    \lim _{t \rightarrow \infty} t^{-\frac{4-\alpha}{2-\alpha}} \log \mathbb{E}\left[u(t, x)\right]
    = \frac{2-\alpha}{4} 
\left(\frac{4}{4-\alpha}\right)^{\frac{4-\alpha}{2-\alpha}}
\theta^{\frac{2}{2-\alpha}}
\left(\frac{\nu}{2}\right)^{-\frac{\alpha}{2-\alpha}}
\mathcal{M}_a^{\frac{4-\alpha}{2-\alpha}}.
    \]

    \item For every real $p\ge2$, the $p$-th moment satisfies
    \[
    \lim _{t \rightarrow \infty} t^{-\frac{4-\alpha}{2-\alpha}} \log \mathbb{E}\left[|u(t, x)|^p\right]
    =p^\frac{4-\alpha}{2-\alpha} \frac{2-\alpha}{4} 
\left(\frac{4}{4-\alpha}\right)^{\frac{4-\alpha}{2-\alpha}}
\theta^{\frac{2}{2-\alpha}}
\left(\frac{\nu}{2}\right)^{-\frac{\alpha}{2-\alpha}}
\mathcal{M}_a^{\frac{4-\alpha}{2-\alpha}}.
    \]

    \item For every fixed $t>0$, the moment growth in $p$ satisfies
    \[
    \lim _{p \rightarrow \infty} p^{-\frac{4-\alpha}{2-\alpha}} \log \mathbb{E}\left[|u(t, x)|^p\right]
    =t^\frac{4-\alpha}{2-\alpha} \frac{2-\alpha}{4}  
\left(\frac{4}{4-\alpha}\right)^{\frac{4-\alpha}{2-\alpha}}
\theta^{\frac{2}{2-\alpha}}
\left(\frac{\nu}{2}\right)^{-\frac{\alpha}{2-\alpha}}
\mathcal{M}_a^{\frac{4-\alpha}{2-\alpha}}.
    \] 
    \end{enumerate}
\end{example}

\begin{example}[SWE]
    When $a=2$, $b=2$ and $r=0$, SFDE \eqref{e:sfde}  reduces to SWE \eqref{e:SWE}. In this case, the mild Stratonovich solution $u(t,x)$ exists if and only if Dalang's condition \eqref{dalang} hold. 
    
    Under Assumptions \ref{A:Main} and \ref{B:Main}, 
   we have the following asymptotics which recovers the exact asymptotics obtained by \cite{ch2024}.
    \begin{enumerate}[(i)]
      \item The first moment satisfies
    \[
    \lim _{t \rightarrow \infty} t^{-\frac{4-\alpha}{3-\alpha}} \log \mathbb{E}\left[u(t, x)\right]
    = \frac{3-\alpha}{2}
\left(\frac{2}{4-\alpha}\right)^{\frac{4-\alpha}{3-\alpha}}
\theta^{\frac{1}{3-\alpha}}
\left(\frac{\nu}{2}\right)^{-\frac{\alpha}{2(3-\alpha)}}
\mathcal{M}_a^{\frac{4-\alpha}{2(3-\alpha)}}.
    \]

    \item For every real $p\ge2$, the $p$-th moment satisfies
    \[
    \lim _{t \rightarrow \infty} t^{-\frac{4-\alpha}{3-\alpha}} \log \mathbb{E}\left[|u(t, x)|^p\right]
    = p^{\frac{4-\alpha}{3-\alpha}}\frac{3-\alpha}{2}
\left(\frac{2}{4-\alpha}\right)^{\frac{4-\alpha}{3-\alpha}}
\theta^{\frac{1}{3-\alpha}}
\left(\frac{\nu}{2}\right)^{-\frac{\alpha}{2(3-\alpha)}}
\mathcal{M}_a^{\frac{4-\alpha}{2(3-\alpha)}}.
    \]

    \item For every fixed $t>0$, the moment growth in $p$ satisfies
    \[
    \lim _{p \rightarrow \infty} p^{-\frac{4-\alpha}{3-\alpha}} \log \mathbb{E}\left[|u(t, x)|^p\right]
    = t^{\frac{4-\alpha}{3-\alpha}}\frac{3-\alpha}{2}
\left(\frac{2}{4-\alpha}\right)^{\frac{4-\alpha}{3-\alpha}}
\theta^{\frac{1}{3-\alpha}}
\left(\frac{\nu}{2}\right)^{-\frac{\alpha}{2(3-\alpha)}}
\mathcal{M}_a^{\frac{4-\alpha}{2(3-\alpha)}}.
    \] 
    \end{enumerate}
\end{example}

 To conclude the introduction, we summarize the main contributions of this paper: 

(1) We establish the existence of Stratonovich  solutions to the stochastic FDEs \eqref{e:sfde} and show that Dalang's condition~\eqref{dalang} is optimal for the existence of square-integrable Stratonovich solutions for a class of stochastic FDEs, including the SHE and the SWE.

(2) We obtain precise asymptotics for the $p$-th moments of the Stratonovich solution for all real $p\ge 2$, extending the result in~\cite{ch2024} concerning the $p$-th moments of the solution to the SWE for $p\in\mathbb N$. In particular, we fill the gap in the proof of the lower bound in \cite{ch2024} (see Remark~\ref{re:gap}).

More precisely, we establish a hypercontractivity property that allows us to derive an upper bound for the $p$-th moment, for any real $p\ge 2$, from the second moment. Hypercontractivity has previously been considered either for equations in the Skorohod sense or for fractional SHEs in the Stratonovich sense that admit a Feynman--Kac representation. In Proposition~\ref{prop:Hypercon}, we establish hypercontractivity for \eqref{e:sfde} in the Stratonovich setting, where no Feynman--Kac formula is available. We believe that this provides a new approach applicable to general stochastic equations whose solutions admit a series expansion; see also Remarks~\ref{rem:hyper-contract0} and~\ref{rem:hyper-contrac}.

Regarding the lower bound, we employ the Laplace transform developed in Lemma~\ref{L:shifted-stable-transform}, thereby not only filling the gap in the proof of \cite{ch2024} (see Remark \ref{re:gap}), but also strengthening the result from integer $p$ to all real $p\ge2$. See Remark~\ref{rem:lower} for details.

The rest of the paper is organized as follows.  Section \ref{se:Stratonovich} introduces the Stratonovich integral and the definition of a mild Stratonovich solution. In Section \ref{se:technical}, we present several preliminary results regarding the Stratonovich integral. Section \ref{se:existence} is dedicated to the proof of Theorems \ref{th:exist} and \ref{th:necessity-dalang}.  In Section~\ref{se:upper}, we establish the upper bound on the asymptotics of the $p$-th moment of the solution and also present the \nameref{proof:1} of Theorem \ref{th:excat} (i). The lower bound is then treated in Section~\ref{se:lower}.  Finally, some auxiliary results are provided in Appendix~\ref{ap:A}.


\section{Mild Stratonovich solution}\label{se:Stratonovich}

In this section, we introduce multiple Stratonovich integrals and formally express the solution to~\eqref{e:sfde} through a series expansion involving multiple Stratonovich integrals.

Let $\cH$ be the Hilbert space associated with the Gaussian noise $\dot{W}$ with covariance \eqref{e:noise-covariance}, i.e., the completion of smooth functions with compact support under the inner product 
\begin{align}
	\langle f,g\rangle_{\cH} 
	&=\int_{\R^d} \int_{\R^d} f(x) \gamma(x-y) g(y) \ud x\ud y
	= \int_{\R^d}  \cF f(\xi) \overline{\cF g(\xi)} \mu(\ud\xi). \nonumber
\end{align}
Let $\{W(f),f\in \cH \}$ be an isonormal Gaussian process with covariance $\E[W(f)W(g)] = \langle f,g\rangle_{\cH}.$ For $f\in \mathcal H$, we also write $W(f)$ as  $$W(f)=\int_{\R^d} f(y) \dot W(y) \ud y  = \int_{\R^d} f(y)  W(\ud y). $$

To properly define a mild Stratonovich solution, we first introduce the Stratonovich integral via a mollification procedure for the noise, following the approach in \cite{ch2024}. Let 
\begin{equation}\label{e:dot-we}
\dot{W}_{\varepsilon}(x):=\int_{\mathbb{R}^d} \dot{W}(y) p_{\varepsilon}(x-y) \ud y, \quad \varepsilon>0, \quad x \in \mathbb{R}^d,
\end{equation}
be the approximation of the noise $\dot W$, with the covariance given by
\begin{equation}\label{e:covar-dot-we}
	 \E\left[ \dot{W}_{\e}(x)\dot{W}_{\delta}(y) \right]
	= \gamma_{\e+\delta}(x-y),
\end{equation}
where $p_{\varepsilon}(x)=(2 \pi \varepsilon)^{-d / 2} \exp \left(-\frac{|x|^2}{2 \varepsilon}\right)$ is the heat kernel, and 
\begin{equation}\label{e:gamma-e}
	\gamma_{\varepsilon}(x)=\int_{\mathbb{R}^d} \gamma(y) p_{\varepsilon}(x-y) \ud y.
\end{equation}

\begin{definition} \label{D:Stratonovich integral}
	Given a random field $\{\Psi(x), x \in \mathbb{R}^d\}$ such that for all $\e>0$,
	$$
	\E\left[\left|\int_{\mathbb{R}^d} \Psi(x) \dot{W}_{\varepsilon}(x) \ud x\right|^2\right]<\infty, 
	$$
	we define the Stratonovich integral of $\Psi(x)$ by
	$$
	\int_{\mathbb{R}^d} \Psi(x) W(\ud x) := 
	\lim _{\varepsilon \rightarrow 0^{+}} \int_{\mathbb{R}^d} \Psi(x) \dot{W}_{\varepsilon}(x) \ud x
	$$
	whenever the limit on the right-hand side exists in $\mathcal{L}^2(\Omega, \mathcal{F}, \mathbb{P})$.
\end{definition}

Now, we are ready to define the solution to \eqref{e:sfde}.
\begin{definition} \label{D: mild solution}
	A  random field $\{u(t, x) ;(t, x) \in \mathbb{R}^{+} \times \mathbb{R}^d\}$ is called a mild Stratonovich solution to equation~\eqref{e:sfde}, if $\int_0^t G(t-s, x-y) u(s, y) \ud s$ is well-defined and Stratonovich integrable, and $u(t, x)$ satisfies
	\begin{equation*}
		u(t, x)=1+ \sqrt{\theta} 
		\int_{\mathbb{R}^d}\left[\int_0^t G(t-s, x-y) u(s, y) \ud s\right] W(\ud y),
	\end{equation*}
	where $G(t,x)$ is defined by \eqref{E:Yab} and the stochastic integral is in the sense of Definition~\ref{D:Stratonovich integral}.
\end{definition}

 We now construct a solution using Picard iteration.  Starting with $u_0(t,x)\equiv1$ and formally iterating the mild equation, we obtain the solution candidate
\begin{equation}\label{e:expan-utx}
	u(t, x)=\sum_{n=0}^{\infty} \theta^{\frac{n}{2}} S_n\left(g_n(\cdot, t, x)\right),
\end{equation}
	with $S_0(g_0)=1$.  More precisely, the iteration gives the recurrence
	\[
	S_{n+1}\bigl(g_{n+1}(\cdot,t,x)\bigr)
	=\int_{\R^d}\left[\int_0^t
	G(t-s,x-y)S_n\bigl(g_n(\cdot,s,y)\bigr)\ud s\right]W(\ud y).
	\]
	If Fubini's theorem holds, then iterating this relation and  making the substitutions
	$s_k=t-r_{n-k+1}$ and $x_k=y_{n-k+1}$, $k=1,\ldots,n$, yield
$$
\begin{aligned}
	& S_n\left(g_n(\cdot, t, x)\right) \\
	& =\int_{\left(\mathbb{R}^d\right)^n} \bigg[ \int_{[0, t]_{<}^n} G(t-r_n, y_n-x) \cdots G(r_2-r_1, y_2-y_1) \ud\mathbf{r} \bigg] W(\ud y_1) \cdots W(\ud y_n) \\
	& =\int_{\left(\mathbb{R}^d\right)^n} \left[\int_{[0, t]_{<}^n}  \bigg(\prod_{k=1}^n G(s_k-s_{k-1}, x_k-x_{k-1})\bigg) \ud\mathbf{s} \right] W(\ud x_1) \cdots W(\ud x_n),
\end{aligned}
$$
where $[0, t]_{<}^n:=\left\{\left(s_1, \cdots, s_n\right) \in[0, t]^n : 0<s_1<s_2<\cdots<s_n<t\right\}$, and  $x_0=x, s_0=0.$
	The required Fubini theorem for this computation is proved later in Lemma~\ref{le:fubini-th}.   This motivates defining $S_n\left(g_n(\cdot,t,x)\right)$ as the $n$-multiple Stratonovich integral of the function
\begin{equation} \label{E:gn}
	g_n(x_1, \cdots, x_n, t, x):=\int_{[0, t]_{<}^n}\bigg(\prod_{k=1}^n G\left(s_k-s_{k-1}, x_k-x_{k-1}\right)\bigg)\ud\mathbf{s} ,  
\end{equation}
 The Stratonovich integrability of $g_n(\cdot,t,x)$ is established
 later in Theorem~\ref{th:strat-intablity}, while the convergence of
 the series in \eqref{e:expan-utx} is proved in Theorem~\ref{th:exist}.

Similar to Definition~\ref{D:Stratonovich integral}, the multiple Stratonovich integration is defined as follows.

\begin{definition} \label{D:multiple Stratonovich integration} 
		Let $f:(\R^d)^n\to\R$ be measurable.  Given
	$\e_1,\ldots,\e_n>0$, denote
	$\e=(\e_1,\ldots,\e_n)$ and set
	\begin{equation*}
				S_{n,\e}(f)
	:=\int_{(\R^d)^n}f(x_1,\ldots,x_n)
	\prod_{k=1}^n\dot W_{\e_k}(x_k)
	\ud x_1\cdots\ud x_n.
	\end{equation*}
	Assume that each $S_{n,\e}(f)$ belongs to
	$\mathcal{L}^2(\Omega,\mathcal{F},\mathbb{P})$.  We
	define the $n$-multiple Stratonovich integral of $f$ by
	$$
	\begin{aligned}
		S_n(f):=\int_{(\R^d)^n}f(x_1,\ldots,x_n)W(\ud x_1)\cdots W(\ud x_n)
	=\lim_{\e\to0}S_{n,\e}(f),
	\end{aligned}
	$$
	whenever the limit exists in $\mathcal{L}^2(\Omega, \mathcal{F}, \P)$.
\end{definition} 

For completeness, the Hu--Meyer formula
	\cite{HM1988} states that, for symmetric $f$,
	\begin{equation}\label{e:hu-meyer}
		S_n(f)
		=\sum_{k=0}^{\lfloor n/2\rfloor}
		\frac{n!}{2^k k!(n-2k)!}
		I_{n-2k}\bigl(\Tr^k f\bigr),
	\end{equation}
	where $I_{n-2k}$ is the $(n-2k)$-multiple Skorohod integral, and
	$\Tr^k f$ is the $k$th trace, defined as the limit in
	$\cH^{\otimes(n-2k)}$ of the corresponding regularized contractions
	whenever this limit exists.  Although~\eqref{e:hu-meyer}
	agrees with the mollification definition in the above when all the traces exist,
	their existence is not automatic. We therefore work directly with
	Definition~\ref{D:multiple Stratonovich integration} in the present paper. We refer to \cite[Appendix A.2]{ch2024} for further details about multiple Stratonovich integrals and the Hu--Meyer formula.

In this work, we mainly focus on the case where the integrand $f$ is deterministic. To compute expectations of such multiple Stratonovich integrals, we recall Wick's formula \cite[p.~201, Lemma~5.2.6]{MR2006}.  If $(Z_1,\ldots,Z_{2n})$ is a centered Gaussian vector, then
\begin{equation}\label{e:wick-formula}
	\left\{
	\begin{aligned}
		\E\prod_{k=1}^{2n}Z_k
		&=\sum_{\mathcal D\in\Pi_{2n}}
		\prod_{(j,k)\in\mathcal D}\E[Z_jZ_k],\\
		\E\prod_{k=1}^{2n-1}Z_k&=0,
	\end{aligned}
	\right.
\end{equation}
where $\Pi_{2n}$ denotes the set of pair partitions of $\{1,\ldots,2n\}$.  As a side remark, the size of $\Pi_{2n}$ is $\#\left(\Pi_{2n}\right)=\frac{(2 n)!}{2^{ n}n!}=(2n-1)!!$.

	We first consider the mollified integrals.  Let
	$f:(\R^d)^{2n}\to\R$ be a measurable function.   By Wick's formula, for every
	$\e=(\e_1,\ldots,\e_{2n})\in(0,\infty)^{2n}$,
	\begin{align}
		\E S_{2n,\e}(f)
		&=\sum_{\mathcal D\in\Pi_{2n}}
		\int_{(\R^d)^{2n}}f(x_1,\ldots,x_{2n})
		\prod_{(j,k)\in\mathcal D}
		\gamma_{\e_j+\e_k}(x_j-x_k)
		\ud x_1\cdots\ud x_{2n}, \nonumber
	\end{align}
	If, in addition, $f$ is
	Stratonovich integrable in the sense of
	Definition~\ref{D:multiple Stratonovich integration}, then its
	defining $\mathcal{L}^2$-convergence gives
	\begin{align}\label{E:Sf2n}
		\E S_{2n}(f)
		&=\lim_{\e\rightarrow0}
		\sum_{\mathcal D\in\Pi_{2n}}
		\int_{(\R^d)^{2n}}f(x_1,\ldots,x_{2n})
		\prod_{(j,k)\in\mathcal D}
		\gamma_{\e_j+\e_k}(x_j-x_k)
		\ud x_1\cdots\ud x_{2n}.
	\end{align}
	Similarly, if $h:(\R^d)^{2n-1}\to\R$ is measurable and
	Stratonovich integrable in the sense of
	Definition~\ref{D:multiple Stratonovich integration}, then
	\begin{equation} 
		\E S_{2n-1}(h)=0.\nonumber
	\end{equation}

	Now let $f,g:(\R^d)^n\to\R$ be measurable and Stratonovich
	integrable in the sense of
	Definition~\ref{D:multiple Stratonovich integration}.  Writing
	$\e=(\e_1,\ldots,\e_n)$ and
	$\e'=(\e_{n+1},\ldots,\e_{2n})$, we have
	\[
	\lim_{\e\to0}S_{n,\e}(f)=S_n(f),\qquad
	\lim_{\e'\to0}S_{n,\e'}(g)=S_n(g)
	\quad\text{in }\mathcal{L}^2(\Omega,\mathcal F,\P).
	\]
	By the Cauchy--Schwarz inequality, the corresponding products converge
	in $\mathcal{L}^1(\Omega,\mathcal F,\P)$.  Hence, applying Wick's
	formula to the regularized expectation yields
	\begin{equation*}
		\begin{aligned}
			\E\left[S_n(f)S_n(g)\right]
			&=\lim_{\e,\e'\to0}
			\E\left[S_{n,\e}(f)S_{n,\e'}(g)\right]\\
			&=\lim_{\e,\e'\to0}
			\sum_{\mathcal D\in\Pi_{2n}}
			\int_{(\R^d)^{2n}}
			f(x_1,\ldots,x_n)g(x_{n+1},\ldots,x_{2n})\\
			&\qquad\times
			\prod_{(j,k)\in\mathcal D}\gamma_{\e_j+\e_k}(x_j-x_k)
			\ud x_1\cdots\ud x_{2n}.
		\end{aligned}
	\end{equation*}

\section{Some technical results}\label{se:technical}

\subsection{Laplace transforms}\label{sec:laplace} 
In this section, we present several results concerning the Laplace transforms of $S_n(g_n(\cdot,t,x))$, which appears in the Stratonovich expansion \eqref{e:expan-utx} of the solution. We note that the Laplace transform provides a useful link between the heat equation and other fractional PDEs and has been used effectively in \cite{bs19,bcc22,ch2024} to study stochastic wave equations through their connection with stochastic heat equations.

Let $X(t), X_1(t), X_2(t), ...$ be independent  $d$-dimensional symmetric $a$-stable processes, which are independent of the noise $\dot{W}$. We denote by $\mathbf{E}_x$ the expectation with respect to the $a$-stable process  starting from $x$.

Given  $\varepsilon_1, \dots, \varepsilon_{2 n}>0$, we denote $\varepsilon=\left(\varepsilon_1, \cdots, \varepsilon_n\right)$  and  $\varepsilon^{\prime}=\left(\varepsilon_{n+1}, \cdots, \varepsilon_{2 n}\right)$.  Set
\begin{equation}\label{e:S-ne}
	S_{n, \varepsilon}\left(g_n(\cdot, t, x)\right):=\int_{\left(\mathbb{R}^d\right)^n} g_n(x_1, \cdots, x_n, t, x) \bigg(\prod_{k=1}^n \dot{W}_{\varepsilon_k}\left(x_k\right)\bigg) \ud x_1 \cdots \ud x_n,
\end{equation}
    where $g_n$  and $\dot W_\varepsilon$ are defined is given in \eqref{E:gn} and \eqref{e:dot-we}, respectively. By \eqref{e:wick-formula}, the covariance is 
\begin{equation}\label{e:cov-Sne}
    \E \left[ S_{n, \varepsilon}\left(g_n(\cdot, t_1, 0)\right)S_{n, \varepsilon'}\left(g_n(\cdot, t_2, 0)\right)\right]=\sum_{\mathcal D\in \Pi_{2n}} F_{\varepsilon, \varepsilon^{\prime}}^{\mathcal{D}}\left(t_1, t_2\right),
    \end{equation}
    where   
    \begin{equation}
        \label{eq:Fvarepsilon}
    \begin{aligned}
		F_{\varepsilon, \varepsilon^{\prime}}^{\mathcal{D}}\left(t_1, t_2\right):=&\int_{\left(\mathbb{R}^d\right)^{2 n}} g_n(x_1, \cdots, x_n, t_1, 0) \: g_n(x_{n+1}, \cdots, x_{2 n}, t_2, 0)  \\
		& \qquad \times \bigg(\prod_{(j, k) \in \mathcal{D}} \gamma_{\e_j+\e_k}\left(x_j-x_k\right)\bigg) \ud x_1 \cdots \ud x_{2 n}. 
	\end{aligned}
    \end{equation}

    \begin{lemma}\label{le:3.1}
	Let $n\in\N$ be fixed. Under Assumption \ref{A:Main}, the following assertions hold:
    \begin{enumerate}[(i)]
        \item For any $\varepsilon_1, \cdots, \varepsilon_n>0$ and $\lambda>0$, we have
	\begin{equation}\label{E:laplace S}
    \begin{aligned}
		& \int_0^{\infty} e^{-\lambda t}\: S_{n, \e}\left(g_n(\cdot, t, x)\right) \ud t\\
		& =\lambda^{b-nr-1} \left(\frac{2}{\nu}\right)^{n+1}   \int_0^{\infty} e^{-2\lambda^b t/\nu} \bigg(\int_{[0, t]_{<}^n} \mathbf{E}_x \left[\prod_{k=1}^n \dot{W}_{\e_k}(X(s_n))\right]  \ud s_1 \cdots \ud s_n\bigg)  \ud t, \quad \text {a.s.} 
	\end{aligned}
     \end{equation}

    \item For any $\lambda_1, \lambda_2>0$, we have
	$$
	\begin{aligned}
		& \int_0^{\infty} \int_0^{\infty} e^{-\lambda_1 t_1-\lambda_2 t_2} \: F_{\e, \e^\prime}^{\cD} (t_1, t_2) \: \ud t_1 \ud t_2 \\
		& \leq  (\lambda_1\lambda_2)^{b-nr-1} \left( \frac{2}{\nu} \right)^{2n+2}  \int_0^{\infty} \int_0^{\infty} \ud t_1 \ud t_2 \: e^{-2\lambda_1^{b} t_1/\nu-2\lambda_2^{b} t_2/\nu} \\
		&\quad \times \mathbf{E}_0 \int_{\left[0, t_1\right]_{<}^n \times\left[0, t_2\right]_{<}^n} \ud s_1 \cdots \ud s_{2 n} \prod_{(j, k) \in \mathcal{D}} \gamma\left(X_{v(j)}\left(s_j\right)-X_{v(k)}\left(s_k\right)\right),
    \end{aligned}
	$$
	where 
    \begin{equation*}
    v(k)=
        \begin{cases}
            1, &\text{if } 1 \leq k \leq n,\\[5pt]
            2, &\text{if } n+1 \leq k \leq 2n.\\
        \end{cases}
    \end{equation*}

    \item  For any $\lambda_1, \lambda_2>0$, we have
	$$
	\begin{aligned}
		& \lim _{\e, \e' \rightarrow 0} \int_0^{\infty} \int_0^{\infty} e^{-\lambda_1 t_1-\lambda_2 t_2} F_{\e, \e^\prime}^{\mathcal{D}}\left(t_1, t_2\right) \ud t_1 \ud t_2 \\
		&= (\lambda_1\lambda_2)^{b-nr-1} \left( \frac{2}{\nu} \right)^{2n+2}  \int_0^{\infty} \int_0^{\infty} \ud t_1 \ud t_2 \: e^{-2\lambda_1^{b} t_1/\nu-2\lambda_2^{b} t_2/\nu} \\
		& \quad \times \mathbf{E}_0 \int_{\left[0, t_1\right]_{<}^n \times\left[0, t_2\right]_{<}^n} \ud s_1 \cdots \ud s_{2 n} \prod_{(j, k) \in \mathcal{D}} \gamma\left(X_{v(j)}\left(s_j\right)-X_{v(k)}\left(s_k\right)\right).
	\end{aligned}
	$$
    \end{enumerate}
\end{lemma}

\begin{proof}
\textbf{Part (i).}
By Lemma \ref{L:laplace G}, we obtain
\begin{align*}
    &\int_0^{\infty} e^{-\lambda t} \mathcal{F} G(t, \cdot)(\xi) \mathrm{d} t = \frac{1}{\lambda^r(\lambda^b + \frac{\nu}{2}|\xi|^a)}\\
    &=\frac{2}{\nu\lambda^r} \int_0^{\infty} e^{-2\lambda^bt/\nu } e^{-|\xi|^a t}\ud t = \frac{2}{\nu\lambda^r} \int_0^{\infty} e^{-2\lambda^b t/\nu } \mathcal Fp(t, \cdot)(\xi)  \ud t ,
\end{align*}
 where 
$p(t, x) = \frac{1}{(2\pi)^d} \int_{\mathbb{R}^d} e^{\iota \xi x} e^{-|\xi|^a t} \ud \xi$
 is the density of an $a$-stable process $X(t)$. This together with Lemma~\ref{le:chenbook-lem227} yields
\begin{equation}\label{E:laplace g}
\begin{aligned}
	& \int_0^{\infty} e^{-\lambda t} g_n(x_1, \cdots, x_n, t, x) \: \ud t \\
	& = \int_0^{\infty} \ud t\: e^{-\lambda t} \int_{[0, t]_{<}^n} \ud s_1 \cdots \ud s_n \prod_{k=1}^n G(s_k - s_{k-1}, x_k - x_{k-1}) \\
	& = \lambda^{-1} \prod_{k=1}^n \int_0^{\infty} e^{-\lambda t} G(t, x_k - x_{k-1}) \ud t = \left( \frac{2}{\nu\lambda^r} \right)^n \lambda^{-1} \prod_{k=1}^n \int_0^{\infty} e^{-2\lambda^b t/\nu } p(t,  x_k-x_{k-1}) \: \ud t \\
    & = \lambda^{b-nr-1} \left( \frac{2}{\nu} \right)^{n+1} \int_0^{\infty} \ud t e^{-2\lambda^b t/\nu } \int_{[0, t]_{<}^n} \ud s_1 \cdots \ud s_n \prod_{k=1}^n p(s_k - s_{k-1}, x_k - x_{k-1}). 
\end{aligned}
\end{equation}
Therefore, we have
$$
\begin{aligned}
	& \int_0^{\infty} e^{-\lambda t} S_{n, \varepsilon}(g_n(\cdot, t, x)) \ud t \\
	&= \int_0^{\infty} \ud t\: e^{-\lambda t} \int_{(\mathbb{R}^d)^n} \ud x_1 \cdots \ud x_n \: g_n(x_1, \cdots, x_n, t, x) \prod_{k=1}^n \dot{W}_{\varepsilon_k}(x_k) \\
	&= \lambda^{b-1} \left( \frac{2}{\nu} \right)^{n+1} \left( \frac{1}{\lambda^r} \right)^n \int_0^{\infty} \ud t \: e^{-2\lambda^b t/\nu} \int_{[0, t]_{<}^n} \ud s_1 \cdots \ud s_n \\
	& \quad \times \int_{(\mathbb{R}^d)^n} \ud x_1 \cdots \ud x_n \bigg(\prod_{k=1}^n p(s_k - s_{k-1}, x_k - x_{k-1}) \dot{W}_{\varepsilon_k}(x_k)\bigg).
\end{aligned}
$$
Now, observe that for fixed $(s_1, \cdots, s_n) \in [0, t]_{<}^n$, the random vector $(X(s_1), \cdots, X(s_n))$ starting from $x$ has the joint density:
$$
f_{s_1, \cdots, s_n}(x_1, \cdots, x_n) = \prod_{k=1}^n p(s_k - s_{k-1}, x_k - x_{k-1}).
$$
Hence, the inner integral can be expressed as an expectation:
$$
\int_{(\mathbb{R}^d)^n} \ud x_1\cdots \ud x_n \prod_{k=1}^n p(s_k - s_{k-1}, x_k - x_{k-1}) \prod_{k=1}^n \dot{W}_{\varepsilon_k}(x_k) = \mathbf{E}_x \prod_{k=1}^n \dot{W}_{\varepsilon_k}(X(s_k)).$$
This completes the proof of (i).

\textbf{Part (ii).}
Applying \eqref{E:laplace g} to see that
$$
\begin{aligned}
	& \int_0^{\infty} \int_0^{\infty} e^{-\lambda_1 t_1 - \lambda_2 t_2} F_{\e, \e'}^{\mathcal{D}}(t_1, t_2) \: \ud t_1 \ud t_2 \\
	& = (\lambda_1\lambda_2)^{b-nr-1} \left( \frac{2}{\nu} \right)^{2n+2}  \int_0^{\infty} \int_0^{\infty} \ud t_1 \ud t_2 \: e^{-2\lambda_1^{b} t_1/\nu-2\lambda_2^{b} t_2/\nu} \\
	& \quad \times \int_{[0, t_1]_{<}^n \times [0, t_2]_{<}^n} \ud s_1 \cdots \ud s_{2n} \int_{(\mathbb{R}^d)^{2n}} \ud x_1 \cdots \ud x_{2n} \bigg(\prod_{(j,k) \in \mathcal{D}} \gamma_{\e_j + \e_k}(x_j - x_k) \bigg) \\
	& \quad \times \bigg( p(s_1, x_1) \prod_{k=2}^n p(s_k - s_{k-1}, x_k - x_{k-1}) \bigg) \bigg( p(s_{n+1}, x_{n+1}) \prod_{k=n+2}^{2n} p(s_k - s_{k-1}, x_k - x_{k-1}) \bigg).
\end{aligned}
$$
For fixed $(s_1, \cdots, s_{2n})$, the function
\begin{align*}
     \bigg( p(s_1, x_1) \prod_{k=2}^n p(s_k - s_{k-1}, x_k - x_{k-1}) \bigg) \bigg( p(s_{n+1}, x_{n+1}) \prod_{k=n+2}^{2n} p(s_k - s_{k-1}, x_k - x_{k-1}) \bigg)
\end{align*}
is the density of the random vector $(X_1(s_1), \cdots, X_1(s_n); X_2(s_{n+1}), \cdots, X_2(s_{2n}))$ starting from $0$. Therefore, we can write
\begin{align}
	& \int_0^{\infty} \int_0^{\infty} e^{-\lambda_1 t_1 - \lambda_2 t_2} F_{\e, \e'}^{\mathcal{D}}(t_1, t_2) \: \ud t_1 \ud t_2  \nonumber \\
	& = (\lambda_1\lambda_2)^{b-nr-1} \left( \frac{2}{\nu} \right)^{2n+2}  \int_0^{\infty} \int_0^{\infty} \ud t_1 \ud t_2 \: e^{-2\lambda_1^{b} t_1/\nu-2\lambda_2^{b} t_2/\nu} \nonumber\\
	& \quad \times \int_{[0, t_1]_{<}^n \times [0, t_2]_{<}^n} \ud s_1 \cdots \ud s_{2n} \: \mathbf{E}_0 \prod_{(j,k) \in \mathcal{D}} \gamma_{\e_j + \e_k} \left(X_{v(j)}(s_j) - X_{v(k)}(s_k)\right). \label{E:laplace F}
\end{align}

Recalling that $\gamma_\e$ is given in \eqref{e:gamma-e}, the Fourier transform and Assumption \ref{A:Main}  yield
$$
\begin{aligned}
	 &\mathbf{E}_0 \prod_{(j,k) \in \mathcal{D}} \gamma_{\e_j + \e_k} \left(X_{v(j)}(s_j) - X_{v(k)}(s_k)\right) \nonumber\\
	  &=\int_{(\mathbb{R}^d)^n} e^{ -\sum_{(j,k) \in \mathcal{D}} \frac{\e_j + \e_k}{2} |\xi_{j,k}|^2}
 \mathbf{E}_0 e^{ \iota \sum_{(j,k) \in \mathcal{D}} \xi_{j,k} \left(X_{v(j)}(s_j) - X_{v(k)}(s_k)\right)} \prod_{(j,k) \in \mathcal{D}} \mu(\ud\xi_{j,k}).
\end{aligned}
$$
Noting that $X$ is a symmetric $\alpha$-stable process real nonnegative  characteristic function, we have 
$\mathbf{E}_0 e^{\iota \sum_{(j, k) \in \mathcal{D}} \xi_{j, k} \left(X_{v(j)}\left(s_j\right)-X_{v(k)}\left(s_k\right)\right)}\geq 0.$   Thus, we get the following upper bound:
$$
\begin{aligned}
	& \int_0^{\infty} \int_0^{\infty} e^{-\lambda_1 t_1 - \lambda_2 t_2} F_{\e, \e'}^{\mathcal{D}}(t_1, t_2) \: \ud t_1 \ud t_2 \\
	&\leq (\lambda_1\lambda_2)^{b-nr-1} \left( \frac{2}{\nu} \right)^{2n+2}  \int_0^{\infty} \int_0^{\infty} \ud t_1 \ud t_2 \: e^{-2\lambda_1^{b} t_1/\nu-2\lambda_2^{b} t_2/\nu} \int_{[0, t_1]_{<}^n \times [0, t_2]_{<}^n} \ud s_1 \cdots \ud s_{2n} \\
	&\quad \times  \int_{(\mathbb{R}^d)^n} \prod_{(j,k) \in \mathcal{D}} \mu(\ud\xi_{j,k}) \: \mathbf{E}_0 \exp \left\{ \iota \sum_{(j,k) \in \mathcal{D}} \xi_{j,k} \cdot \left(X_{v(j)}(s_j) - X_{v(k)}(s_k)\right) \right\} \\
	& =  (\lambda_1\lambda_2)^{b-nr-1} \left( \frac{2}{\nu} \right)^{2n+2}  \int_0^{\infty} \int_0^{\infty} \ud t_1 \ud t_2 \: e^{-2\lambda_1^{b} t_1/\nu-2\lambda_2^{b} t_2/\nu} \\
	& \quad \times \mathbf{E}_0 \int_{[0, t_1]_{<}^n \times [0, t_2]_{<}^n} \ud s_1 \cdots \ud s_{2n} \prod_{(j,k) \in \mathcal{D}} \gamma\left(X_{v(j)}(s_j) - X_{v(k)}(s_k)\right).
\end{aligned}
$$
This proves (ii).

\textbf{Part (iii).}   By \eqref{E:laplace F} and  monotone convergence theorem, we obtain
$$
\begin{aligned}
	& \lim_{\e, \e' \rightarrow 0} \int_0^{\infty} \int_0^{\infty} e^{-\lambda_1 t_1 - \lambda_2 t_2} F_{\e, \e'}^{\mathcal{D}}(t_1, t_2) \: \ud t_1 \ud t_2 \\
	  &= (\lambda_1\lambda_2)^{b-nr-1} \left( \frac{2}{\nu} \right)^{2n+2}  \int_0^{\infty} \int_0^{\infty} \ud t_1 \ud t_2 \: e^{-2\lambda_1^{b} t_1/\nu-2\lambda_2^{b} t_2/\nu} \int_{[0, t_1]_{<}^n \times [0, t_2]_{<}^n} \ud s_1 \cdots \ud s_{2n} \\
	&\quad \times  \int_{(\mathbb{R}^d)^n} \prod_{(j,k) \in \mathcal{D}} \mu(\ud\xi_{j,k}) \: \mathbf{E}_0 \exp \left\{ \iota \sum_{(j,k) \in \mathcal{D}} \xi_{j,k} \cdot \left(X_{v(j)}(s_j) - X_{v(k)}(s_k)\right) \right\} \\
	& =  (\lambda_1\lambda_2)^{b-nr-1} \left( \frac{2}{\nu} \right)^{2n+2}  \int_0^{\infty} \int_0^{\infty} \ud t_1 \ud t_2 \: e^{-2\lambda_1^{b} t_1/\nu-2\lambda_2^{b} t_2/\nu} \\
	& \quad \times \mathbf{E}_0 \int_{[0, t_1]_{<}^n \times [0, t_2]_{<}^n} \ud s_1 \cdots \ud s_{2n} \prod_{(j,k) \in \mathcal{D}} \gamma\left(X_{v(j)}(s_j) - X_{v(k)}(s_k)\right).
\end{aligned}
$$
This completes the proof of (iii) and the lemma. 
\end{proof}

\begin{theorem} \label{T:laplace Sn}
	 Under Assumption \ref{A:Main}, the function $g_n(\cdot, t, x)$ defined in \eqref{E:gn} is Stratonovich integrable in the sense of Definition \ref{D:multiple Stratonovich integration}. Furthermore, for any $\lambda>0$, we have
     \begin{equation}\label{e:lap-sngn}
         \int_0^{\infty} e^{-\lambda t} S_n\left(g_n(\cdot, t, x)\right) \ud t = \frac{1}{n!} \lambda^{b-nr-1} \left(\frac{2}{\nu}\right)^{n+1}  \int_0^{\infty}  e^{-2\lambda^b t/\nu} \mathbf{E}_x\left[\int_0^t \dot{W}(X(s)) d s\right]^n \ud t, \, \text{a.s.}
     \end{equation}
	\end{theorem}
    \begin{proof}
    We first explain the stochastic integral appearing on the right-hand side of \eqref{e:lap-sngn}. It is defined by
	$$
	\int_0^t \dot{W}(X(s)) \ud s := \lim _{\e \rightarrow 0^+} \int_0^t \dot{W}_{\varepsilon}(X(s)) \ud s, \quad \text { in } \mathcal{L}^2\left(\Omega, \mathcal{F}, \mathbf{P}_x \otimes \mathbb{P}\right),
	$$
	where the existence of the limit on the right hand is established in Lemma \ref{L:Wdefined}, under Assumption~\ref{A:Main}. Conditioning on the $a$-stable process $X$, it is normally distributed with mean zero and  variance
	$$
	\int_0^t \int_0^t \gamma\left(X(s)-X(r)\right) \ud s \ud r.
	$$
    So we have
    \begin{equation}\label{E:even Gaussian moment}
        \begin{split}
            &\mathbb{E} \otimes \mathbf{E}_x\left[\int_0^t \dot{W}(X(s)) \ud s\right]^n \\
    &= \left\{
	\begin{array}{ll}
		\displaystyle n!!\,\mathbf{E}_0\left[\int_0^t \int_0^t \gamma\left(X(s)-X(r)\right) \ud s \ud r\right]^{\frac{n}{2}},  & \text{when } n \text{ is even}; \\&\\
		0, & \text{when } n \text{ is odd}.
	\end{array}
	\right. 
        \end{split}
    \end{equation}
	 By Lemma \ref{L:n momentbound}, the above $n$-th moment is finite for all $n\in\mathbb N$. Consequently, the quenched moment
	$$
	\mathbf{E}_x\left[\int_0^t \dot{W}(X(s)) \ud s\right]^n
	$$
	exists almost surely and the right-hand side of \eqref{e:lap-sngn} is well-defined for any $\lambda>0$. 

    Taking $\e_1=\cdots=\e_n=\delta$ in \eqref{E:laplace S}, we have
    \begin{equation}\label{e:1}
        \int_0^{\infty} e^{-\lambda t}\: S_{n, \delta}\left(g_n(\cdot, t, x)\right) \ud t 
		 =\lambda^{b-nr-1} \left(\frac{2}{\nu}\right)^{n+1} \frac{1}{n!}   \int_0^{\infty} e^{-2\lambda^b t/\nu} \mathbf{E}_x \bigg[\int_0^t  \dot{W}_{\delta}(X(s))  \ud s \bigg]^n  \ud t.
    \end{equation}
    We now let $\delta\to0^+$ on the right-hand side. Notice that
    \begin{equation*}
       \lim _{\delta \to 0^+} \mathbf{E}_x \bigg[ \int_0^t \dot{W}_{\delta}(X(s)) \ud s \bigg]^n = \mathbf{E}_x \bigg[ \int_0^t \dot{W}(X(s)) \ud s \bigg]^n, \quad \text { in } \mathcal{L}^2\left(\Omega, \mathcal{F}, \mathbb{P}\right).
    \end{equation*}
    In addition, applying Jensen inequality twice, we get
    \begin{align*}
        &\E\bigg\{ \mathbf{E}_x \bigg[ \int_0^t \dot{W}_{\delta}(X(s)) \ud s \bigg]^n \bigg\}^2 
        \le  \mathbf{E}_0 \otimes\E \bigg[ \int_0^t \dot{W}_{\delta}(X(s)) \ud s \bigg]^{2n}  \\
        &\le \int_{\R^d} p_\delta(y)  \mathbf{E}_0 \otimes\E \bigg[ \int_0^t \dot{W}(y+X(s)) \ud s \bigg]^{2n}  \ud y
        =  \mathbf{E}_0 \otimes\E \bigg[ \int_0^t \dot{W}(X(s)) \ud s \bigg]^{2n}\\
        &=   (2n)!! \mathbf{E}_0\left[\int_0^t \int_0^t \gamma\left(X(s)-X(r)\right) \ud s \ud r\right]^n.
    \end{align*}
    By Lemma \ref{L:n momentbound}, the right-hand side has at most a polynomial increasing rate in $t$. Apply the dominated convergence theorem to see that 
    \begin{equation}\label{e:2}
        \lim_{\delta\to0^+} \int_0^{\infty} e^{-2\lambda^b t/\nu} \mathbf{E}_x \bigg[\int_0^t  \dot{W}_{\delta}(X(s))  \ud s \bigg]^n  \ud t =  \int_0^{\infty} e^{-2\lambda^b t/\nu} \mathbf{E}_x \bigg[\int_0^t  \dot{W}(X(s))  \ud s \bigg]^n  \ud t,
    \end{equation}
    in $\mathcal{L}^2\left(\Omega, \mathcal{F}, \mathbb{P}\right)$.
    
    The Stratonovich integrability of $g_n(\cdot,t,x)$ shall be established in Theorem \ref{th:strat-intablity} to make sense of left-hand side of \eqref{e:lap-sngn}. By the stationarity in $x$, it suffices to prove the following:
\begin{align*}
    \lim_{\delta\to 0^+} \int_0^{\infty} e^{-\lambda t}\: S_{n, \delta}\left(g_n(\cdot, t, 0)\right) \ud t
    = \int_0^{\infty} e^{-\lambda t}\: S_{n}\left(g_n(\cdot, t, 0)\right) \ud t, \text{ in } \mathcal{L}^2\left(\Omega, \mathcal{F}, \mathbb{P}\right),
\end{align*}
which will be established in Theorem \ref{th:strat-intablity} (ii).  This together with \eqref{e:1} and \eqref{e:2} implies the desird result.
\end{proof}

    The following corollary will be used in the proof of Theorem \ref{th:exist}.

\begin{corollary}\label{coro:3.5}
    Let Assumption \ref{A:Main} hold. Then, we have for any $\lambda_1, \lambda_2>0$, 
    \begin{equation}
        \label{eq:ESS 2}
    	\begin{aligned}
		& \int_0^{\infty} \int_0^{\infty} \ud t_1 \ud t_2 \: e^{-\lambda_1 t_1-\lambda_2 t_2} \: \mathbb{E} [S_n\left(g_n\left(\cdot, t_1, 0\right)\right) S_n\left(g_n\left(\cdot, t_2, 0\right)\right)]  \\
		&=\frac{1}{(n!)^2} (\lambda_1\lambda_2)^{b-nr-1} \left(\frac{2}{\nu}\right)^{2n+2}   \int_0^{\infty} \int_0^{\infty} \ud t_1 \ud t_2 e^{-2\lambda_1^{b} t_1/\nu-2\lambda_2^{b} t_2/\nu}  \\
		& \quad \times \E\otimes \mathbf{E}_0 \left[\int_0^{t_1} \dot{W}\left(X_1(s)\right) \ud s\right]^n \left[\int_0^{t_2} \dot{W}\left(X_2(s)\right) \ud s\right]^n<\infty. 
	\end{aligned}
    \end{equation}
Moreover, for any $\e=(\e_1,\dots,\e_n)$, $\e'=(\e_{n+1},\dots,\e_{2n})$, we have
    \begin{align*}
        & \int_0^{\infty} \int_0^{\infty} \ud t_1 \ud t_2 \: e^{-\lambda_1 t_1-\lambda_2 t_2} \: \mathbb{E} [S_{n,\e}\left(g_n\left(\cdot, t_1, 0\right)\right) S_{n,\e'}\left(g_n\left(\cdot, t_2, 0\right)\right)] \\
		&\le \int_0^{\infty} \int_0^{\infty} \ud t_1 \ud t_2 \: e^{-\lambda_1 t_1-\lambda_2 t_2} \: \mathbb{E} [S_n\left(g_n\left(\cdot, t_1, 0\right)\right) S_n\left(g_n\left(\cdot, t_2, 0\right)\right)]<\infty.
    \end{align*}  
\end{corollary}

\begin{proof}
The equality in~\eqref{eq:ESS 2} follows from  Theorem~\ref{T:laplace Sn} and  the fniteness is due to Lemma~\ref{L:n momentbound} and Lemma~\ref{L:Wdefined}.    By \eqref{e:cov-Sne}, \eqref{eq:Fvarepsilon} and Lemma \ref{le:3.1} (ii), we have
    \begin{equation}\label{eq:ESS 1}
     \begin{aligned}
        & \int_0^{\infty} \int_0^{\infty} \ud t_1 \ud t_2 e^{-\lambda_1 t_1-\lambda_2 t_2} \mathbb{E} [S_{n,\e}\left(g_n\left(\cdot, t_1, 0\right)\right) S_{n,\e'}\left(g_n\left(\cdot, t_2, 0\right)\right)] \\
		&= \sum_{\mathcal{D} \in \Pi_{2n}}\int_0^{\infty} \int_0^{\infty} \ud t_1 \ud t_2 e^{-\lambda_1 t_1-\lambda_2 t_2} F_{\varepsilon, \varepsilon^{\prime}}^{\mathcal{D}}\left(t_1, t_2\right)\\
        &\leq(\lambda_1\lambda_2)^{b-nr-1} \left( \frac{2}{\nu} \right)^{2n+2}  \sum_{\mathcal{D} \in \Pi_{2n}}  \int_0^{\infty} \int_0^{\infty} \ud t_1 \ud t_2  e^{-2\lambda_1^{b} t_1/\nu-2\lambda_2^{b} t_2/\nu}\\
		&\quad \times \mathbf{E}_0 \int_{\left[0, t_1\right]_{<}^n \times\left[0, t_2\right]_{<}^n} \ud s_1 \cdots \ud s_{2 n} \prod_{(j, k) \in \mathcal{D}} \gamma\left(X_{v(j)}\left(s_j\right)-X_{v(k)}\left(s_k\right)\right). 
    \end{aligned}
    \end{equation}    
Thus, to complete the proof it remains to verify that the right‑hand side of \eqref{eq:ESS 1} coincides with that of \eqref{eq:ESS 2}.

We first consider the inner sum over $\mathcal{D}$ and the expectation integral 
with respect to $s_1,\dots,s_{2n}$ in \eqref{eq:ESS 1}.
Since the sum runs over all pair partitions of $\{1,\dots,2n\}$, any permutation of the first $n$ 
integration variables $(s_1,\dots,s_n)$ can be absorbed by relabeling the indices in the pairings, 
leaving the total sum unchanged; the same holds for the last $n$ variables.  
Hence, under $\sum_{\mathcal{D}}$, the integrand is symmetric with respect to these two groups 
of variables separately.  Exploiting this symmetry yields
\begin{align}
        & \sum_{\mathcal{D} \in \Pi_{2n}} 
		\mathbf{E}_0 \int_{\left[0, t_1\right]_{<}^n \times\left[0, t_2\right]_{<}^n} \ud s_1 \cdots \ud s_{2 n} \prod_{(j, k) \in \mathcal{D}} \gamma\left(X_{v(j)}\left(s_j\right)-X_{v(k)}\left(s_k\right)\right)\nonumber\\
        =&\frac{1}{(n!)^2}\sum_{\mathcal{D} \in \Pi_{2n}} \mathbf{E}_0 \int_{\left[0, t_1\right]^n \times\left[0, t_2\right]^n} \ud s_1 \cdots \ud s_{2 n} \prod_{(j, k) \in \mathcal{D}} \gamma\left(X_{v(j)}\left(s_j\right)-X_{v(k)}\left(s_k\right)\right)\nonumber\\
        =&\frac{1}{(n!)^2}\sum_{\mathcal{D} \in \Pi_{2n}} \prod_{(j, k) \in \mathcal{D}} \int_0^{t_{v(j)}} \int_0^{t_{v(k)}} \gamma\left(X_{v(j)}(s)-X_{v(k)}(r)\right) d s d r. \nonumber
    \end{align}
    where the second equality holds because the integrand is already a product over disjoint pairs.

Now, we consider the inner expectation product in \eqref{eq:ESS 2}:
$$
 \E\otimes \mathbf{E}_0 \left[\int_0^{t_1} \dot{W}\left(X_1(s)\right) \ud s\right]^n \left[\int_0^{t_2} \dot{W}\left(X_2(s)\right) \ud s\right]^n.
$$
Applying Wick's formula to this product of $2n$ Gaussian stochastic integrals yields
\begin{align}
 &\E\otimes \mathbf{E}_0 \left[\int_0^{t_1} \dot{W}\left(X_1(s)\right) \ud s\right]^n \left[\int_0^{t_2} \dot{W}\left(X_2(s)\right) \ud s\right]^n \nonumber\\
= &\sum_{\mathcal{D}\in\Pi_{2n}}\prod_{(j,k)\in\mathcal{D}}
\int_0^{t_{v(j)}}\!\int_0^{t_{v(k)}}\gamma\bigl(X_{v(j)}(s)-X_{v(k)}(r)\bigr)\,\ud s\,\ud r . \nonumber
\end{align}
This shows that the right‑hand side of \eqref{eq:ESS 1} coincides with the right‑hand side of \eqref{eq:ESS 2}, which completes the proof.
\end{proof}

\subsection{Stratonovich integrability}

\begin{lemma} \label{continuity theorem}
     Under Assumptions \ref{A:Main} and \ref{A:G}, the limit $
\lim _{\e, \e^{\prime} \rightarrow 0} F_{\e, \e^{\prime}}^{\mathcal{D}}(t_1, t_2)$, where $F_{\e, \e^{\prime}}^{\mathcal{D}}(t_1, t_2)$ is defined in~\eqref{eq:Fvarepsilon},  
exists for every $n \geq 1, t_1, t_2>0$ and every pair partition $\mathcal{D} \in \Pi_{2n}$. Moreover, the resulting limiting function is continuous in $t_1, t_2$.
\end{lemma}


\begin{proof}
It is obvious that $F_{\e, \e^{\prime}}^{\mathcal{D}}(t_1, t_2)$ is non-negative, non-decreasing and continuous on $\R^+\times\R^+$. By Lemma \ref{le:3.1}(iii), the limit
\begin{equation*}
    \lim _{\e, \e' \rightarrow 0} \int_0^{\infty} \int_0^{\infty} e^{-\lambda_1 t_1-\lambda_2 t_2} F_{\e, \e^\prime}^{\mathcal{D}}\left(t_1, t_2\right) \ud t_1 \ud t_2
\end{equation*}
 exists  and is a positive finite number for $\lambda_1,\lambda_2>0$ due to Corollary~\ref{coro:3.5}. By continuity theorem for Laplace transform (see, e.g., Theorem 2a, Section 1, Chapter XIII in \cite{feller1971}), the measures $e^{-k(t_1+t_2)} F_{\e,\e^\prime}^{\mathcal{D}} (t_1, t_2)\ud t_1 \ud t_2$  converges weakly to  $e^{-k(t_1+t_2)}F^{\mathcal{D}} (t_1, t_2)\ud t_1\ud t_2$  on $\R^+\times\R^+$ with arbitrarily fixed $k>0$. 
 Therefore, by Lemma \ref{le:continu-density}, to prove this lemma, it is sufficient to show that
\begin{equation}\label{e:left-con}
    \lim_{\delta_1,\delta_2\to0} \sup_{\e,\e'} \left| F_{\e, \e^{\prime}}^{\mathcal{D}}(t_1+\delta_1, t_2+\delta_2)-F_{\e, \e^{\prime}}^{\mathcal{D}}(t_1, t_2)\right|=0,
\end{equation}
for every $t_1,t_2\in\R^+$.

Denote
\begin{align*}
    & \mathcal{J}_{\e, \e^{\prime}}\left(\left[0, t_1\right]_{<}^n \times \left[0, t_2\right]_{<}^{n} \right):= F_{\e, \e^{\prime}}^{\mathcal{D}}(t_1, t_2) \\
    &= \int_{[0, t_1]_{<}^n \times [0, t_2]_{<}^n} \ud s_1 \cdots \ud s_{2n} \int_{(\mathbb{R}^d)^{2n}} \ud x_1 \cdots \ud x_{2n} \bigg(\prod_{(j,k) \in \mathcal{D}} \gamma_{\e_j + \e_k}(x_j - x_k) \bigg) \\
    &\quad \times \bigg( G ( s_ 1 , x_1 ) \prod_{l=2}^{n} G(s_l-s_{l-1} ,x_l-x _{l-1}) \bigg)  \bigg(G\left(s_{n+1}, x_{n+1}\right) \prod_{k=n+2}^{2 n-1} G\left(s_l-s_{l-1}, x_k-x_{k-1}\right)\bigg).
\end{align*}
To prove \eqref{e:left-con}, all we need is 
\begin{equation*}
   \begin{cases}
       &\displaystyle \lim_{\delta_1,\delta_2\downarrow0} \sup_{\e,\e'} \left| \mathcal{J}_{\e, \e^{\prime}} \bigg( \Big\{[0, t_1]_{<}^n \times [0, t_2]_{<}^{n} \Big\} \setminus \Big\{ [0, t_1-\delta_1]_{<}^n \times [0, t_2-\delta_2]_{<}^{n} \Big\} \bigg) \right|=0,\\
       & \displaystyle \lim_{\delta_1,\delta_2\downarrow0} \sup_{\e,\e'} \left| \mathcal{J}_{\e, \e^{\prime}} \bigg( \Big\{[0, t_1+\delta_1]_{<}^n \times [0, t_2 +\delta_2]_{<}^{n} \Big\} \setminus \Big\{ [0, t_1]_{<}^n \times [0, t_2]_{<}^{n} \Big\} \bigg) \right|=0.
   \end{cases}
\end{equation*}
 Since the arguments are similar, we only give the proof for the first limit.

Applying the general set identity
\begin{equation*}
    (A_1\times A_2)\setminus(B_1\times B_2) = (A_1\times (A_2\setminus B_2)\cup ((A_1\setminus B_1)\times A_2),
\end{equation*}
 with $B_1\subseteq A_1$ and $B_2\subseteq A_2$, we can see that
 \begin{align*}
     & \Big\{[0, t_1]_{<}^n \times [0, t_2]_{<}^{n} \Big\} \setminus \Big\{ [0, t_1-\delta_1]_{<}^n \times [0, t_2-\delta_2]_{<}^{n} \Big\}\\
     &= \Big\{[0, t_1]_{<}^n \times \left([0, t_2]_{<}^{n}\setminus [0, t_2-\delta_2]_{<}^{n}\right) \Big\} \cup \Big\{ ([0, t_1]_{<}^n\setminus  [0, t_1-\delta_1]_{<}^n) \times [0, t_2]_{<}^{n} \Big\}\\
     &\subseteq \Big\{[0, t_1]_{<}^n \times \left([0, t_2]_{<}^{n-1}\times [t_2-\delta_2, t_2] \right) \Big\} \cup \Big\{ ([0, t_1]_{<}^{n-1} \times  [t_1-\delta_1, t_1]) \times [0, t_2]_{<}^{n} \Big\}.
 \end{align*}
 Here, we can use the convention that $G(t,x)=0$ for $t<0$.

Therefore, the problem is reduced to
\begin{equation}\label{e:Je-Je}
    \lim _{\delta \rightarrow 0^{+}} \sup _{\e, \e^{\prime}} \mathcal{J}_{\e, \e^{\prime}}\left(\left[0, t_1\right]_{<}^n \times\left[0, t_2\right]_{<}^{n-1} \times\left[t_2-\delta, t_2\right]\right)=0,
\end{equation}
and
\begin{equation*}
    \lim _{\delta \rightarrow 0^{+}} \sup _{\e, \e^{\prime}} \mathcal{J}_{\e, \e^{\prime}}\left(\left[0, t_1\right]_{<}^{n-1} \times\left[t_1-\delta, t_1\right] \times\left[0, t_2\right]_{<}^n\right)=0.
\end{equation*}
Due to similarity, we only prove \eqref{e:Je-Je}. By Fubini's theorem, we have
\begin{align*}
& \mathcal{J}_{\e, \e^{\prime}}\left(\left[0, t_1\right]_{<}^n \times\left[0, t_2\right]_{<}^{n-1} \times\left[t_2-\delta, t_2\right]\right) \\
& =\int_{\left(\mathbb{R}^d\right)^{2 n-1}} \ud x_1 \cdots \ud x_{2 n-1} \bigg(\prod_{(j, k) \in \mathcal{D}^{\prime}} \gamma_{\e_j+\e_k}\left(x_j-x_k\right)\bigg) 
\int_{\left[0, t_1\right]_{<}^n \times\left[0, t_2\right]_{<}^{n-1}} \ud s_1 \cdots \ud s_{2 n-1} \\
& \quad \times \bigg( G ( s _ { 1 } , x _ { 1 } ) \prod _ { l = 2 } ^ { n } G ( s _ { l } - s _ { l - 1 } , x _ { l } - x _ { l - 1 } ) \bigg)  \bigg(G\left(s_{n+1}, x_{n+1}\right) \prod_{k=n+2}^{2 n-1} G\left(s_l-s_{l-1}, x_k-x_{k-1}\right)\bigg) \\
& \quad \times \int_{\mathbb{R}^d} \gamma_{\e_{j_0}+\e_{2 n}}\left(x_{2 n}-x_{j_0}\right) \ud x_{2 n}  \int_{t_2-\delta}^{t_2} G\left(s_{2 n}-s_{2 n-1}, x_{2 n}-x_{2 n-1}\right) \ud s_{2 n},
\end{align*}
where $1 \leq j_0 \leq 2 n-1$ satisfies $\left(j_0, 2 n\right) \in \mathcal{D}$ and where $\mathcal{D}^{\prime} \in \Pi_{2(n-1)}$ is given by $\mathcal{D}^{\prime}=\mathcal{D} \backslash \left(j_0, 2 n\right)$.
By the Fourier transform and Fubini's theorem, we obtain
\begin{align*}
    &\int_{\mathbb{R}^d} \gamma_{\e_{j_0}+\e_{2 n}}(x_{2 n}-x_{j_0})  \ud  x_{2 n} \int_{t_2-\delta}^{t_2} G\left(s_{2 n}-s_{2 n-1}, x_{2 n}-x_{2 n-1}\right) \ud s_{2 n} \\
 &= \int_{\mathbb{R}^d} \gamma_{\e_{j_0}+\e_{2 n}}(x_{2 n}-x_{j_0}) \ud x_{2 n}  \int_{0 \vee\left(t_2-s_{2 n-1}-\delta\right)}^{t_2-s_{2 n-1}} G\left(s, x_{2 n}-x_{2 n-1}\right) \ud s \\
 &= \int_{\mathbb{R}^d} \mu(\ud \xi) \exp \left\{-\frac{\e_{j_0}+\e_{2 n}}{2}|\xi|^2\right\} \int_{0 \vee\left(t_2-s_{2 n-1}-\delta\right)}^{t_2-s_{2 n-1}} \ud s \\
&\quad \times \int_{\mathbb{R}^d} \exp \left\{\iota \xi \cdot\left(x_{2 n}-x_{j_0}\right)\right\} G\left(s, x_{2 n}-x_{2 n-1}\right) \ud x_{2 n}.
\end{align*}
The right hand side is equal to
\begin{equation}
\begin{aligned}\label{e:mlf-con-0}
    &\int_{\mathbb{R}^d}  \mu(\ud \xi)\exp \left\{-\frac{\e_{j_0}+\e_{2 n}}{2}|\xi|^2+\iota \xi \cdot\left(x_{2 n-1}-x_{j_0}\right)\right\}    \int_{0 \vee (t_2-s_{2n-1}-\delta)}^{t_{2-s_{2 n-1}}} \ud s \\
& \qquad \times \int_{\mathbb{R}^d} \exp \left\{\iota \xi \cdot\left(x_{2 n}-x_{2 n-1}\right)\right\} G\left(s, x_{2 n}-x_{2 n-1}\right) \ud x_{2 n} \\
 &= \int_{\mathbb{R}^d} \mu(\ud \xi)  \exp\left\{-\frac{\e_{j_0}+\e_{2 n}}{2}|\xi|^2+\iota \xi \cdot\left(x_{2 n-1}-x_{j_0}\right)\right\}   \int_{0 \vee (t_2-s_{2n-1}-\delta)}^{t_{2-s_{2 n-1}}} s^{b+r-1} E_{b,b+r}(-\tfrac12 \nu s^b |\xi|^a)  \ud s \\
 & \leq\int_{\mathbb{R}^d} \mu(\ud \xi)\left|\int_{0 \vee\left(t_2-s_{2 n-1}-\delta\right)}^{t_2-s_{2 n-1}} s^{b+r-1} E_{b,b+r}(-\tfrac12 \nu s^b |\xi|^a) \ud s\right|,
\end{aligned}
\end{equation}
where the equality is due to \eqref{E:FZ}.
Applying the identity
\begin{equation*}
   \frac{\ud}{\ud s} s^{b+r} E_{b,b+r+1}(-\tfrac12 \nu s^b |\xi|^a)
    = s^{b+r-1} E_{b,b+r}(-\tfrac12 \nu s^b |\xi|^a),
\end{equation*}
we obtain that for $b\in(0,2)$,
\begin{align*}
&\left|\int_{0 \vee\left(t_2-s_{2 n-1}-\delta\right)}^{t_2-s_{2 n-1}} s^{b+r-1} E_{b,b+r}(-\tfrac12 \nu s^b |\xi|^a) \ud s\right| \notag\\
&=\left|\int_{0 \vee\left(t_2-s_{2 n-1}-\delta\right)}^{t_2-s_{2 n-1}} \ud s^{b+r} E_{b,b+r+1}(-\tfrac12 \nu s^b |\xi|^a)  \right|  \notag\\
&\le \bigg|\frac{\tilde{c} (t_2-s_{2 n-1})^{b+r} }{1+\tfrac12 \nu (t_2-s_{2 n-1})^b  |\eta|^a} \bigg| +
\bigg|\frac{\tilde{c} |t_2-s_{2 n-1}-\delta|^{b+r} }{1+\tfrac12 \nu |t_2-s_{2 n-1}-\delta|^b  |\eta|^a} \bigg|,
\end{align*}
where in the inequality we have used the fact that there is $\tilde{c}$ depending on $\varrho_1$ and $\varrho_2$ such that
\begin{equation*} 
    \left|E_{\varrho_1, \varrho_2}(z)\right| \leq \frac{\tilde{c}}{1+|z|},
\end{equation*}
 for all $\varrho_1 \in(0,2)$, $\varrho_2\in\R$ and $z<0$ (see \cite[Theorem 1.6]{Podlubny1999Fractional}). 
By Assumption \ref{A:Main},  
\begin{equation*}
    \int_{\mathbb{R}^d} \bigg|\frac{\tilde{c} (t_2-s_{2 n-1})^{b+r} }{1+\tfrac12 \nu (t_2-s_{2 n-1})^b  |\eta|^a}\bigg|  +
\bigg|\frac{\tilde{c} |t_2-s_{2 n-1}-\delta|^{b+r} }{1+\tfrac12 \nu |t_2-s_{2 n-1}-\delta|^b  |\eta|^a} \bigg|\mu(\ud\eta)<\infty.
\end{equation*}
Thus, by \eqref{e:mlf-con-0} and  dominated convergence theorem, for $b\in(0,2)$
\begin{align*}
    &\lim_{\delta\to0^+}\int_{\mathbb{R}^d} \ud x_{2 n} \gamma_{\e_{j_0}+\e_{2 n}}\left(x_{2 n}-x_{j_0}\right) \int_{t_2-\delta}^{t_2} G\left(s_{2 n}-s_{2 n-1}, x_{2 n}-x_{2 n-1}\right) \ud s_{2 n}\\
    &\le  \lim_{\delta\to0^+} \int_{\mathbb{R}^d} \left|\int_{0 \vee\left(t_2-s_{2 n-1}-\delta\right)}^{t_2-s_{2 n-1}} s^{b+r-1} E_{b,b+r}(-\tfrac12 \nu s^b |\xi|^a) \ud s\right| \mu(\ud \xi)\\
    &=\int_{\mathbb{R}^d} \lim_{\delta\to0^+} \left|\int_{0 \vee\left(t_2-s_{2 n-1}-\delta\right)}^{t_2-s_{2 n-1}} s^{b+r-1} E_{b,b+r}(-\tfrac12 \nu s^b |\xi|^a) \ud s\right| \mu(\ud \xi)=0, 
\end{align*}
Hence, there is a function $h(\delta)>0$ independent of $\left(\e, \e^{\prime}\right)$ such that
$$
\int_{\mathbb{R}^d} \ud x_{2 n} \gamma_{\e_{j_0}+\e_{2 n}}\left(x_{2 n}-x_{j_0}\right) \int_{t_2-\delta}^{t_2} G\left(s_{2 n}-s_{2 n-1}, x_{2 n}-x_{2 n-1}\right) \ud s_{2 n} \leq h(\delta),
$$
and  $h(\delta) \rightarrow 0$ as $\delta \rightarrow 0^{+}$  for all $b\in(0,2]$, noting that  case (ii) of Assumption~\ref{A:G} has been considered in Lemma \cite[Lemma 3.6]{ch2024}. Consequently,
\begin{equation*}
     \mathcal{J}_{\e, \e^{\prime}}\left(\left[0, t_1\right]_{<}^n \times\left[0, t_2\right]_{<}^{n-1} \times\left[t_2-\delta, t_2\right]\right) \le h(\delta) \mathcal{K}_{\e, \e^\prime}(t_1,t_2),
\end{equation*}
where
\begin{align*}
    &\mathcal{K}_{\e, \e^\prime}(t_1,t_2): =\int_{\left(\mathbb{R}^d\right)^{2 n-1}} \ud x_1 \cdots \ud x_{2 n-1} \bigg(\prod_{(j, k) \in \mathcal{D}^{\prime}} \gamma_{\e_j+\e_k}\left(x_j-x_k\right)\bigg) 
\int_{\left[0, t_1\right]_{<}^n \times\left[0, t_2\right]_{<}^{n-1}} \ud s_1 \cdots \ud s_{2 n-1} \\
& \quad \times \bigg( G ( s _ { 1 } , x _ { 1 } ) \prod _ { l = 2 } ^ { n } G ( s _ { l } - s _ { l - 1 } , x _ { l } - x _ { l - 1 } ) \bigg)  \bigg(G\left(s_{n+1}, x_{n+1}\right) \prod_{k=n+2}^{2 n-1} G\left(s_l-s_{l-1}, x_k-x_{k-1}\right)\bigg).
\end{align*}
Then, to prove \eqref{e:Je-Je}, it suffices to show that
\begin{equation*}
   \sup_{\e, \e^\prime} \mathcal{K}_{\e, \e^\prime}(t_1,t_2) < \infty.
\end{equation*}
By a similar computation leading to  Lemma \ref{le:3.1} (ii), we can obtain
\begin{align*}
    &\int_0^\infty \int_0^\infty e^{-r-r'} \mathcal{K}_{\e, \e^\prime}(r,r') \ud r\ud r'\\
    &\le \left( \frac{2}{\nu} \right)^{2n+2}  \int_0^{\infty} \int_0^{\infty} \ud r \ud r' \: e^{-2 r/\nu-2 r'/\nu} \\
		&\quad \times \mathbf{E}_0 \int_{\left[0, r\right]_{<}^n \times\left[0, r'\right]_{<}^{n-1}} \ud s_1 \cdots \ud s_{2 n} \prod_{(j, k) \in \mathcal{D}'} \gamma\left(X_{v(j)}\left(s_j\right)-X_{v(k)}\left(s_k\right)\right),
\end{align*}
for any $\e, \e^\prime$, the right-hand side of which is finite by  Lemma \ref{L:n momentbound}.
By non-negativity and  monotonicity of $\mathcal{K}_{\e, \e^\prime}(t_1,t_2)$ in $t_1$ and $t_2$, we have
\begin{align*}
    \sup_{\e, \e^\prime} \mathcal{K}_{\e, \e^\prime}(t_1,t_2)
    &=e^{t_1+t_2} \sup_{\e, \e^\prime} \mathcal{K}_{\e, \e^\prime}(t_1,t_2)  \int_{t_1}^\infty \int_{t_2}^\infty e^{-r-r'} \ud r\ud r' \\
    &\le e^{t_1+t_2} \: \sup_{\e, \e^\prime} \int_{t_1}^\infty \int_{t_2}^\infty e^{-r-r'} \mathcal{K}_{\e, \e^\prime}(r,r') \ud r\ud r' \\
    &\le e^{t_1+t_2} \: \sup_{\e, \e^\prime} \int_0^\infty \int_0^\infty e^{-t_1-t_2} \mathcal{K}_{\e, \e^\prime}(r,r') \ud r\ud r' <\infty.
\end{align*}
This completes the proof.
\end{proof}

In the following result, we establish the Stratonovich integrability of $g_n(\cdot,t,x)$ and $\cL^2$-convergence of the Laplace transform
\begin{equation*}
    \int_0^{\infty} e^{-\lambda t} S_{n, \e}\left(g_n(\cdot, t, x)\right) \ud t
\end{equation*}
as $\e\to0$.

\begin{theorem}[Stratonovich integrability]\label{th:strat-intablity}
 Let Assumptions \ref{A:Main} and \ref{A:G} hold.
\begin{enumerate}[(i)]
    \item The $\mathcal{L}^2$-limit
$$
\lim _{\e_1, \cdots, \e_n \rightarrow 0^{+}} S_{n, \varepsilon}\left(g_n(\cdot, t, x)\right)
$$
exists for any $n \geq 1$ and $(t, x) \in \mathbb{R}^{+} \times \mathbb{R}^d$. Consequently, $g_n(\cdot,t,x)$ is integrable in the sense of Definition \ref{D:multiple Stratonovich integration}, and the above limit is $S_n(g_n(\cdot,t,x))$.

\item For any $\lambda>0$, we have
\begin{equation*}
    \lim _{\e \rightarrow 0} \int_0^{\infty} e^{-\lambda t} S_{n, \e}\left(g_n(\cdot, t, x)\right) \ud t = \int_0^{\infty} e^{-\lambda t} S_n\left(g_n(\cdot, t, x)\right) \ud t
\end{equation*}
in $\mathcal{L}^2(\Omega, \mathcal{F}, \mathbb{P})$.
\end{enumerate}
\end{theorem}
\begin{proof}
  To establish the $\mathcal{L}^2$-convergence of the  integral $S_{n, \e}\left(g_n(\cdot, t, x)\right)$, it suffices by the Cauchy criterion to compute their covariance and show that the limit exists as $\e$ tend to zero.  Applying Wick's formula~\eqref{e:wick-formula} reduces this covariance to a sum over pair partitions of the quantities $F_{\varepsilon,\varepsilon'}^{\mathcal{D}}(t,t)$ defined in~\eqref{eq:Fvarepsilon}, and Lemma~\ref{continuity theorem} guarantees the convergence of each term. This yields part (i).  
  
   For part (ii), the goal is to show
	\begin{equation}
		\lim_{\varepsilon \rightarrow 0} \mathbb{E}\Bigl|\int_0^{\infty} e^{-\lambda t} S_{n,\varepsilon}\bigl(g_n(\cdot, t, 0)\bigr) \ud t
		- \int_0^{\infty} e^{-\lambda t} S_n\bigl(g_n(\cdot, t, 0)\bigr) \ud t\Bigr|^2 = 0. \nonumber
	\end{equation} 
    The square expands into the sum of the pure $\varepsilon$-term, the pure limit term, and a cross term.  The pure $\varepsilon$-term converges to the pure limit term by Wick's formula together with Lemma~\ref{le:3.1}(iii) and dominated convergence. For the cross term, part (i) gives the representation 
    $$\mathbb{E}\bigl[S_{n,\varepsilon}(g_n(\cdot,t_1,0))\,S_n(g_n(\cdot,t_2,0))\bigr]
   = \lim_{\varepsilon'\to0}\mathbb{E}\bigl[S_{n,\varepsilon}(g_n(\cdot,t_1,0))\,
     S_{n,\varepsilon'}(g_n(\cdot,t_2,0))\bigr],$$ and the convergence to the same limit follows from Fatou's lemma and the Cauchy--Schwarz inequality. The detailed proof is similar to that of \cite[Theorem~3.8]{ch2024}, and is therefore omitted.
\end{proof}

The following lemma is the Fubini’s theorem for the multiple Stratonovich integral with the integrand $g_n(\cdot,t,x)$.

\begin{lemma}[Fubini’s theorem]\label{le:fubini-th}
 Under Assumptions \ref{A:Main} and \ref{A:G}, we have
\begin{equation}
\lim _{\varepsilon_2, \cdots, \varepsilon_n \rightarrow 0^{+}} S_{n, \varepsilon}\left(g_n(\cdot, t, x)\right)=\int_{\mathbb{R}^d}\left(\int_0^t G(t-s, y-x) S_{n-1}\left(g_{n-1}(\cdot, s, y)\right) \ud s\right) \dot{W}_{\varepsilon_1}(y) \ud y \label{E:fubini-i}
\end{equation}
and
\begin{equation}
    \label{E:fubini-ii}
\lim _{\varepsilon_1 \rightarrow 0^{+}} \int_{\mathbb{R}^d}\left(\int_0^t G(t-s, y-x) S_{n-1}\left(g_{n-1}(\cdot, s, y)\right) \ud s\right) \dot{W}_{\varepsilon_1}(y) \ud y=S_n\left(g_n(\cdot, t, x)\right)
\end{equation}
where the limits are taken in $\mathcal{L}^p(\Omega, \mathcal{F}, \mathbb{P})$ for any $p \geq 1$. 
\end{lemma}

\begin{proof}
	Let $\tilde{\varepsilon}:= (\varepsilon_2,\dots,\varepsilon_n)$.
	We first prove that \eqref{E:fubini-i} holds in $\mathcal{L}^1$.
	By the definition of $S_{n,\varepsilon}$ and Fubini's theorem,
	\begin{align*}
		S_{n,\varepsilon}(g_n(\cdot,t,x))
		&= \int_{\mathbb{R}^d} \Bigl(\int_0^t G(t-s,y-x)\,S_{n-1,\tilde{\varepsilon}}(g_{n-1}(\cdot,s,y))\,\ud s\Bigr)\,
		\dot{W}_{\varepsilon_1}(y)\,\ud y .
	\end{align*}
	By Theorem~\ref{th:strat-intablity} (i), $S_{n-1,\tilde{\varepsilon}}(g_{n-1}(\cdot,t,x))$
	converges to $S_{n-1}(g_{n-1}(\cdot,t,x))$ in $\mathcal{L}^2$ as $\tilde{\varepsilon}\to 0$.
	Hence
	\begin{align}
		&S_{n,\varepsilon}(g_n(\cdot,t,x)) - \int_{\mathbb{R}^d}\Bigl(\int_0^t G(t-s,y-x)\,S_{n-1}(g_{n-1}(\cdot,s,y))\,\ud s\Bigr)\,
		\dot{W}_{\varepsilon_1}(y)\,\ud y \nonumber \\
		&= \int_0^t \ud s \int_{\mathbb{R}^d}
		\bigl[S_{n-1,\tilde{\varepsilon}}(g_{n-1}(\cdot,s,y)) - S_{n-1}(g_{n-1}(\cdot,s,y))\bigr]\,
		\dot{W}_{\varepsilon_1}(y)\,G(t-s,y-x)\,\ud y . \nonumber
	\end{align}
	Taking the $\mathcal{L}^1$norm and applying the Cauchy--Schwarz inequality, we get
	\begin{align*}		
    &\mathbb{E}\Bigl|S_{n,\varepsilon}(g_n(\cdot,t,x)) - \int_{\mathbb{R}^d}\Bigl(\int_0^t G(t-s,y-x)S_{n-1}(g_{n-1}(\cdot,s,y))\ud s\Bigr)\dot{W}_{\varepsilon_1}(y)\ud y\Bigr| \\
		&\le \Bigl(\mathbb{E}\int_0^t \ud s \int_{\mathbb{R}^d}
		\bigl[S_{n-1,\tilde{\varepsilon}}(g_{n-1}(\cdot,s,y)) - S_{n-1}(g_{n-1}(\cdot,s,y))\bigr]^2
		G(t-s,y-x)\,\ud y \Bigr)^{1/2} \\
    		&\qquad \times \Bigl(\mathbb{E}\int_0^t \ud s \int_{\mathbb{R}^d} |\dot{W}_{\varepsilon_1}(y)|^2 G(t-s,y-x)\,\ud y\Bigr)^{1/2}.
	\end{align*}
	By the spatial stationarity of $g_{n-1}$,
	\[
	\mathbb{E}\bigl[S_{n-1,\tilde{\varepsilon}}(g_{n-1}(\cdot,s,y)) - S_{n-1}(g_{n-1}(\cdot,s,y))\bigr]^2
	= \mathbb{E}\bigl[S_{n-1,\tilde{\varepsilon}}(g_{n-1}(\cdot,s,0)) - S_{n-1}(g_{n-1}(\cdot,s,0))\bigr]^2,
	\]
	and $\mathbb{E}|\dot{W}_{\varepsilon_1}(y)|^2 = \gamma_{2\varepsilon_1}(0)$ is independent of $y$.
	Moreover, by \eqref{E:FZ},
	\[
	\int_{\mathbb{R}^d} G(t-s,y-x)\,\ud y = \mathcal{F}G(t-s,\cdot)(0) = \frac{(t-s)^{b+r-1}}{\Gamma(b+r)} .
	\]
	Hence there exists a constant $C$ such that
	\begin{align*}
		&\mathbb{E}\Bigl|S_{n,\varepsilon}(g_n(\cdot,t,x)) - \int_{\mathbb{R}^d}\Bigl(\int_0^t G(t-s,y-x)S_{n-1}(g_{n-1}(\cdot,s,y))\ud s\Bigr)\dot{W}_{\varepsilon_1}(y)\ud y\Bigr| \\
		&\le C \sqrt{\gamma_{2\varepsilon_1}(0)}\Bigl(\int_0^t \mathbb{E}\bigl[S_{n-1,\tilde{\varepsilon}}(g_{n-1}(\cdot,s,0)) - S_{n-1}(g_{n-1}(\cdot,s,0))\bigr]^2 \ud s\Bigr)^{1/2}.
	\end{align*}
 By  monotonicity of $\mathcal F_{\varepsilon, \varepsilon'}^{\mathcal D}(s,s)$  in $s$ (see \eqref{e:cov-Sne}),
  for all $s\in[0,t]$ we have
$$
	\mathbb{E}\big|S_{n-1,\tilde{\varepsilon}}(g_{n-1}(\cdot,s,0))\big|^2\le  \mathbb{E}\big|S_{n-1,\tilde{\varepsilon}}(g_{n-1}(\cdot,t,0))\big|^2
	\le 2\,\mathbb{E}\bigl[S_{n-1}(g_{n-1}(\cdot,t,0))\bigr]^2 < \infty
$$
for $\tilde \varepsilon$ sufficiently small by Theorem~\ref{th:strat-intablity} (i). Similarly, we also have $$\mathbb{E}\bigl[S_{n-1}(g_{n-1}(\cdot,s,0))\bigr]^2\le \mathbb{E}\bigl[S_{n-1}(g_{n-1}(\cdot,t,0))\bigr]^2<\infty$$ for all $s\in[0,t].$ Then by Theorem~\ref{th:strat-intablity} (i) and  the dominated convergence theorem, the limit in~\eqref{E:fubini-i} holds in
$\mathcal{L}^1$.


We now show that the convergence in \eqref{E:fubini-i} in fact holds in $\mathcal{L}^p$
for every $p\ge 1$.  By the Hu--Meyer formula \cite{HM1988,Hu2017Analysis} and the
hypercontractivity property (or alternatively, by Proposition~\ref{prop:Hypercon} below, which,
although stated for the full sum $u(t,x)$, is proved by establishing the estimate for each
term in the Stratonovich expansion separately, and therefore applies equally to
$S_{n-1,\tilde{\varepsilon}}(g_{n-1}(\cdot,t,x))$), we have
\[
\sup_{\tilde{\varepsilon}} \mathbb{E}\bigl|S_{n-1,\tilde{\varepsilon}}(g_{n-1}(\cdot,t,x))\bigr|^{2p} < \infty
\qquad\text{for every } p\ge 1.
\]
Consequently,
\[
\bigl\{ |S_{n-1,\tilde{\varepsilon}}(g_{n-1}(\cdot,t,x))|^{p}
\mid \tilde{\varepsilon}=(\varepsilon_2,\dots,\varepsilon_n),\ \varepsilon_i>0 \bigr\}
\]
is uniformly integrable.  Since $\mathcal{L}^1$-convergence implies convergence in probability,
Vitali convergence theorem yields convergence in $\mathcal{L}^p$ for all $p\ge 1$.

Finally, letting $\varepsilon_1\to 0$ in \eqref{E:fubini-i} and applying Theorem~\ref{th:strat-intablity} (i)
gives \eqref{E:fubini-ii}.
\end{proof}

\section{Existence of the solution}\label{se:existence}

In this section, under Assumptions~\ref{A:Main} and \ref{A:G}, we  prove the existence of a mild Stratonovich solution to~\eqref{e:sfde} and provide an upper bound for the second moment (see Theorem~\ref{th:exist}). Moreover, we show that Dalang's condition~\eqref{dalang} is also necessary for the existence of a square-integrable Stratonovich solution in Theorem~\ref{th:necessity-dalang}.

We first outline our strategy for proving the existence of a solution. Recall that $S_{n,\e}$ is given by~\eqref{e:S-ne}. 
By  Theorem \ref{th:strat-intablity}, we  define $S_n$ as an $\cL^2$-limit of $S_{n,\e}$ in the sense of Definition~\ref{D:multiple Stratonovich integration}.
This allows us to define the formal series representation \eqref{e:expan-utx} for the solution $u(t,x)$.
By Definitions~\ref{D:Stratonovich integral} and~\ref{D: mild solution}, to prove that $u(t,x)$ is a mild Stratonovich solution, it suffices to show 
$$
\lim_{\e_1 \rightarrow 0^{+}} \int_{\mathbb{R}^d} \left( \int_0^t G(t-s, x-y) u(s, y) \ud s \right) \dot{W}_{\varepsilon_1}(y) \ud y = \int_{\mathbb{R}^d} \left( \int_0^t G(t-s, x-y) u(s, y) \ud s \right) W(\ud y),
$$
in $\mathcal{L}^2(\Omega, \mathcal{F}, \mathbb{P})$. Given the expression~\eqref{e:expan-utx} of $u(t,x)$, this reduces to proving the following two results:
\begin{enumerate}[(i)]
    \item for each $n \geq 1$, 
$$
\int_{\mathbb{R}^d} \left( \int_0^t G(t-s, x-y) S_{n-1}(g_{n-1}(\cdot, s, y)) ds \right) \dot{W}_{\varepsilon}(y) \ud y
$$
converges to $S_n(g_n(\cdot, t, x))$ in $\mathcal{L}^2(\Omega, \mathcal{F}, \mathbb{P})$, which has been verified in Lemma~\ref{le:fubini-th};

\item for any $t>0$,
\begin{equation}\label{e:existence-con1}
\sum_{n=0}^{\infty} \theta^{n/2} \mathbb{E} \left[|S_n(g_n(\cdot, t, x))|^2\right]^{1/2} <\infty,
\end{equation}
and
\begin{equation}\label{e:existence-con2}
\lim_{M\to\infty}\sup_{\e>0}\sum_{n=M}^{\infty} \theta^{n/2}  \mathbb{E} \left[|S_{n,\e}(g_n(\cdot, t, x))|^2\right]^{1/2} =0,
\end{equation}
which will be proved later.
\end{enumerate}

\subsection{Proof of Theorem \ref{th:exist}}\label{proof:exist}
\begin{proof}[Proof of Theorem \ref{th:exist}]

To prove the existence of the solution, it suffices to show~\eqref{e:existence-con1} and~\eqref{e:existence-con2} hold.

{\bf Step 1: when $b+2r\geq1$.} In this case, Assumption \ref{A:Main} implies \eqref{e:dalang-plus}.
Cauchy-Schwartz inequality together with \eqref{E:even Gaussian moment}  yields
\begin{align*}
	&\mathbb{E}  \otimes \mathbf{E}_0 \left[\int_0^{t_1} \dot{W}\left(X_1(s)\right) \ud s\right]^n \left[\int_0^{t_2} \dot{W}\left(X_2(s)\right) \ud s\right]^n\ \\
	&\le \left(\mathbb{E}  \otimes \mathbf{E}_0 \left[\int_0^{t_1} \dot{W}\left(X(s)\right) \ud s\right]^{2n}\right)^{1/2} \left(\mathbb{E}  \otimes \mathbf{E}_0 \left[\int_0^{t_2} \dot{W}\left(X(s)\right) \ud s\right]^{2n}\right)^{1/2} \\
	&\leq \frac{(2 n)!}{2^n n!}\left\{\mathbf{E}_0\left[\int_0^{t_1} \int_0^{t_1} \gamma(X(s)-X(r)) \ud s \ud r\right]^n\right\}^{1 / 2}  
	\left\{\mathbf{E}_0\left[\int_0^{t_2} \int_0^{t_2} \gamma(X(s)-X(r)) \ud s \ud r\right]^n\right\}^{1 / 2}.
\end{align*}

Let $t>0$ be fixed. By \eqref{eq:ESS 2}, we have
\begin{align}
	&	\int_0^{\infty}  \int_0^{\infty} \ud t_1 \ud t_2  e^{-\lambda t_1-\lambda t_2} \mathbb{E}\left[S_n\left(g_n\left(\cdot, t_1, 0\right)\right) S_n\left(g_n\left(\cdot, t_2, 0\right)\right)\right] \nonumber\\
	&\le \frac{1}{(n!)^2} \lambda^{2b-2nr-2} \left(\frac{2}{\nu}\right)^{2n+2}  \frac{(2 n)!}{2^n n!}    \int_0^{\infty} \int_0^{\infty} \ud t_1 \ud t_2  e^{-2\lambda_1^{b} t_1/\nu-2\lambda_2^{b} t_2/\nu} \nonumber\\
	& \qquad \times\left\{\mathbf{E}_0\left[\int_0^{t_1} \int_0^{t_1} \gamma(X(s)-X(r)) \ud s \ud r\right]^n\right\}^{1 / 2}\left\{\mathbf{E}_0\left[\int_0^{t_2} \int_0^{t_2} \gamma(X(s)-X(r)) \ud s \ud r\right]^n\right\}^{1 / 2} \nonumber \\
	&= \frac{(2 n)!}{2^n (n!)^3} \lambda^{2b-2nr-2} \left(\frac{2}{\nu}\right)^{2n+2}   \bigg\{\int_0^{\infty} \ud \tilde{t} e^{-2\lambda_1^{b} \tilde{t}/\nu}  \bigg(\mathbf{E}_0 \bigg[ \int_0^{\tilde{t}} \int_0^{\tilde{t}} \gamma(X(s)-X(r)) \ud s \ud r \bigg]^n\bigg)^{1 / 2}\bigg\}^2.  \label{E:existence proof1}
\end{align}
 We now show that there exist positive constants $C_{1,\eta}$ and $C_{2,\eta}$ such that
\begin{equation}
	\mathbf{E}_0 \exp\!\bigg( \frac{\eta}{\tilde{t}} \int_0^{\tilde{t}}\!\!\!\int_0^{\tilde{t}} \gamma\big(X(s)-X(r)\big) \ud s \ud r \bigg) \le C_{1,\eta} \exp( C_{2,\eta} \tilde{t} ). \label{eq:Hamiltonbound1-new}
\end{equation}
Expanding the exponential as a power series and applying Lemma~\ref{L:n momentbound}, we obtain
\begin{equation*}
	\begin{aligned}
		\mathbf{E}_0 \exp\!\bigg( \frac{\eta}{\tilde{t}} \int_0^{\tilde{t}}\!\!\!\int_0^{\tilde{t}} \gamma\big(X(s)-X(r)\big) \ud s \ud r \bigg) 
		&= \sum_{n=0}^{\infty} \frac{\eta^n}{n! \tilde{t}^{n}} \mathbf{E}_0\!\left[ \int_0^{\tilde{t}}\!\!\!\int_0^{\tilde{t}} \gamma\big(X(s)-X(r)\big) \ud s \ud r \right]^n \\
		&\le \sum_{n=0}^{\infty} \frac{(2n-1)!!}{n!} \eta^n \sum_{k=0}^{n} \binom{n}{k} \frac{\tilde{t}^{k}}{k!} m_N^{k} (A_0\varepsilon_N)^{n-k},
	\end{aligned}
\end{equation*}
where $m_N = \mu(|\xi|\le N)$ and $\varepsilon_N = \int_{|\xi|\ge N} |\xi|^{-a}\mu(d\xi)$.

Note that $(2n-1)!!/n! \le 2^{n}$. Hence
\begin{equation}\label{eq:Hamiltonbound2-new}
	\begin{aligned}
		\mathbf{E}_0 \exp\!\bigg( \frac{\eta}{\tilde{t}} \int_0^{\tilde{t}}\!\!\!\int_0^{\tilde{t}} \gamma\big(X(s)-X(r)\big) \ud s \ud r \bigg)
		&\le \sum_{n=0}^{\infty} \sum_{k=0}^{n} \binom{n}{k} \frac{(2\eta\tilde{t} m_N)^{k}}{k!} (2\eta A_0\varepsilon_N)^{n-k} \\
		&= \sum_{k=0}^{\infty} \frac{(2\eta\tilde{t} m_N)^{k}}{k!} \sum_{j=0}^{\infty} \binom{k+j}{k} (2\eta A_0\varepsilon_N)^{j}, \qquad (j = n-k). 
	\end{aligned}
\end{equation}
Now choose $N$ large enough so that $2\eta A_0\varepsilon_N<1$. Then we have
\begin{equation*}
	\sum_{j=0}^{\infty} \binom{k+j}{k} (2\eta A_0\varepsilon_N)^{j} = \frac{1}{(1 - 2\eta A_0\varepsilon_N)^{k+1}}.
\end{equation*}
Substituting this into \eqref{eq:Hamiltonbound2-new},
\begin{equation*}
	\begin{aligned}
		\mathbf{E}_0 \exp\!\bigg( \frac{\eta}{\tilde{t}} \int_0^{\tilde{t}}\!\!\!\int_0^{\tilde{t}} \gamma\big(X(s)-X(r)\big) \ud s \ud r \bigg)
		&\le \frac{1}{1-2\eta A_0\varepsilon_N} \sum_{k=0}^{\infty} \frac{1}{k!} \left( \frac{2\eta\tilde{t} m_N}{1-2\eta A_0\varepsilon_N} \right)^{\!k} \\
		&= \frac{1}{1-2\eta A_0\varepsilon_N} \exp\!\left( \frac{2\eta m_N}{1-2\eta A_0\varepsilon_N} \tilde{t} \right).
	\end{aligned}
\end{equation*}
Now set
\begin{equation*}
	C_{1,\eta} := \frac{1}{1-2\eta A_0\varepsilon_N},\qquad
	C_{2,\eta} := \frac{2\eta m_N}{1-2\eta A_0\varepsilon_N}.
\end{equation*}
Then we obtain \eqref{eq:Hamiltonbound1-new}.

Using the fact that, for any $\tilde{t}>0$, 
\begin{equation*}
	\frac{\eta^n}{n!\tilde{t}^n} \mathbf{E}_0\left[\int_0^{\tilde{t}} \int_0^{\tilde{t}} \gamma(X(s)-X(r)) \ud s \ud r\right]^n \leq \mathbf{E}_0 \exp \left\{\frac{\eta}{\tilde{t}} \int_0^{\tilde{t}} \int_0^{\tilde{t}} \gamma(X(s)-X(r)) \ud s \ud r\right\},
\end{equation*}
we get
\begin{equation*}
	\mathbf{E}_0\left[\int_0^{\tilde{t}} \int_0^{\tilde{t}} \gamma(X(s)-X(r)) \ud s \ud r\right]^n \leq C_{1,\eta}\, n! \left(\frac{\tilde{t}}{\eta}\right)^{\!n} \exp\!\big(C_{2,\eta}\,\tilde{t}\big).
\end{equation*}

Substituting this into the integral in \eqref{E:existence proof1} yields
\begin{align*}
	&\int_0^{\infty}e^{-2\lambda_1^{b} \tilde{t}/\nu}
	\left(\mathbf{E}_0\left[\int_0^{\tilde{t}} \int_0^{\tilde{t}} \gamma(X(s)-X(r)) \ud s \ud r\right]^n\right)^{\!1/2} \ud \tilde{t} \\
	&\qquad \leq \sqrt{C_{1,\eta}}\, (n!)^{1/2}\eta^{-n/2}
	\int_0^{\infty} \exp\!\left(-\left(\frac{2\lambda^{b}}{\nu}-\frac{C_{2,\eta}}{2}\right)\tilde{t}\right) \tilde{t}^{\,n/2} \ud \tilde{t}.
\end{align*}
If $\lambda^{b}\geq \frac{\nu C_{2,\eta}}{2}$, then
\begin{equation*}
	\frac{2\lambda^{b}}{\nu}-\frac{C_{2,\eta}}{2}
	\geq\frac{\lambda^{b}}{\nu}.
\end{equation*}
Hence, for such $\lambda$,
\begin{align}
	&\int_0^{\infty}e^{-2\lambda_1^{b} \tilde{t}/\nu}
	\left(\mathbf{E}_0\left[\int_0^{\tilde{t}} \int_0^{\tilde{t}} \gamma(X(s)-X(r)) \ud s \ud r\right]^n\right)^{\!1/2} \ud \tilde{t}\nonumber\\
	&\qquad \leq C_\eta\, (n!)^{1/2}\eta^{-n/2}
	\left(\frac{\nu}{\lambda^{b}}\right)^{n/2+1}
	\Gamma\!\left(\frac{n}{2}+1\right),
	\label{E:single-integral-bound}
\end{align}
where $C_\eta$ is a constant depending only on $\eta$.

Combining \eqref{E:existence proof1} with \eqref{E:single-integral-bound} and using Stirling's formula, we obtain, for every $\lambda$ satisfying $\lambda^b\geq \frac{\nu C_{2,\eta}}{2}$,
\begin{align}
	&\int_0^{\infty}\int_0^{\infty} e^{-\lambda(t_1+t_2)}
	\mathbb{E}\left[S_n\left(g_n(\cdot,t_1,0)\right)
	S_n\left(g_n(\cdot,t_2,0)\right)\right]\ud t_1\ud t_2
	\leq C_\eta\left(\frac{C}{\eta}\right)^n
	n!\,\lambda^{-(b+2r)n-2}.
	\label{E:existence proof3-new}
\end{align}

By the fact that the moment
$
\mathbb{E}\left[S_n\left(g_n\left(\cdot, t_1, 0\right)\right) S_n\left(g_n\left(\cdot, t_2, 0\right)\right)\right]
$
is non-negative and non-decreasing in $t_1$ and $t_2,$ we have
\begin{align}
	\int_0^{\infty}\int_0^{\infty} e^{-\lambda(t_1+t_2)}
	\mathbb{E}\left[S_n\left(g_n(\cdot,t_1,0)\right)
	S_n\left(g_n(\cdot,t_2,0)\right)\right]\ud t_1\ud t_2\geq \lambda^{-2}e^{-2\lambda t}
	\mathbb{E}\left[S_n\left(g_n(\cdot,t,0)\right)\right]^2.\label{E:existence proof4-new}
\end{align}
Now choose 
$$\lambda=\frac{(b+2r)n}{2t}.$$
Since this $\lambda$ increases with $n$, there exists $n_0$ such that for all $n\geq n_0$ the condition $\lambda^b\geq \frac{\nu C_{2,\eta}}{2}$ holds. Combining \eqref{E:existence proof3-new} and \eqref{E:existence proof4-new} and substituting $\lambda=\frac{(b+2r)n}{2t}$, we obtain
\begin{align}
	\mathbb{E}\left[S_n\left(g_n(\cdot,t,0)\right)\right]^2
	&\leq C_\eta\left(\frac{C}{\eta}\right)^n n!
	\left(\frac{2et}{(b+2r)n}\right)^{(b+2r)n}
	\leq C_\eta
	\frac{1}{(n!)^{b+2r-1}} \left(\frac{t^{b+2r}}{\eta}\right)^n. \nonumber
\end{align}

For $1\leq n<n_0$, the left-hand side of
\eqref{E:single-integral-bound}, with $\lambda=\frac{(b+2r)n}{2t}$, is also finite.
Indeed, by \eqref{nbound1} the $n$th moment inside the square root is bounded by a polynomial in $\tilde t$, and the factor
$e^{-2\lambda^b\tilde t/\nu}$ guarantees integrability.
Since only finitely many such $n$ occur, their contribution can be absorbed into the constant.

By Corollary~\ref{coro:3.5}, the same argument gives, uniformly in $\varepsilon$,
\begin{equation}
	\mathbb{E}\left[S_{n,\varepsilon}\left(g_n(\cdot,t,0)\right)\right]^2
	\leq
	C_\eta
	\frac{1}{(n!)^{b+2r-1}} \left(\frac{t^{b+2r}}{\eta}\right)^n.
	\label{e:S-n-e-2-new}
\end{equation}
for all sufficiently large $n$.

If $b+2r>1$, we take $\eta=1$. If $b+2r=1$, then for the fixed $t$ and $\theta$ under consideration we choose $\eta$ sufficiently large so that $\frac{t^{b+2r}}{\eta}<1$. In either case,
\begin{equation*}
	\sum_{n=0}^\infty \theta^{n/2}
	\left(\mathbb{E}\left[S_n\left(g_n(\cdot,t,0)\right)\right]^2\right)^{1/2}<\infty.
\end{equation*}
Thus \eqref{e:existence-con1} follows.  The uniform estimate \eqref{e:S-n-e-2-new} proves
\eqref{e:existence-con2} in the same way.  This proves the existence under (i) in Theorem \ref{th:exist}.

{\bf Step 2: when $b+2r<1$.} In this case, \eqref{e:dalang-plus} implies Assumption \ref{A:Main} and $\delta<\frac{2(b+r)-1}{b}<1$.
We retain the ordered-integral estimate preceding the last simplification in
\cite[Remark~3.6]{song2017}.  More precisely, taking $\beta_0=0$,
$\Psi(\xi)=|\xi|^a$, and $\varepsilon_0=1-\delta$ in that argument, the
ordering of the $2n$ time variables, followed by the independent-increment
estimate in \cite[(3.2)--(3.3)]{song2017}, yields
\begin{align}
	&\mathbf E_0\left[
	\int_0^{\tilde t}\int_0^{\tilde t}
	\gamma(X(s)-X(r))\ud s\ud r
	\right]^n\nonumber\\
	&\quad\leq C^n(2n-1)!! \tilde t^n
	\int_{0<z_1<\cdots<z_n<\tilde t}
	\prod_{j=1}^n
	\left(\int_{\mathbb R^d}e^{-(z_j-z_{j-1})|\xi|^a}
	\mu(\ud\xi)\right)\ud z_1\cdots\ud z_n,
\nonumber
\end{align}
where $z_0=0$. Define
\begin{equation*}
	\mathcal I_n(\tilde t):=	\int_{0<z_1<\cdots<z_n<\tilde t}
	\prod_{j=1}^n
	\left(\int_{\mathbb R^d}e^{-(z_j-z_{j-1})|\xi|^a}
	\mu(\ud\xi)\right)\ud z_1\cdots\ud z_n.
\end{equation*}

By Cauchy-Schwarz inequality, we have
\begin{align}
	&\left\{\int_0^\infty e^{-2\lambda^b\tilde t/\nu}
	\left\{\mathbf E_0\left[
	\int_0^{\tilde t}\int_0^{\tilde t}
	\gamma(X(s)-X(r))\ud s\ud r
	\right]^n\right\}^{1/2}\ud\tilde t\right\}^2\nonumber\\
	&\leq  C^n(2n-1)!!\int_0^\infty e^{-2\lambda^b\tilde t/\nu}
\tilde t^n\ud\tilde t \int_0^\infty e^{-2\lambda^b\tilde t/\nu}
	\mathcal I_n(\tilde t) \ud\tilde t \nonumber\\
	&\leq  C^n(2n-1)!!  n! \left( \frac{\nu}{2\lambda^b} \right)^{n+1}
	\int_0^\infty e^{-2\lambda^b\tilde t/\nu}
	\mathcal I_n(\tilde t) \ud\tilde t.                   	\label{E:strong-CS-LT}
\end{align}
 By \eqref{e:dalang-plus} and \cite[Lemma~3.9]{song2017}, for every $\lambda>0$, 
\begin{align*}
	\int_0^\infty e^{-2\lambda^bs/\nu}
	\left(
	\int_{\mathbb R^d}e^{-s|\xi|^a}\mu(\ud\xi)
	\right)\ud s
	&\leq
	C\int_0^\infty e^{-2\lambda^bs/\nu}
	(1+s^{-\delta})\ud s\\
	&=
	C\left\{
	\frac{\nu}{2\lambda^b}
	+
	\Gamma(1-\delta)
	\left(\frac{\nu}{2\lambda^b}\right)^{1-\delta}
	\right\}.
\end{align*}
In particular, whenever $2\lambda^b/\nu\geq1$,
\begin{equation}
	\int_0^\infty e^{-2\lambda^bs/\nu}
	\left(
	\int_{\mathbb R^d}e^{-s|\xi|^a}\mu(\ud\xi)
	\right)\ud s
	\leq C\lambda^{-b(1-\delta)}.
	\nonumber
\end{equation}
On the other hand, Lemma~\ref{le:chenbook-lem227} gives
\begin{align*}
\int_0^\infty e^{-2\lambda^b\tilde t/\nu}
	\mathcal I_n(\tilde t)\ud\tilde t=
	\frac{\nu}{2\lambda^b}
	\left\{
	\int_0^\infty e^{-2\lambda^bs/\nu}
	\left(
	\int_{\mathbb R^d}e^{-s|\xi|^a}\mu(\ud\xi)
	\right)\ud s
	\right\}^n.
\end{align*}
Consequently, if $2\lambda^b/\nu\geq1$, then
\begin{equation}
	\int_0^\infty e^{-2\lambda^b\tilde t/\nu}
	\mathcal I_n(\tilde t)\ud\tilde t
	\leq
	C^n\lambda^{-b}\lambda^{-b(1-\delta)n}.
	\label{E:strong-ordered-LT}
\end{equation}
Substituting \eqref{E:strong-ordered-LT} into
\eqref{E:strong-CS-LT}, we obtain
\begin{align}
	&\left\{\int_0^\infty e^{-2\lambda^b\tilde t/\nu}
	\left\{\mathbf E_0\left[
	\int_0^{\tilde t}\int_0^{\tilde t}
	\gamma(X(s)-X(r))\ud s\ud r
	\right]^n\right\}^{1/2}\ud\tilde t\right\}^2 \leq
	C^n(2n-1)!!\,n!
	\lambda^{-b\{(2-\delta)n+2\}}.
	\label{E:strong-Hamiltonian-LT}
\end{align}

Substituting \eqref{E:strong-Hamiltonian-LT} into \eqref{E:existence proof1}, we obtain
\begin{align}
	&\int_0^\infty\int_0^\infty e^{-\lambda(t_1+t_2)}
	\mathbb E\left[
	S_n(g_n(\cdot,t_1,0))
	S_n(g_n(\cdot,t_2,0))
	\right]\ud t_1\ud t_2\nonumber\\
	&\quad\leq
	\frac{(2n)!}{2^n(n!)^3}
	\lambda^{2b-2nr-2}
	\left(\frac{2}{\nu}\right)^{2n+2}
	C^n(2n-1)!!\,n!
	\lambda^{-b\{(2-\delta)n+2\}}\nonumber\\
		&\quad\leq
	C^n n!\lambda^{-\{[2(b+r)-b\delta]n+2\}},
	\nonumber
\end{align}
where the last inequality follows from the estimate
$$
\frac{(2n)!}{2^n(n!)^3}(2n-1)!!n! \leq C^n n!.
$$

Therefore, similarly to \eqref{E:existence proof4-new}, we easily obtain
\begin{equation}
	\mathbb E\left[S_n(g_n(\cdot,t,0))\right]^2
	\leq
	C^n n!e^{2\lambda t}
	\lambda^{-[2(b+r)-b\delta]n}.
	\label{E:strong-Sn-lambda}
\end{equation}
Let $2(b+r)-b\delta>1$ and take
$$
\lambda=\frac{[2(b+r)-b\delta]n}{2t}.
$$
For all sufficiently large $n$, this choice satisfies
$2\lambda^b/\nu\geq1$, and hence all the preceding estimates apply.
It follows from \eqref{E:strong-Sn-lambda} that
\begin{align}
	\mathbb E\left[S_n(g_n(\cdot,t,0))\right]^2
	&\leq
	C^n n!t^{[2(b+r)-b\delta]n}
	n^{-[2(b+r)-b\delta]n}\nonumber\\
	&\leq
	\frac{C^nt^{[2(b+r)-b\delta]n}}
	{(n!)^{2(b+r)-b\delta-1}}.
\nonumber
\end{align}
Repeating the above discussion, we obtain the existence under (ii).

{\bf Step 3: when \eqref{E:scalef} hold.}
Taking $\lambda = \eta = 1$ in \eqref{E:existence proof3-new}, we have
\begin{align}
	\int_0^{\infty} & \int_0^{\infty} \ud t_1 \ud t_2 \: e^{-t_1-t_2} \mathbb{E}\left[S_n\left(g_n\left(\cdot, t_1, 0\right)\right) S_n\left(g_n\left(\cdot, t_2, 0\right)\right)\right] 
    \leq n! C^n. \label{E:existence proof4}
\end{align}
 By the monotonicity of $g_n(\cdot, t, 0)$ in $t$, together with \eqref{E:scalef} and \eqref{e:scal-G}, it follows that 
 \begin{align}
 	&\mathbb{E}\left[S_n\left(g_n\left(\cdot, t_1, 0\right)\right) S_n\left(g_n\left(\cdot, t_2, 0\right)\right)\right]  \nonumber\\
 	&\geq \mathbb{E}\left[S_n\left(g_n(\cdot, \min _{1 \leq j \leq 2} t_j, 0)\right)\right]^2 
    =\mathbb{E}\left[S_n\left(g_n(\cdot, 1, 0)\right)\right]^2\left(\min _{1 \leq j \leq 2} t_j\right)^{[2(b+r)-b \alpha / a] n}. \nonumber
 \end{align}
 Thus,
 \begin{equation}
     \label{E:existence proof5}
 \begin{aligned}
 	 \int_0^{\infty} & \int_0^{\infty} \ud t_1 \ud t_2 e^{-t_1-t_2} \mathbb{E}\left[S_n\left(g_n\left(\cdot, t_1, 0\right)\right) S_n\left(g_n\left(\cdot, t_2, 0\right)\right)\right]  \\ 	\geq&\mathbb{E}\left[S_n\left(g_n(\cdot, 1, 0)\right)\right]^2  \int_0^{\infty} \int_0^{\infty} \ud t_1 \ud t_2 e^{- t_1- t_2}\left(\min _{1 \leq j \leq 2} t_j\right)^{[2(b+r)-b \alpha / a] n} .
 \end{aligned}
 \end{equation}
 
 Recall that for two independent exponential random variables with parameter $1$, their minimum follows an exponential distribution with parameter $2$. Therefore, 
 \begin{equation}\label{E:existence proof6}
 \begin{aligned}
 	\int_0^{\infty} \int_0^{\infty} \ud t_1 \ud t_2 e^{- t_1- t_2}\left(\min _{1 \leq j \leq 2} t_j\right)^{[2(b+r)-b \alpha / a] n} &=2 \int_0^{\infty} e^{-2 t} t^{[2(b+r)-b \alpha / a] n} d t \\
 	&=2^{-[2(b+r)-b \alpha / a] n} \Gamma(1+[2(b+r)-b \alpha / a] n),
 \end{aligned}
 \end{equation}
which holds under the condition $2a(b+r)-b \alpha> 0.$

Combining \eqref{E:existence proof4}, \eqref{E:existence proof5} and \eqref{E:existence proof6} yields 
\begin{equation}
    \label{E:existence proof7}
\begin{aligned}
&\mathbb{E}\left[S_n\left(g_n(\cdot, t, 0)\right)\right]^2 
    = \mathbb{E}\left[S_n\left(g_n(\cdot, 1, 0)\right)\right]^2t^{[2(b+r)-b \alpha / a] n} \\
	&\leq \frac{C^n n! (2t)^{[2(b+r)-b \alpha / a] n}}{\Gamma(1+[2(b+r)-b \alpha / a])} 
    =\frac{C^n  (2t)^{[2(b+r)-b \alpha / a] n}}{\left(n!\right)^{2(b+r)-b \alpha / a-1}}. 
\end{aligned}
\end{equation}
From \eqref{E:existence proof7}, we require
\begin{equation*}
	2(b+r)-\frac{b\alpha}{a}-1>0,
\end{equation*}
or, equivalently,
\begin{equation*}
	\alpha<\frac{a}{b}\bigl(2(b+r)-1\bigr).
\end{equation*}
Combining this with Assumption \ref{A:Main} and the equivalent form $\alpha<\min\{a,d\}$ of the scaling condition \eqref{E:scalef}, we obtain that under \eqref{eq:scal-dalang-new} we have
\begin{align}
  &\mathbb{E}u^2(t,x)\leq \left(\sum_{n=0}^{\infty}\theta^{n/2}\left\{\mathbb{E}\left[S_n\left(g_n(\cdot, t, 0)\right)\right]^2\right\}^{1 / 2} \right)^2 \nonumber\\
   &\leq \left[\sum_{n=0}^{\infty}\theta^{n/2}\left(\frac{C^n  t^{[2(b+r)-b \alpha / a] n}}{\left(n!\right)^{2(b+r)-b \alpha / a-1}}\right)^{1 / 2} \right]^2
  \leq e^{C(2t)^{\frac{2 a(b+r)-b \alpha}{2 a(b+r)-b \alpha-a}}}. \nonumber
  \end{align}

  Now consider the case $2(b+r)-\frac{b\alpha}{a}-1=0$. Then \eqref{E:existence proof7} reduces to
$$
\mathbb{E}\left[S_n\left(g_n(\cdot,t,0)\right)\right]^2 \leq C^n t^n.
$$
  	Choose $t_0=t_0(\theta)>0$ so that
	$C \theta t_0<1$.  The resulting geometric series, together with
	its uniform mollified analogue, proves \eqref{e:existence-con1} and
	\eqref{e:existence-con2} for $0<t<t_0$, and hence gives the asserted local
	solution at the critical value. This completes the proof. 
    \end{proof}

\subsection{Proof of Theorem \ref{th:necessity-dalang}}\label{proof:nece}

\begin{proof}[Proof of Theorem \ref{th:necessity-dalang}]
 It suffices to show that if Dalang's condition~\eqref{dalang} is violated, then   
 $$\mathbb{E}\bigl[S_{2}(g_2(\cdot,t,0))\bigr]^2 = \infty.$$ 

	Since $g_2(\cdot,t,0)$ is nonnegative under condition \eqref{E:Pos}, we have
	\begin{align*}
		\mathbb{E} \bigl[S_2(g_2(\cdot, t, 0))\bigr]^2
		&= \sum_{\mathcal{D} \in \Pi_2} \int_{(\mathbb{R}^d)^4} \mathrm{d}x_1\mathrm{d}x_2\mathrm{d}x_3\mathrm{d}x_4 \Bigl(\prod_{\mathcal{D} \in \Pi_2} \gamma(x_j-x_k)\Bigr) g_2(x_1, x_2, t, 0) g_2(x_3, x_4, t, 0) \\
		&\geq \int_{(\mathbb{R}^d)^4} \mathrm{d}x_1\mathrm{d}x_2\mathrm{d}x_3\mathrm{d}x_4 \, \gamma(x_1-x_2) \gamma(x_3-x_4) g_2(x_1, x_2, t, 0) g_2(x_3, x_4, t, 0) \\
		&= \biggl( \int_{(\mathbb{R}^d)^2} \gamma(x_2-x_1) g_2(x_1, x_2, t, 0) \, \mathrm{d}x_1 \mathrm{d}x_2 \biggr)^2.
	\end{align*}
	For the integral inside the square, we obtain
	\begin{align*}
		& \int_{\left(\mathbb{R}^d\right)^2} \gamma\left(x_2-x_1\right) g\left(x_1, x_2, t, 0\right) \ud x_1 \ud x_2 \\
		& =\int_{[0, t]_{<}^2} \ud s_1 \ud s_2 \int_{\left(\mathbb{R}^d\right)^2} \gamma\left(x_2-x_1\right) G\left(s_1, x_1\right) G\left(s_2-s_1, x_2-x_1\right) \ud x_1 \ud x_2 \\
		& =\int_{[0, t]_{<}^2}\left(\int_{\mathbb{R}^d} G\left(s_1, x\right) \ud x\right)\left(\int_{\mathbb{R}^d} \gamma(x) G\left(s_2-s_1, x\right) \ud x\right) \ud s_1 \ud s_2.
	\end{align*}
By \eqref{E:FZ}, and 
$$\int_{\mathbb{R}^d} G\left(s_1, x\right) \ud x=\frac{s_{1}^{b+r-1}}{\Gamma(b+r)}, $$ we have
\begin{equation*}
    \begin{aligned}
 		& \int_{(\mathbb{R}^d)^2} \gamma(x_2-x_1) g_2(x_1, x_2, t, 0) \, \mathrm{d}x_1 \mathrm{d}x_2\\
		&=\int_{\mathbb{R}^d} \mu(\mathrm{d}\xi)
		\int_{0<s_1<s_2<t}
		\frac{s_1^{b+r-1}}{\Gamma(b+r)} (s_2-s_1)^{b+r-1}	E_{b,b+r}\Bigl(-\frac{\nu}{2}(s_2-s_1)^b |\xi|^a\Bigr)
		\mathrm{d}s_1 \mathrm{d}s_2
		\\
        &=\int_{\mathbb{R}^d} \mu(\mathrm{d}\xi) \sum_{n=0}^{\infty}
\frac{(-\frac{\nu}{2}|\xi|^a)^n}{\Gamma(b+r)\Gamma(b n+b+r)}
\int_0^t \mathrm ds_2
\int_0^{s_2} s_1^{b+r-1}(s_2-s_1)^{b n+b+r-1}\,\mathrm ds_1,
	\end{aligned}
    \end{equation*}
where the last identity follows from \eqref{e:mlf}.   Using the change of variables $s_1=s_2 z$ and Beta function, we have
\begin{align*}
    \int_0^{s_2}s_1^{b+r-1}(s_2-s_1)^{b n+b+r-1}\,\mathrm ds_1
&=
s_2^{b n+2b+2r-1}
\int_0^1 z^{b+r-1}(1-z)^{b n+b+r-1}\,\mathrm dz\\
&=s_2^{b n+2b+2r-1}\,
\frac{\Gamma(b+r)\Gamma(b n+b+r)}{\Gamma(b n+2b+2r)}.
\end{align*}
Then, integrating over $s_2$, we can see that
\begin{align*}
    & \int_{(\mathbb{R}^d)^2} \gamma(x_2-x_1) g_2(x_1, x_2, t, 0) \, \mathrm{d}x_1 \mathrm{d}x_2\\
    &=\int_{\mathbb{R}^d} \mu(\mathrm{d}\xi)\sum_{n=0}^{\infty}
\frac{(-\frac{\nu}{2}|\xi|^a)^n t^{b n+2b+2r}}{\Gamma(b n+2b+2r+1)} =\int_{\mathbb{R}^d} \mu(\mathrm{d}\xi) t^{2b+2r}E_{b,2b+2r+1}\Big(-\frac{\nu}{2}t^b|\xi|^a\Big).
\end{align*}

For $0<b<2$, by (1.8.28) of \cite{kst2006} ,
\begin{align*}
t^{2b+2r} E_{b,2b+2r+1}
\left(-\frac{\nu}{2}t^b|\xi|^a\right)
=\frac{2t^{b+2r}}{\nu\Gamma(b+2r+1)}|\xi|^{-a}
+O\left(|\xi|^{-2a}\right),
\qquad |\xi|\longrightarrow\infty.
\end{align*}
 Thus, we can find a constant $C$ such that  for some sufficiently large $R$,   	 $$\mathbb{E}\bigl[S_{2}(g_2(\cdot,t,0))\bigr]^2\ge C t^{2b+2r} \int_{\{|\xi|>R\}} |\xi|^{-a} \mu(\ud\xi),$$
the right-hand side of which is infinite if Dalang's condition~\eqref{dalang} is violated, noting that $\mu$ is locally integrable.

For $b=2$, the necessity follows directly from \cite[Theorem 1.1]{ch2024}. This completes the proof.
\end{proof}

\section{Upper bound}\label{se:upper}
 In this section, we establish an upper bound on the $p$-th moment of the solution to \eqref{e:sfde} for all real $p\in [2,\infty)$ in Theorem~\ref{T:upperbound}, and also provide the \nameref{proof:1} of Theorem \ref{th:excat} (i). 

We first establish a hypercontractivity property (see Proposition~\ref{prop:Hypercon}) for the mild Stratonovich solution to~\eqref{e:sfde}, which enables us to bound  $p$-th moments of the solution for $p>2$ by the second moment. 

Let $\widetilde{W}$ be an independent copy of $W$.  Given   $\tau \geq 0$,  set $\rho=\sqrt{e^{2\tau}-1}$. For any $\varepsilon_1,\dots,\varepsilon_n>0$, write $\varepsilon=(\varepsilon_1,\dots,\varepsilon_n)$ and define
\begin{equation}\label{e:tilde-Sn}
    \widetilde{S}_{n,\varepsilon}(g_n(\cdot,t,x))
    := \int_{(\mathbb{R}^d)^n} g_n(x_1,\dots,x_n,t,x)\,
      \widetilde{\mathbb{E}} \prod_{k=1}^n 
        \bigl(e^{\tau}\dot{W}_{\varepsilon_k}(x_k) + \iota\rho\dot{\widetilde{W}}_{\varepsilon_k}(x_k)\bigr)\ud x_1\cdots \ud x_n,
\end{equation}
where $\widetilde{\mathbb{E}}$ denotes the expectation with respect to $\widetilde{W}$,  $g_n$ is given by \eqref{E:gn}, and $\dot{W}_\varepsilon,$ $\dot{\widetilde{W}}_\varepsilon$ are defined by \eqref{e:dot-we}.
The multiple Stratonovich integral $\widetilde{S}_n(g_n(\cdot,t,x))$ is then defined as the $\mathcal{L}^2$‑limit
\[
    \widetilde{S}_n(g_n(\cdot,t,x)) := \lim_{\varepsilon\to0^+} \widetilde{S}_{n,\varepsilon}(g_n(\cdot,t,x)),
\]
and we consider the series
\begin{equation}\label{e:u2}
    \widetilde{u}(t,x):=\sum_{n=0}^{\infty}\theta^{\frac{n}{2}}\,\widetilde{S}_n(g_n(\cdot,t,x)).
\end{equation}
We emphasize that that $\widetilde{S}_n(g_n)$ depends implicitly on $\tau$, or equivalently on  $\rho=\sqrt{e^{2\tau}-1}$.

To justify that $\widetilde{S}_n$ is well-defined as an $\mathcal{L}^2$‑limit and that the series~\eqref{e:u2} converges, we need to compute $\mathbb{E}\bigl[\widetilde{S}_n(f)\,\widetilde{S}_n(g)\bigr]$ for measurable functions $f,g:\mathbb{R}^{n}\to\mathbb{R}$. Given $\varepsilon_1,\dots,\varepsilon_{2n}>0$, denote $\varepsilon=(\varepsilon_1,\dots,\varepsilon_n)$ and $\varepsilon'=(\varepsilon_{n+1},\dots,\varepsilon_{2n})$, and set
$$
\widetilde{S}_{n,\varepsilon}(f)
:=\int_{(\mathbb{R}^d)^n} f(x_1,\dots,x_n)\,
 \widetilde{\mathbb{E}} \prod_{k=1}^n 
        \bigl(e^{\tau}\dot{W}_{\varepsilon_k}(x_k) + \iota\rho\dot{\widetilde{W}}_{\varepsilon_k}(x_k)\bigr) \ud x_1\cdots \ud x_n .
$$
The exact formula for 
$\mathbb{E}\bigl[\widetilde{S}_{n,\varepsilon}(f)\,
\widetilde{S}_{n,\varepsilon'}(g)\bigr]$
is provided in the lemma below.

\begin{lemma}\label{L:Wick2}
  For any $\tau\ge0$ set $\rho=\sqrt{e^{2\tau}-1}$.
  For any $n\in\mathbb{N}$, $x_1,\dots,x_{2n}\in\mathbb{R}^d$ and $\varepsilon_1,\dots,\varepsilon_{2n}>0$, let
  $I=\{1,\dots,n\}$, $J=\{n+1,\dots,2n\}$.  Then
  \begin{equation}\label{e:Wick-i}
  \begin{aligned}
    &\mathbb{E}_W\Biggl[ \widetilde{\mathbb{E}}\prod_{k\in I}\bigl(e^{\tau}\dot{W}_{\varepsilon_k}(x_k) + \iota\rho\dot{\widetilde{W}}_{\varepsilon_k}(x_k)\bigr)
    \cdot \widetilde{\mathbb{E}}\prod_{k\in J}\bigl(e^{\tau}\dot{W}_{\varepsilon_k}(x_k) + \iota\rho\dot{\widetilde{W}}_{\varepsilon_k}(x_k)\bigr) \Biggr] \\
    & = \sum_{\mathcal{D}\in\Pi_{2n}} \prod_{(u,v)\in\mathcal{D}} \gamma_{\varepsilon_u+\varepsilon_v}^{(\tau)}(x_u-x_v),
  \end{aligned}
  \end{equation}
  where $\Pi_{2n}$ is the set of all pair partitions of $\{1,\dots,2n\}$, and
  \[
    \gamma_{\varepsilon_u+\varepsilon_v}^{(\tau)}(x_u-x_v) =
    \begin{cases}
      \gamma_{\varepsilon_u+\varepsilon_v}(x_u-x_v),                & \text{if } u,v\in I \text{ or } u,v\in J,\\[6pt]
      e^{2\tau}\,\gamma_{\varepsilon_u+\varepsilon_v}(x_u-x_v),     & \text{otherwise.}
    \end{cases}
  \]
  \end{lemma}

\begin{proof}
  Set $a := \iota\rho e^{-\tau}$.  Note that $e^{\tau}\dot{W}_{\varepsilon}(x) + \iota\rho\dot{\widetilde{W}}_{\varepsilon}(x) = e^{\tau}\bigl(\dot{W}_{\varepsilon}(x) + a\dot{\widetilde{W}}_{\varepsilon}(x)\bigr)$, and hence  
  the left-hand side of \eqref{e:Wick-i} equals
\begin{equation}\label{e:Wick-i'}
    e^{2n\tau}\,
    \mathbb{E}_W\Bigl[ \widetilde{\mathbb{E}}\prod_{k\in I}\bigl(\dot{W}_{\varepsilon_k}(x_k) + a\dot{\widetilde{W}}_{\varepsilon_k}(x_k)\bigr)
    \cdot \widetilde{\mathbb{E}}\prod_{k\in J}\bigl(\dot{W}_{\varepsilon_k}(x_k) + a\dot{\widetilde{W}}_{\varepsilon_k}(x_k)\bigr) \Bigr].
\end{equation}
By the Wick formula  \eqref{e:wick-formula} and the covariance~\eqref{e:covar-dot-we},
  \begin{align}
    \widetilde{\mathbb{E}}\Bigl[\prod_{k\in I}\bigl(\dot{W}_{\varepsilon_k}(x_k)+a\dot{\widetilde{W}}_{\varepsilon_k}(x_k)\bigr)\Bigr]
    &= \sum_{\substack{S_I \subseteq I \\ |S_I| \text{ even}}}
        \Bigl( \prod_{i \in I\setminus S_I} \dot{W}_{\varepsilon_i}(x_i) \Bigr)
        \Bigl( \sum_{\mathcal{P} \in \Pi_{S_I}} 
            \prod_{(u,v) \in \mathcal{P}} a^2 \gamma_{\varepsilon_u+\varepsilon_v}(x_u-x_v) \Bigr), \label{eq:inner-I-final2} \\[6pt]
    \widetilde{\mathbb{E}}\Bigl[\prod_{k\in J}\bigl(\dot{W}_{\varepsilon_k}(x_k)+a\dot{\widetilde{W}}_{\varepsilon_k}(x_k)\bigr)\Bigr]
    &= \sum_{\substack{S_J \subseteq J \\ |S_J| \text{ even}}}
        \Bigl( \prod_{j \in J\setminus S_J} \dot{W}_{\varepsilon_j}(x_j) \Bigr)
        \Bigl( \sum_{\mathcal{Q} \in \Pi_{S_J}} 
            \prod_{(r,s) \in \mathcal{Q}} a^2 \gamma_{\varepsilon_r+\varepsilon_s}(x_r-x_s) \Bigr), \label{eq:inner-J-final2}
  \end{align}
  where $S_I$ is a subset of $I$ and $\Pi_{S_I}$ denotes the set of all pair partitions of $S_I$.
  For $S_I = \varnothing$, we set $\Pi_\varnothing = \{\varnothing\}$ and interpret the empty product as $1$.
  The definitions of $S_J$ and $\Pi_{S_J}$ are analogous to those of $S_I$ and $\Pi_{S_I}$.
	Multiplying \eqref{eq:inner-I-final2} and \eqref{eq:inner-J-final2} and taking the expectation over $W$, we obtain
	\begin{equation}\label{eq:after-expect2}
    \begin{aligned}
		&\mathbb{E}_W\Bigl[ \widetilde{\mathbb{E}}\prod_{k\in I}\bigl(\dot{W}_{\varepsilon_k}(x_k)+a\dot{\widetilde{W}}_{\varepsilon_k}(x_k)\bigr)
		\cdot \widetilde{\mathbb{E}}\prod_{k\in J}\bigl(\dot{W}_{\varepsilon_k}(x_k)+a\dot{\widetilde{W}}_{\varepsilon_k}(x_k)\bigr) \Bigr]  \\
		&= \sum_{\substack{S_I \subseteq I \\ |S_I| \text{ even}}}
		\sum_{\substack{S_J \subseteq J \\ |S_J| \text{ even}}}
		\sum_{\mathcal{P} \in \Pi_{S_I}}
		\sum_{\mathcal{Q} \in \Pi_{S_J}}
		\Bigl( \prod_{(u,v) \in \mathcal{P}} a^2 \gamma_{\varepsilon_u+\varepsilon_v}(x_u-x_v) \prod_{(r,s) \in \mathcal{Q}} a^2 \gamma_{\varepsilon_r+\varepsilon_s}(x_r-x_s) \Bigr) \\
		&\qquad \times \mathbb{E}_W\Bigl[ \prod_{i \in I\setminus S_I} \dot{W}_{\varepsilon_i}(x_i) \prod_{j \in J\setminus S_J} \dot{W}_{\varepsilon_j}(x_j) \Bigr].
	\end{aligned}
    \end{equation}
	Set $U := (I\setminus S_I) \cup (J\setminus S_J)$.  Since $W$ is a centered Gaussian field, the Wick formula gives
	\begin{equation}
		\mathbb{E}_W\Bigl[ \prod_{k \in U} \dot{W}_{\varepsilon_k}(x_k) \Bigr]
		= \sum_{\mathcal{R} \in \Pi_U} \prod_{\{p,q\} \in \mathcal{R}} \gamma_{\varepsilon_p+\varepsilon_q}(x_p-x_q), \label{eq:Wick-W2}
	\end{equation}
	where $\Pi_U$ is the set of all pair partitions of $U$.
	Substituting \eqref{eq:Wick-W2} into \eqref{eq:after-expect2} yields
	\begin{equation}
	    \label{eq:after-expect3}
    \begin{aligned}
		&\mathbb{E}_W\Bigl[ \widetilde{\mathbb{E}}\prod_{k\in I}\bigl(\dot{W}_{\varepsilon_k}(x_k)+a\dot{\widetilde{W}}_{\varepsilon_k}(x_k)\bigr)
		\cdot \widetilde{\mathbb{E}}\prod_{k\in J}\bigl(\dot{W}_{\varepsilon_k}(x_k)+a\dot{\widetilde{W}}_{\varepsilon_k}(x_k)\bigr) \Bigr]  \\
		&= \sum
		\Bigl( \prod_{(u,v) \in \mathcal{P}} a^2 \gamma_{\varepsilon_u+\varepsilon_v}(x_u-x_v) \prod_{(r,s) \in \mathcal{Q}} a^2 \gamma_{\varepsilon_r+\varepsilon_s}(x_r-x_s) \prod_{(p,q) \in \mathcal{R}} \gamma_{\varepsilon_p+\varepsilon_q}(x_p-x_q) \Bigr),  
	\end{aligned}
    \end{equation}
	where the summation $\sum$  means $\sum_{\substack{S_I \subseteq I \\ |S_I| \text{ even}}}
		\sum_{\substack{S_J \subseteq J \\ |S_J| \text{ even}}}
		\sum_{\mathcal{P} \in \Pi_{S_I}}
		\sum_{\mathcal{Q} \in \Pi_{S_J}}\sum_{\mathcal{R} \in \Pi_U}.$
        
	To prove the desired \eqref{e:Wick-i}, now we rewrite the  sum in \eqref{eq:after-expect3} as a sum over pair partitions $\mathcal D\in \Pi_{2n}$.

    For each $\mathcal P\in \Pi_{S_I}, \mathcal Q\in \Pi_{S_J}, \mathcal R\in \Pi_{S_U}$, set $\mathcal{D} := \mathcal{P}\cup\mathcal{Q}\cup\mathcal{R}$, and clearly	 $\mathcal{D}\in\Pi_{2n}$.  Decompose $\mathcal{D}$ as
	\[
	\mathcal{D}_1=\{(u,v)\in\mathcal{D}\mid u,v\in I\},\quad
	\mathcal{D}_2=\{(r,s)\in\mathcal{D}\mid r,s\in J\},\quad
	\mathcal{D}_3=\{(u,r)\in\mathcal{D}\mid u\in I,\;r\in J\}.
	\]
	From the construction we must have $\mathcal{P}\subseteq\mathcal{D}_1$, $\mathcal{Q}\subseteq\mathcal{D}_2$, and
	\begin{equation}\label{eq:R-from-D2}
		\mathcal{R}=(\mathcal{D}_1\setminus\mathcal{P})\cup(\mathcal{D}_2\setminus\mathcal{Q})\cup\mathcal{D}_3.
	\end{equation}
	
	Conversely, fix a  pair partition $\mathcal{D}\in\Pi_{2n}$ and choose arbitrary subsets $\mathcal{P}\subseteq\mathcal{D}_1$, $\mathcal{Q}\subseteq\mathcal{D}_2$.
	Define $S_I = \bigcup_{(u,v)\in\mathcal{P}}\{u,v\}$, $S_J = \bigcup_{(r,s)\in\mathcal{Q}}\{r,s\}$.
	Then $S_I\subseteq I$, $S_J\subseteq J$, both have even cardinality, and $\mathcal{P}\in\Pi_{S_I}$, $\mathcal{Q}\in\Pi_{S_J}$.
	Let $U = (I\setminus S_I)\cup(J\setminus S_J)$ and define $\mathcal{R}$ by \eqref{eq:R-from-D2}.
	Then $\mathcal{R}$ is a pair partition of $U$, so $\mathcal{R}\in\Pi_U$.
	
	
	This gives a bijection between the set of the triples $(\mathcal P, \mathcal Q, \mathcal R)$ included in the sum of  \eqref{eq:after-expect3} and the set of triples $(\mathcal{D},\mathcal{P},\mathcal{Q})$ satisfying $\mathcal{D}=\mathcal P\cup \mathcal Q\cup \mathcal R \in\Pi_{2n}$ with $\mathcal{P}\subseteq\mathcal{D}_1$, $\mathcal{Q}\subseteq\mathcal{D}_2$.
	Hence the  sum in \eqref{eq:after-expect3} equals
	\begin{equation}\label{eq:regrouped2}
		\begin{aligned}
			\sum_{\mathcal{D}\in\Pi_{2n}}\;
			\sum_{\mathcal{P}\subseteq\mathcal{D}_1}\;
			\sum_{\mathcal{Q}\subseteq\mathcal{D}_2}\;
			&\Bigl( \prod_{(u,v)\in\mathcal{P}} a^2 \gamma_{\varepsilon_u+\varepsilon_v}(x_u-x_v) \Bigr)
			\Bigl( \prod_{(r,s)\in\mathcal{Q}} a^2 \gamma_{\varepsilon_r+\varepsilon_s}(x_r-x_s) \Bigr) \Bigl( \prod_{(p,q)\in\mathcal{R}} \gamma_{\varepsilon_p+\varepsilon_q}(x_p-x_q) \Bigr),
		\end{aligned}
	\end{equation}
	where $\mathcal{R}$ is given by \eqref{eq:R-from-D2}.

	Now fix $\mathcal{D}$ and simplify the inner sums.
	For a pair $(u,v)\in\mathcal{D}_1$, if it belongs to $\mathcal{P}$ it contributes $a^2 \gamma_{\varepsilon_u+\varepsilon_v}(x_u-x_v)$; if not, it belongs to $\mathcal{R}$ and contributes $\gamma_{\varepsilon_u+\varepsilon_v}(x_u-x_v)$.
	Therefore its factor can be written as $	 \left(a^2\mathbf{1}_{\{(u,v)\in\mathcal{P}\}} + \mathbf{1}_{\{(u,v)\notin\mathcal{P}\}}\right)\gamma_{\varepsilon_u+\varepsilon_v}(x_u-x_v).$
	The same description applies to $\mathcal{D}_2$ with $\mathcal{Q}$.
	Every pair $(p,q)\in\mathcal{D}_3$ belongs to $\mathcal{R}$ and therefore contributes $\gamma_{\varepsilon_p+\varepsilon_q}(x_p-x_q)$.
	Thus the product in \eqref{eq:regrouped2} becomes
	\begin{align*}
		&\prod_{(u,v)\in\mathcal{D}_1} \bigl( a^2\mathbf{1}_{\{(u,v)\in\mathcal{P}\}} + \mathbf{1}_{\{(u,v)\notin\mathcal{P}\}} \bigr) \gamma_{\varepsilon_u+\varepsilon_v}(x_u-x_v)\\
		&\quad\times \prod_{(r,s)\in\mathcal{D}_2} \bigl( a^2\mathbf{1}_{\{(r,s)\in\mathcal{Q}\}} + \mathbf{1}_{\{(r,s)\notin\mathcal{Q}\}} \bigr) \gamma_{\varepsilon_r+\varepsilon_s}(x_r-x_s)
		\times \prod_{(p,q)\in\mathcal{D}_3} \gamma_{\varepsilon_p+\varepsilon_q}(x_p-x_q).
	\end{align*}
	
	Summing over all $\mathcal{P}\subseteq\mathcal{D}_1$ and using the fact 
    $$(a+b)^n= \sum_{k=0}^n \binom{n}{k} a^k b^{n-k}=\sum_{I\subset\{1,\dots,n\}}a^{|I|}b^{|I^c|}=\sum_{I\subset\{1,\dots, n\}}\prod_{k=1}^n\left(a\mathbf 1_{\{k\in I\}}+b\mathbf 1_{\{k\notin I\}}\right),$$ 
we obtain 
	\begin{align*}
		\sum_{\mathcal{P}\subseteq\mathcal{D}_1} \prod_{(u,v)\in\mathcal{D}_1} \bigl( a^2\mathbf{1}_{\{(u,v)\in\mathcal{P}\}} + \mathbf{1}_{\{(u,v)\notin\mathcal{P}\}} \bigr) \gamma_{\varepsilon_u+\varepsilon_v}(x_u-x_v)
		= \prod_{(u,v)\in\mathcal{D}_1} (1+a^2) \gamma_{\varepsilon_u+\varepsilon_v}(x_u-x_v).
	\end{align*}
	Similarly,
	\begin{align*}
		\sum_{\mathcal{Q}\subseteq\mathcal{D}_2} \prod_{(r,s)\in\mathcal{D}_2} \bigl( a^2\mathbf{1}_{\{(r,s)\in\mathcal{Q}\}} + \mathbf{1}_{\{(r,s)\notin\mathcal{Q}\}} \bigr) \gamma_{\varepsilon_r+\varepsilon_s}(x_r-x_s)
		= \prod_{(r,s)\in\mathcal{D}_2} (1+a^2) \gamma_{\varepsilon_r+\varepsilon_s}(x_r-x_s).
	\end{align*}
	The $\mathcal{D}_3$ part is independent of $\mathcal{P},\mathcal{Q}$ and hence remains unchanged.

	Multiplying these three factors yields, for the fixed $\mathcal{D}$,
	\begin{align*}
		&\prod_{(u,v)\in\mathcal{D}_1} (1+a^2) \gamma_{\varepsilon_u+\varepsilon_v}(x_u-x_v)
		\;\times\;
		\prod_{(r,s)\in\mathcal{D}_2} (1+a^2) \gamma_{\varepsilon_r+\varepsilon_s}(x_r-x_s)
		\;\times\;
		\prod_{(p,q)\in\mathcal{D}_3} \gamma_{\varepsilon_p+\varepsilon_q}(x_p-x_q)\\
		&\qquad = \prod_{(u,v)\in\mathcal{D}} \tilde\gamma_{\varepsilon_u+\varepsilon_v}(x_u-x_v),
	\end{align*}
	where
	\[
	\tilde\gamma_{\varepsilon_u+\varepsilon_v}(x_u-x_v) =
	\begin{cases}
		(1+a^2)\,\gamma_{\varepsilon_u+\varepsilon_v}(x_u-x_v), & \text{if } u,v\in I \text{ or } u,v\in J,\\[4pt]
		\gamma_{\varepsilon_u+\varepsilon_v}(x_u-x_v),         & \text{otherwise}.
	\end{cases}
	\]
	
	Summing this over all $\mathcal{D}\in\Pi_{2n}$ we obtain
	\[
	\mathbb{E}_W\Bigl[ \widetilde{\mathbb{E}}\prod_{k\in I}\bigl(\dot{W}_{\varepsilon_k}(x_k)+a\dot{\widetilde{W}}_{\varepsilon_k}(x_k)\bigr)
	\cdot \widetilde{\mathbb{E}}\prod_{k\in J}\bigl(\dot{W}_{\varepsilon_k}(x_k)+a\dot{\widetilde{W}}_{\varepsilon_k}(x_k)\bigr) \Bigr]
	= \sum_{\mathcal{D}\in\Pi_{2n}} \prod_{(u,v)\in\mathcal{D}} \tilde\gamma_{\varepsilon_u+\varepsilon_v}(x_u-x_v),
	\]
and the desired result \eqref{e:Wick-i} follows from \eqref{e:Wick-i'} and the fact     $1+a^2= e^{-2\tau}$.	
\end{proof}

Now we return to the definition of $\widetilde{S}_{n,\varepsilon}$ in \eqref{e:tilde-Sn}. Our goal is to show that the mollified approximations $\widetilde{S}_{n,\varepsilon}(g_n(\cdot, t, x))$ converge in $\mathcal{L}^2$ as $\varepsilon\to 0$, thereby defining $\widetilde{S}_n(g_n)$ as a Stratonovich integral in the sense of Definition~\ref{D:multiple Stratonovich integration}, and that the series $\widetilde{u}(t,x)$ in \eqref{e:u2} converges in $\mathcal{L}^2$.

We first compute the covariance of $\widetilde{S}_{n,\varepsilon}(g_n)$. For any $\varepsilon=(\varepsilon_1,\dots,\varepsilon_n)$ and $\varepsilon'=(\varepsilon_{n+1},\dots,\varepsilon_{2n})$, $t_1,t_2>0$, and $x\in\mathbb{R}^d$, applying Lemma \ref{L:Wick2} together with the definition of $\widetilde{S}_{n,\varepsilon}$, we obtain
\begin{align*}
  &\mathbb{E}\bigl[\widetilde{S}_{n,\varepsilon}(g_n(\cdot, t, x))\,\widetilde{S}_{n,\varepsilon'}(g_n(\cdot, t, x))\bigr] \\
  &= \int_{(\mathbb{R}^d)^{2n}} g_n(x_1,\dots,x_n, t, x)\,g_n(x_{n+1},\dots,x_{2n}, t, x)  \sum_{\mathcal{D}\in\Pi_{2n}} \prod_{(u,v)\in\mathcal{D}} \gamma_{\varepsilon_u+\varepsilon_v}^{(\tau)}(x_u-x_v) \, \ud x_1\cdots \ud x_{2n},
\end{align*}
where $\gamma_{\varepsilon_u+\varepsilon_v}^{(\tau)}$ is given in Lemma~\ref{L:Wick2}. Under Assumption~\ref{A:Main}, letting the mollification parameters tend to zero and applying the dominated convergence theorem as in the proof of Lemma~\ref{le:3.1} (iii), we obtain the covariance formula for the limiting integrals
\begin{equation}
    \label{E:cov-tildeS}
\begin{aligned}
&\mathbb{E}\bigl[\widetilde{S}_n(g_n(\cdot, t_1, x))\,\widetilde{S}_n(g_n(\cdot, t_2, x))\bigr] \\
  &= \sum_{\mathcal{D}\in\Pi_{2n}}
  \int_{(\mathbb{R}^d)^{2n}} g_n(x_1,\!\dots\!,x_n, t_1, x)\,g_n(x_{n+1},\!\dots\!,x_{2n}, t_2, x)
  \prod_{(u,v)\in\mathcal{D}} \gamma^{(\tau)}(x_u-x_v)\, \ud x, 
\end{aligned}
\end{equation}
where
$$
  \gamma^{(\tau)}(x_u-x_v) =
  \begin{cases}
    \gamma(x_u-x_v), & \text{if } u,v\in I=\{1,\dots,n\} \text{ or } u,v\in J=\{n+1,\dots,2n\},\\[4pt]
    e^{2\tau}\gamma(x_u-x_v), & \text{otherwise}.
  \end{cases}
$$

Comparing \eqref{E:cov-tildeS} with the analogous formula \eqref{E:Sf2n} for $S_n$, the only difference is that every cross pairing between $I$ and $J$ carries an extra factor $e^{2\tau}$. Under Assumption~\ref{A:Main}, all $\mathcal{L}^2$ estimates proved for $S_n$ in Sections \ref{se:technical} and \ref{se:existence} remain to hold for $\widetilde{S}_n$ after properly enlarging the generic constants. Consequently, by the same argument as in the proof of Theorem \ref{th:strat-intablity}, the mollified approximations $\widetilde{S}_{n,\varepsilon}(g_n(\cdot, t, x))$ form a Cauchy sequence in $\mathcal{L}^2$. Therefore, the limit $\widetilde{S}_n(g_n(\cdot, t, x))$ exists in $\mathcal{L}^2$ and is well-defined as a multiple Stratonovich integral in the sense of Definition \ref{D:multiple Stratonovich integration}. Moreover, the series $\widetilde{u}(t,x)$ in~\eqref{e:u2} converges in $\mathcal{L}^2$. 

Now we are ready to proof the Hypercontractivity property:

\begin{proposition}[Hypercontractivity property]\label{prop:Hypercon}
  Let $u(t,x)$ and $\widetilde{u}(t,x)$ be defined as in \eqref{e:expan-utx} and~\eqref{e:u2}, respectively. 
  Under Assumption~\ref{A:Main}, for any $\tau\ge0$, set $p =e^{2\tau}+1$. Then we have
  \begin{equation}\label{e:hypercontractivity}
      \|u(t,x)\|_p \leq\|\widetilde{u}(t,x)\|_2.
  \end{equation}
\end{proposition}

\begin{remark}\label{rem:hyper-contract0}
Hypercontractivity property was first proved in \cite{Huang2017Large} for stochastic heat equations (the parabolic Anderson model); see also \cite{Le2016Aremark}. It has since been further investigated and extended in \cite{Chen2018Temporal, bcc22, lswz26}.  We emphasize that all of the aforementioned works concern either solutions understood in the Skorohod sense (see \cite{Huang2017Large,Le2016Aremark,bcc22,lswz26}) or Stratonovich stochastic heat equations that admit a Feynman--Kac representation (see \cite{Chen2018Temporal}). The techniques developed in these works therefore do not readily extend to~\eqref{e:sfde} in the Stratonovich setting; see also Remark~\ref{rem:hyper-contrac}.

\end{remark}

\begin{proof}

Let $W^{\prime}$ be an independent copy of $W$. For any $F \in \cL^2(\Omega)$, there is a measurable mapping $\psi_F$ from $\mathbb{R}^{\mathcal{H}}$ to $\mathbb{R}$ such that $F=\psi_F ( W)$.
Let $\left\{T_\tau, \tau \geq 0\right\}$ be the Ornstein-Uhlenbeck semigroup associated with $W$, which  has the following representation (Mehler's formula \cite[eq.~(1.67)]{Nualart2006TheMalliavin})
\begin{equation}\label{e:Mehler}
T_\tau(F)=\mathbb{E}^{\prime}\left[\psi_F\left(e^{-\tau} W+\sqrt{1-e^{-2 \tau}} W^{\prime}\right)\right],
\end{equation}
where $\mathbb{E}^{\prime}$ denotes the expectation with respect to $W^{\prime}$. The following hypercontractivity property (see, e.g., \cite[Theorem~1.4.1] {Nualart2006TheMalliavin}) holds:
\begin{equation}
    \label{e:hypercontractivity'}
\left\|T_\tau F\right\|_p \leq\|F\|_q,
\end{equation}
 where  $q \in(1, \infty)$ and $p=1+e^{2 \tau}(q-1)$. 

 For $F_n:=\widetilde{S}_n
\left(g_n(\cdot, t, x)\right)$ given in \eqref{e:tilde-Sn},  by Mehler's formula \eqref{e:Mehler} we have, recalling that $\rho=\sqrt{e^{2 \tau}-1}$,
\begin{align}
T_\tau(F_n)&=\int_{\left(\mathbb{R}^d\right)^n}	g_n\left(x_1, \cdots, x_n, t, x\right)\mathbb{E}^{\prime} \mathbb{\widetilde E} \prod_{k=1}^n\left( W\left(\ud x_k\right)+\rho W^{\prime} \left(\ud x_k\right) + \iota\rho \widetilde{W}\left(\ud x_k\right)\right). \nonumber
\end{align}

By an argument analogous to Theorem \ref{T:laplace Sn}, we obtain
\begin{align}
&\int_0^{\infty}e^{-\lambda t}T_\tau(F_n)\ud t\nonumber\\
& =\frac{1}{n!} \lambda^{b-nr-1}\left(\frac{2}{\nu}\right)^{n+1} \int_0^{\infty}e^{-2\lambda^b t/\nu} \mathbf{E}_x\mathbb{E}^{\prime} \mathbb{\widetilde E} \left[\int_0^{t}\left(\dot{W}\left(X(s)\right)+\rho\dot{W}^{\prime} \left(X(s)\right)+\iota\rho\dot{\widetilde{W}}\left(X(s)\right) \right) \ud s\right]^{n} \ud t.\nonumber
\end{align}

 Conditioned on the stable process $X$, the random variables 
$\int_0^{t}\dot{W}^{\prime}\left(X(s)\right)\ud s$ and $\int_0^{t}\dot{\widetilde{W}}\left(X(s)\right)\ud s$ 
are independent centered Gaussian with the same variance, by Lemma~\ref{L:Wdefined} and the independence of $W^{\prime}$ and $\widetilde{W}$. Hence,
\[
\rho\int_0^{t}\dot{W}^{\prime}\left(X(s)\right)\ud s + \iota\rho\int_0^{t}\dot{\widetilde{W}}\left(X(s)\right)\ud s
\]
is a centered complex  Gaussian random variable whose real and imaginary parts are independent and have the variance $\rho^2\int_0^{t}\!\int_0^{t}\gamma(X(s)-X(r))\ud s\ud r$. Consequently,
\[
\mathbb{E}^{\prime}\widetilde{\mathbb{E}}\left[\rho\int_0^{t}\dot{W}^{\prime}\left(X(s)\right)\ud s + \iota\rho\int_0^{t}\dot{\widetilde{W}}\left(X(s)\right)\ud s\right]^m = 0
\]
for every positive integer $m$. Expanding the $n$-th power via the binomial theorem, all terms containing a factor from $W^{\prime}$ or $\widetilde{W}$ have zero conditional expectation by the above observation. Hence,
\begin{align*}
	\mathbf{E}_x\mathbb{E}^{\prime}\widetilde{\mathbb{E}}\left[\int_0^{t}\left(\dot{W}\left(X(s)\right)+\rho\dot{W}^{\prime} \left(X(s)\right)+\iota\rho\dot{\widetilde{W}}\left(X(s)\right) \right) \ud s\right]^{n}
	= \mathbf{E}_x\left[\int_0^{t} \dot{W}\left(X(s)\right)\ud s\right]^{n}.
\end{align*}
Substituting this back, we obtain
\begin{align}
	\int_0^{\infty}e^{-\lambda t}T_\tau(F_n)\ud t
	&=\frac{1}{n!} \lambda^{b-nr-1}\left(\frac{2}{\nu}\right)^{n+1} \int_0^{\infty}e^{-2\lambda^b t/\nu} \mathbf{E}_x\left[\int_0^{t} \dot{W}\left(X(s)\right)\ud s\right]^{n} \ud t. \nonumber
\end{align}


Due to the uniqueness theorem for Laplace transforms, we get $T_\tau(F_n)=S_n\left(g_n(\cdot, t, x)\right)$, a.s. Let $F=\sum_{n=0}^{\infty}\theta^{\frac{n}{2}}F_n$,  by the hypercontractivity \eqref{e:hypercontractivity'} with $q=2$ and $p=e^{2\tau}+1$,  we have
\begin{align*}
\mathbb{E}|u(t,x)|^p =\bigg\| \sum_{n=0}^{\infty}\theta^{\frac{n}{2}}S_n\left(g_n(\cdot, t, x)\right) \bigg\|_p ^p=\left\|T_\tau F\right\|_p ^p\leq\left(\mathbb{E}|\widetilde{u}(t,x)|^2 \right)^{p/2}. 
\end{align*}
This completes the proof.
\end{proof}

In order to obtain an upper bound for the $p$-th moment of $u(t,x)$ by the hypercontractivity \eqref{e:hypercontractivity}, we now investigate the asymptotic growth of the second moment of $\widetilde{u}(t, x)$.  For this purpose, we first introduce the following large deviation lemma.

\begin{lemma} \label{LDP}
	Under Assumptions \ref{A:Main} and \ref{B:Main}, let $X=\{X(t)\}_{t\ge 0}$ be a symmetric $\alpha$-stable process.
	Then we have
\begin{align}\label{e:asymp-n-moment}
	&\lim _{n \rightarrow \infty} \frac{1}{n}\log (n!)^{-\frac{\alpha}{a}}	\mathbf{E}_0\left[\int_0^1 \int_0^1 \gamma(X(s)-X(r)) \ud s \ud r\right]^n 
	=  \log\left\{4\left(\frac{a \mathcal{M}_a}{2a-\alpha}\right)^{\frac{2a-\alpha}{a}}\right\},
\end{align}
	where $\mathcal{M}_a$ is defined in~\eqref{e:de-Ma}.
\end{lemma}

\begin{proof}
	Define
	$$
	\mathcal{E}_a =  \sup_{g \in \mathcal{F}_{a,d}} \left\{ \int_{\mathbb{R}^d \times \mathbb{R}^d} \gamma(x-y) g^2(x) g^2(y) \, \ud x \ud y - \int_{\mathbb{R}^d} |\xi|^a |\widehat{g}(\xi)|^2 \, \ud\xi \right\},
	$$
	where $\mathcal{F}_{a,d}$ is given in \eqref{e:cFad}.
	By \cite[Lemma 7.1]{Chen2018Temporal}, under Assumption~\ref{A:Main}, we have $\mathcal{E}_a < \infty$.
	
	By  \cite[Theorem 1.1]{Chen2018Temporal} (with the choices $\rho=0$, $p=1$, $\beta_0=0$, and replacing the kernel $|\cdot|^{-\beta}$ by $\gamma(\cdot)$), it holds that, for any $\theta\geq0$,
	\begin{align} 
		& \lim _{t \rightarrow \infty} \frac{1}{t} \log \mathbf{E}_0 \exp \left\{\frac{\theta}{t}\int_0^t \int_0^t \gamma\left(X_s-X_r\right) \ud s \ud r\right\}= \theta^{\frac{a}{a-\alpha}}\mathcal{E}_a.
		\nonumber
	\end{align} 	
	Let $Y_t=t^{-\frac{\alpha}{a}}\int_0^1\int_0^1 \gamma\left(X_{s}-X_{r}\right) dsdr.$ By G\"artner-Ellis theorem on $\mathbb{R}_+$(see, e.g., Theorem 1.2.4 in \cite{Chenbook}), $Y_t$ satisfies a large deviation principle on $\mathbb{R}_+$ with the rate function
\begin{equation}
I(\lambda)=\sup_{\theta\geq0}\left\{\lambda \theta-\theta^{\frac{a}{a-\alpha}}\mathcal{E}\right\} = \frac{\alpha}{a}\left(\frac{a-\alpha}{a \mathcal{E}_a}\right)^{\frac{a-\alpha}{\alpha}} \lambda^{\frac{a}{\alpha}}.\nonumber
\end{equation}
	Applying Varadhan's integral lemma (see, e.g., Theorem 1.1.6 in \cite{Chenbook}), we obtain
$$
\begin{aligned}
	& \lim _{t \rightarrow \infty} \frac{1}{t} \log \mathbf{E}_0 \exp \left\{t \log \left(t^{-\frac{\alpha}{a}}\int_0^1\int_0^1 \gamma\left(X_{s}-X_{r}\right) \ud s \ud r\right)\right\} \\
	& =\sup _{\lambda>0}\left\{\log \lambda-\frac{\alpha}{a}\left(\frac{a-\alpha}{a \mathcal{E}_\delta}\right)^{\frac{a-\alpha}{\alpha}} \lambda^{\frac{a}{\alpha}}\right\}=-\frac{\alpha}{a}+\log \left(\frac{a \mathcal{E}}{a-\alpha}\right)^{\frac{a-\alpha}{a}}.
\end{aligned}
$$
Letting $t=n$, we get
\begin{equation}
	\lim _{n \rightarrow \infty} \frac{1}{n} \log n^{-\frac{\alpha n}{a}} \mathbf{E}_0\left[\int_0^1\int_0^1 \gamma\left(X_{s}-X_{r}\right) \ud s \ud r\right]^n=-\frac{\alpha}{a}+\log \left(\frac{a \mathcal{E}}{a-\alpha}\right)^{\frac{a-\alpha}{a}}. \label{eq:nlmintforEX}
\end{equation}

	By an argument similar to the proof of Lemma A.2 in \cite{chsx2015}, which is applicable thanks to Assumption~\ref{B:Main}, we obtain
	\begin{equation}
		\mathcal{E}_a = \frac{a-\alpha}{a}  2^{\frac{\alpha}{a-\alpha}} \left( \frac{2a}{2a-\alpha} \right)^{\frac{2a-\alpha}{a-\alpha}}  \mathcal{M}_a^{\frac{2a-\alpha}{a-\alpha}}. \label{eq:varscaling}
	\end{equation}
	Combining \eqref{eq:nlmintforEX} and \eqref{eq:varscaling}, together with Stirling's formula, we obtain the desired \eqref{e:asymp-n-moment}.
\end{proof}

\begin{proposition}\label{L:upper2lemma}
 Recall that  $\mathcal M_a$ and $\widetilde{u}(t, x)$ are  defined in \eqref{e:de-Ma}  and \eqref{e:u2} respectively.
Under Assumptions \ref{A:Main}, \ref{B:Main} and \ref{A:G}, we have
 \begin{align*}
    &\limsup_{n \rightarrow \infty} \frac{1}{n} \log \left(n!\right)^{2(b+r)-b \alpha / a-1} \theta^{n} \sum_{l_1+l_2= 2n} \mathbb{E} \prod_{j=1}^2 \widetilde{S}_{l_j}\left(g_{l_j}\left(\cdot, 1, 0\right)\right)\\
    &\leq\log  \left\{\left(\frac{2}{2(b+r)-b \alpha / a}\right)^{[2(b+r)-b \alpha / a] }\left(\frac{\nu}{2}\right)^{-\frac{\alpha}{a}} \theta p \mathcal{M}_a^{\frac{2a-\alpha}{a}}\right\}.
\end{align*}
\end{proposition} 

 \begin{proof}
 {\bf Step 1.}
     We first prove the following inequality 
\begin{equation}\label{e:upperaim1}
 \begin{aligned}
	& \limsup _{n \rightarrow \infty} \frac{1}{n} \log \frac{1}{n!}\int_0^{\infty} \int_0^{\infty} \ud t_1 \ud t_2 \: e^{- t_1- t_2}  \theta^{n} \sum_{l_1+l_2= 2n}\mathbb{E} \prod_{j=1}^2 \widetilde{S}_{l_j}\left(g_{l_j}\left(\cdot, t_j, 0\right)\right)  \\
    &\qquad \leq \log\left\{  \left(\frac{\nu}{2}\right)^{-\frac{\alpha}{a}} \theta p \mathcal{M}_{a}^{\frac{2a-\alpha}{a}}\right\}.
	\end{aligned}
    \end{equation}
Note that 
\begin{equation}\label{e:upper2}
\begin{aligned}
		& \int_0^{\infty} \int_0^{\infty} \ud t_1 \ud t_2 \: e^{-t_1- t_2}  \theta^{n} \sum_{l_1+l_2= 2n} \mathbb{E} \prod_{j=1}^2 \widetilde{S}_{l_j}\left(g_{l_j}\left(\cdot, t_j, 0\right)\right)  \\
		&= \left(\frac{2}{\nu}\right)^{2n+2} \theta^{n}  \int_0^{\infty} \int_0^{\infty} \ud t_1 \ud t_2   e^{-2 t_1/\nu -2 t_2/\nu } \sum_{l_1+l_2= 2n} \frac{1}{l_1!} \frac{1}{l_2!}\\
& \qquad \times \mathbf{E}_x \mathbb E\bigg\{\mathbb{\widetilde E} \left[\int_0^{t_1} \left(e^{\tau}\dot{W}(X_1(s))+\iota\rho\dot{\widetilde{W}}(X_1(s))\right) \ud s\right]^{l_1} \mathbb{\widetilde E} \left[\int_0^{t_2} \left(e^{\tau}\dot{W}(X_2(s))+\iota\rho\dot{\widetilde{W}}(X_2(s))\right) \ud s\right]^{l_2}\bigg\} \\
		 &= \frac{\theta^{n} }{2^n n!}\left(\frac{2}{\nu}\right)^{2n+2}  \int_0^{\infty} \int_0^{\infty} \ud t_1 \ud t_2 \: e^{-2t_1/\nu -2 t_2/\nu }\\
& \qquad \times\mathbf{E}_0\left[e^{2\tau}\sum_{j,k=1}^{2}\int_0^{t_j} \int_0^{t_k} \gamma\left(X_j(s)-X_k(r)\right) \ud s \ud r-\rho^2\sum_{i=1}^{2}\int_0^{t_i} \int_0^{t_i} \gamma\left(X_i(s)-X_i(r)\right) \ud s \ud r\right]^n,
	\end{aligned}
    \end{equation}
where the two equalities follow from  analogous arguments leading to Theorem~\ref{T:laplace Sn} and Lemma~\ref{l:characteristic function}, respectively.    

Recalling that $\rho^2=e^{2\tau}-1, p=e^{2\tau}+1$, and using the symmetry of $\gamma$, we obtain
$$
\begin{aligned}
&e^{2\tau}\sum_{j,k=1}^2
\int_0^{t_j}\int_0^{t_k}
\gamma\bigl(X_j(s)-X_k(r)\bigr)\,\ud s\,\ud r
-\rho^2\sum_{i=1}^2
\int_0^{t_i}\int_0^{t_i}
\gamma\bigl(X_i(s)-X_i(r)\bigr)\,\ud s\,\ud r\\
&\quad=
\frac p2\sum_{j,k=1}^2
\int_0^{t_j}\int_0^{t_k}
\gamma\bigl(X_j(s)-X_k(r)\bigr)\,\ud s\,\ud r\\
&\qquad-
\frac{\rho^2}{2}
\left[
\sum_{i=1}^2
\int_0^{t_i}\int_0^{t_i}
\gamma\bigl(X_i(s)-X_i(r)\bigr)\,\ud s\,\ud r
-2\int_0^{t_1}\int_0^{t_2}
\gamma\bigl(X_1(s)-X_2(r)\bigr)\,\ud s\,\ud r
\right].
\end{aligned}
$$

Using the Fourier transform, we have
$$
\begin{aligned}
&\sum_{i=1}^2
\int_0^{t_i}\int_0^{t_i}
\gamma\bigl(X_i(s)-X_i(r)\bigr)\,\ud s\,\ud r
-2\int_0^{t_1}\int_0^{t_2}
\gamma\bigl(X_1(s)-X_2(r)\bigr)\,\ud s\,\ud r\\
&=
\int_{\mathbb R^d}
\left|
\int_0^{t_1}e^{\iota\xi\cdot X_1(s)}\,\ud s
-\int_0^{t_2}e^{\iota\xi\cdot X_2(s)}\,\ud s
\right|^2\mu(\ud\xi)
\geq0.
\end{aligned}
$$

Hence, substituting the above estimate into \eqref{e:upper2}, we obtain the following estimate
\begin{equation}\label{e:upper3}
\begin{aligned}
&\int_0^\infty\int_0^\infty e^{-t_1-t_2}\theta^n
\sum_{l_1+l_2=2n}\mathbb E\prod_{j=1}^2
\widetilde S_{l_j}\bigl(g_{l_j}(\cdot,t_j,0)\bigr)\,
\ud t_1\,\ud t_2\\
		 &\leq \frac{(\theta p)^{n} }{4^n n!}\left(\frac{2}{\nu}\right)^{2n+2}  \int_0^{\infty} \int_0^{\infty} \ud t_1 \ud t_2  e^{-2t_1/\nu -2 t_2/\nu }\mathbf{E}_0\left[\sum_{j,k=1}^{2}\int_0^{t_j} \int_0^{t_k} \gamma\left(X_j(s)-X_k(r)\right) \ud s \ud r\right]^n. 
\end{aligned}
\end{equation}
 
Using the Fourier transform and Jensen inequality, we have the following estimate
  \begin{align*}
	&\sum_{j,k=1}^{2}\int_0^{t_j} \int_0^{t_k} \gamma\left(X_j(s)-X_k(r)\right) \ud s \ud r 
	=\int_{\mathbb{R}^d} \bigg|\sum_{j=1}^{2}\int_0^{t_j}e^{\iota\xi\cdot X_j(s)} \ud s\bigg|^2\mu(\ud\xi),\\
	&=\left(t_1+t_2\right)^2\int_{\mathbb{R}^d} \bigg|\sum_{j=1}^{2}\frac{t_j}{t_1+t_2}\cdot\frac{1}{t_j}\int_0^{t_j}e^{\iota\xi\cdot X_j(s)} \ud s\bigg|^2\mu(\ud\xi)
	\leq\left(t_1+t_2\right)\sum_{j=1}^{2}t_j\int_{\mathbb{R}^d} \bigg|\frac{1}{t_j}\int_0^{t_j}e^{\iota\xi\cdot X_j(s)} \ud s\bigg|^2\mu(\ud\xi)\\
	&=	\left(t_1+t_2\right)\sum_{j=1}^{2}t_i^{\frac{a-\alpha}{a}}\int_0^{1} \int_0^{1} \gamma\left(X_j(s)-X_j(r)\right) \ud s \ud r,
\end{align*}
where the final inequality follows from \begin{equation*} 
	    \int_0^{t_i} \int_0^{t_i} \gamma\left(X_i(s)-X_i(r)\right) \ud s \ud r \stackrel{d}{=} t_i^{\frac{2a-\alpha}{a}} \int_0^1 \int_0^1 \gamma\left(X_i(s)-X_i(r)\right) \ud s \ud r, \quad i=1, 2.
	\end{equation*}
This together with \eqref{e:upper3}  yields that, for any $\delta>0$, 
\begin{equation}\label{eq:highmomentS 1}
\begin{aligned}
&\int_0^\infty\int_0^\infty \ud t_1 \ud t_2 e^{-t_1-t_2}\theta^n
\sum_{l_1+l_2=2n}\mathbb E\prod_{j=1}^2
\widetilde S_{l_j}\bigl(g_{l_j}(\cdot,t_j,0)\bigr)  \\
		\leq & \frac{(\theta p)^{n} }{4^n n!}\left(\frac{2}{\nu}\right)^{2n+2} \int_0^\infty\int_0^\infty\ud t_1\ud t_2   e^{-2t_1/\nu -2 t_2/\nu } \left(t_1+t_2\right)^n \\
		& \times \mathbf{E}_0 \left[\sum_{j=1}^{m} t_i^{\frac{a-\alpha}{a}}  \int_0^1 \int_0^1 \gamma\left(X_j(s)-X_j(r)\right) \ud s \ud r \right]^n \\
		= & \frac{(\theta p)^{n} }{4^n }\left(\frac{2}{\nu}\right)^{2n+2}  \sum_{l_1+l_2=n} \int_0^\infty\int_0^\infty\ud t_1\ud t_2   e^{-2t_1/\nu -2 t_2/\nu } \left(t_1+t_2\right)^n\prod_{j=1}^2 t_j^{\frac{a-\alpha}{a} l_j} \\
        &\times \frac{1}{l_{1}!l_{2}!}\prod_{j=1}^2 \mathbf{E}_0\left[\int_0^1 \int_0^1 \gamma(X_j(s)-X_j(r)) \ud s \ud r\right]^{l_j}, 
\end{aligned}
\end{equation}
where the last step follows from Newton's multi-nominal formula.

By Lemma~\ref{LDP} and the positivity of $\gamma(\cdot)$, for any $\delta>0$ there exists a constant $C_\delta>0$ such that 
\begin{align}
	\mathbf{E}_0\left[\int_0^1 \int_0^1 \gamma(X(s)-X(r)) \ud s \ud r\right]^n \leq C_\delta(n!)^{\frac{\alpha}{a}} \bigg[(1+\delta) 4\left(\frac{a \mathcal{M}_a}{2a-\alpha}\right)^{\frac{2a-\alpha}{a}} \bigg]^n, \quad n=1,2, \cdots. \label{E:Local time estimate}
\end{align}
Plugging  \eqref{E:Local time estimate} into \eqref{eq:highmomentS 1}, we have
\begin{equation}\label{e:upper5}
\begin{aligned}
&\int_0^\infty\int_0^\infty \ud t_1 \ud t_2 e^{-t_1-t_2}\theta^n
\sum_{l_1+l_2=2n}\mathbb E\prod_{j=1}^2
\widetilde S_{l_j}\bigl(g_{l_j}(\cdot,t_j,0)\bigr)  \\
		\leq & C_\delta \frac{(\theta p)^{n} }{n!}\left(\frac{2}{\nu}\right)^{2n+2} \bigg[(1+\delta) \left(\frac{a \mathcal{M}_a}{2a-\alpha}\right)^{\frac{2a-\alpha}{a}} \bigg]^n\\
        &\times \sum_{l_1+l_2=n}
	\prod_{j=1}^2(l_j!)^{-\frac{a-\alpha}{a}}\int_0^\infty\int_0^\infty (t_1+t_2)^n e^{-2t_1/\nu -2 t_2/\nu } \prod_{j=1}^2t_j^{\frac{a-\alpha}{a}l_j}	\ud t_1\ud t_2. 
\end{aligned}
\end{equation}

We now estimate the following term
\begin{equation*}
	\sum_{l_1+l_2=n}
	\prod_{j=1}^2(l_j!)^{-\frac{a-\alpha}{a}}\int_0^\infty\int_0^\infty (t_1+t_2)^n e^{-2t_1/\nu -2 t_2/\nu }\prod_{j=1}^2t_j^{\frac{a-\alpha}{a}l_j}	\ud t_1\ud t_2.
\end{equation*}
Applying the change of variables
\begin{equation*}
	t_1=rs,\qquad t_2=r(1-s),	\qquad r>0,\quad 0<s<1,
\end{equation*}
and using $l_1+l_2=n$, we obtain
\begin{equation}
    \label{eq:beta-gamma-m2}
\begin{aligned}
	&\int_0^\infty\int_0^\infty 	(t_1+t_2)^n e^{-2t_1/\nu -2 t_2/\nu }\prod_{j=1}^2t_j^{\frac{a-\alpha}{a}l_j}\ud t_1\ud t_2\\
	&=\int_0^\infty r^{\frac{2a-\alpha}{a}n+1}e^{-2r/\nu}\ud r\int_0^1s^{\frac{a-\alpha}{a}l_1}(1-s)^{\frac{a-\alpha}{a}l_2}\,\ud s	\\
	&=\left(\frac{\nu}{2}\right)^{\frac{2a-\alpha}{a}n+2}
	\frac{\Gamma\left(\frac{2a-\alpha}{a}n+2\right)	}{\Gamma\left(\frac{a-\alpha}{a}n+2\right)}	\prod_{j=1}^2\Gamma\left(\frac{a-\alpha}{a}l_j+1\right).
\end{aligned}
\end{equation}
Hence,
\begin{align*}
	&\sum_{l_1+l_2=n}\prod_{j=1}^2(l_j!)^{-\frac{a-\alpha}{a}}\int_0^\infty\int_0^\infty 	(t_1+t_2)^n e^{-2t_1/\nu -2 t_2/\nu }\prod_{j=1}^2t_j^{\frac{a-\alpha}{a}l_j}
	\ud t_1\ud t_2\\
	&=	\left(\frac{\nu}{2}\right)^{\frac{2a-\alpha}{a}n+2}
	\frac{\Gamma\left(\frac{2a-\alpha}{a}n+2\right)}{\Gamma\left(\frac{a-\alpha}{a}n+2\right)}\sum_{l_1+l_2=n}\prod_{j=1}^2\frac{\Gamma\left(\frac{a-\alpha}{a}l_j+1\right)}{(l_j!)^{\frac{a-\alpha}{a}}}.
\end{align*}
By Stirling's formula, we have
\begin{align*}
	\sum_{l_1+l_2=n}\prod_{j=1}^2\frac{\Gamma\left(\frac{a-\alpha}{a}l_j+1\right)}{(l_j!)^{\frac{a-\alpha}{a}}}
	&\leq C\left(\frac{a-\alpha}{a}\right)^{\frac{a-\alpha}{a}(l_1+l_2)}
	\sum_{l_1+l_2=n} \prod_{j=1}^2(l_j+1)^{\frac{\alpha}{2a}}\\
	&\leq C\left(\frac{a-\alpha}{a}\right)^{\frac{a-\alpha}{a}n}(n+1)^{1+\frac{\alpha}{a}}.
\end{align*}
Moreover, Stirling's formula also yields
\begin{equation*}
	\frac{\Gamma\left(\frac{2a-\alpha}{a}n+2\right)}{\Gamma\left(\frac{a-\alpha}{a}n+2\right)}\leq C n!\left(\frac{2a-\alpha}{a}\right)^{\frac{2a-\alpha}{a}n}\left(\frac{a-\alpha}{a}\right)^{-\frac{a-\alpha}{a}n}.
\end{equation*}
Combining this with \eqref{eq:beta-gamma-m2}, we obtain
\begin{equation}
    \label{eq:m2-beta-bound}
\begin{aligned}	&\sum_{l_1+l_2=n}\prod_{j=1}^2(l_j!)^{-\frac{a-\alpha}{a}}
	\int_0^\infty\int_0^\infty 	(t_1+t_2)^n e^{-2t_1/\nu -2 t_2/\nu }
	\prod_{j=1}^2t_j^{\frac{a-\alpha}{a}l_j}\ud t_1\ud t_2	\\
	&\leq C(n+1)^{1+\frac{\alpha}{a}}n!\left(\frac{\nu}{2}\right)^{\frac{2a-\alpha}{a}n+2}\left(\frac{2a-\alpha}{a}\right)^{\frac{2a-\alpha}{a}n}.
\end{aligned}
\end{equation}

Combining \eqref{e:upper5} and \eqref{eq:m2-beta-bound}, we thus conclude that
 \begin{align*}
	& \limsup _{n \rightarrow \infty} \frac{1}{n} \log \frac{1}{n!}\int_0^{\infty} \int_0^{\infty} \ud t_1 \ud t_2  e^{- t_1- t_2}  \theta^n\sum_{l_1+l_2= 2n}\mathbb{E} \prod_{j=1}^2 \widetilde{S}_{l_j}\left(g_{l_j}\left(\cdot, t_j, 0\right)\right)  \notag\\
	& \leq \log  \left\{\left(\frac{\nu}{2}\right)^{-\frac{\alpha}{a}} (1+\delta)\theta p\mathcal{M}_{a}^{\frac{2a-\alpha}{a}}\right\},
	\end{align*}
which completes the proof of \eqref{e:upperaim1} by  letting  $\delta \to 0^{+}.$

{\bf Step 2.} 
We now complete the proof of the lemma. First, appealing to the scaling properties~\eqref{E:scalef} and \eqref{e:scal-G}, we have
\begin{equation} \label{E:scal S}
    \sum_{l_1+\cdots+l_p=2 n} \mathbb{E} \prod_{j=1}^p \widetilde S_{l_j}\left(g_{l_j}(\cdot, t, 0)\right)=t^{[2(b+r)-b \alpha / a] n}\sum_{l_1+\cdots+l_p=2 n} \mathbb{E} \prod_{j=1}^p \widetilde S_{l_j}\left(g_{l_j}(\cdot, 1,0)\right), \quad \forall t>0.
\end{equation} 
From  the monotonicity of $g_n(\cdot, t, 0)$ in $t$ and the above scaling property, it follows that
\begin{align}
	\sum_{l_1+l_2= 2n}\mathbb{E} \prod_{j=1}^2 \widetilde{S}_{l_j}\left(g_{l_j}\left(\cdot, t_j, 0\right)\right) &\geq 	\sum_{l_1+l_2= 2n} \mathbb{E} \prod_{j=1}^2 \widetilde{S}_{l_j} \left(g_{l_j}(\cdot, \min _{1 \leq j \leq 2} t_j, 0) \right)\nonumber \\
	&=\bigg\{\sum_{l_1+l_2= 2n}\mathbb{E} \prod_{j=1}^2 \widetilde{S}_{l_j}\left(g_{l_j}\left(\cdot, 1, 0\right)\right)\bigg\} \left(\min _{1 \leq j \leq 2} t_j\right)^{[2(b+r)-b \alpha / a] n}. \nonumber
\end{align}
Thus, we get 
\begin{align*}
	&  \int_0^{\infty} \int_0^{\infty} \ud t_1 \ud t_2 \: e^{- t_1- t_2}  \sum_{l_1+l_2= 2n} \mathbb{E} \prod_{j=1}^2 \widetilde{S}_{l_j}\left(g_{l_j}\left(\cdot, t_j, 0\right)\right)  \\
	&\geq \bigg\{\sum_{l_1+l_2= 2n}\mathbb{E} \prod_{j=1}^2 \widetilde{S}_{l_j}\left(g_{l_j}\left(\cdot, 1, 0\right)\right)\bigg\}  \int_0^{\infty} \int_0^{\infty} \ud t_1 \ud t_2 \: e^{- t_1- t_2}\left(\min _{1 \leq j \leq 2} t_j\right)^{[2(b+r)-b \alpha / a] n} \\
    &= \bigg\{\sum_{l_1+l_2= 2n}\mathbb{E} \prod_{j=1}^2 \widetilde{S}_{l_j}\left(g_{l_j}\left(\cdot, 1, 0\right)\right)\bigg\} 2^{-[2(b+r)-b \alpha / a] n} \Gamma(1+[2(b+r)-b \alpha / a] n),
\end{align*}
where the equality is due to \eqref{E:existence proof6}.
This yields
$$
\begin{aligned}
	& \theta^n\sum_{l_1+l_2= 2n}\mathbb{E} \prod_{j=1}^2 \widetilde{S}_{l_j}\left(g_{l_j}\left(\cdot, 1, 0\right)\right) \\
 &\leq	  \frac{2^{[2(b+r)-b \alpha / a] n}}{\Gamma(1+[2(b+r)-b \alpha / a] n)} \int_0^{\infty} \int_0^{\infty} \ud t_1 \ud t_2  e^{- t_1- t_2} \theta^n \sum_{l_1+l_2= 2n}\mathbb{E} \prod_{j=1}^2 \widetilde{S}_{l_j}\left(g_{l_j}\left(\cdot, t_j, 0\right)\right) \\
&\leq	 C  n^{\frac{|1-2(b+r)+b \alpha / a|}{2}} \frac{ 2^{[2(b+r)-b \alpha / a] n}\left(\frac{\nu}{2}\right)^{-\frac{\alpha}{a}n}\left(\theta p \mathcal{M}_a^{\frac{2a-\alpha}{a}}\right)^{n}}{\left(n!\right)^{2(b+r)-b \alpha / a-1}[2(b+r)-b \alpha / a]^{[2(b+r)-b \alpha / a]n}},
\end{aligned}
$$
where the second step follows from \eqref{e:upperaim1} and
\begin{equation*}
    \Gamma(1+c_0n)  \ge C  n^{-\frac{|1-c_0|}{2}}  c_0^{c_0n}  (n!)^{c_0}, \quad \text{for } c_0>0,
\end{equation*}
which can be deduced from Stirling formula. This yields the desired result and the proof is completed. 
 \end{proof}

 \begin{theorem}  \label{T:upperbound}
	Let Assumptions \ref{A:Main}, \ref{B:Main} and \ref{A:G} hold. Recall that $u(t, x)$ is the  mild Stratonovich solution to equation \eqref{e:sfde}. 
		 For every real $p\ge2$,
			\[
			\lim _{t \rightarrow \infty} t^{-\beta} \log \mathbb{E}\left(|u(t, x)|^p\right)
			\le \frac{p^\beta}{2} \mathbf M(a,b,r,\theta, \nu, \alpha).
			\]
			For every fixed $t>0$,
			\[
			\lim _{p \rightarrow \infty} p^{-\beta} \log \mathbb{E}\left(|u(t, x)|^p\right)
			\le \frac{t^\beta}{2} \mathbf M(a,b,r,\theta, \nu, \alpha).
			\]
     \end{theorem}
     
\begin{proof}
   By Proposition~\ref{prop:Hypercon}, we obtain for any $p\ge 2$
\begin{equation}
    \label{eq:u2-expansion}
\begin{aligned}
& \log \mathbb{E}\left[|u(t, x)|^p\right] 
\leq \frac{p}{2}\log \mathbb{E} [\widetilde{u}^2(t, x)] \\
&\leq \frac{p}{2} \log \sum_{n=0}^{\infty} t^{[2(b+r)-b \alpha / a] n} \bigg(\theta^{n}\sum_{l_1+l_2= 2n} \mathbb{E} \prod_{j=1}^2 \widetilde{S}_{l_j}\left(g_{l_j}\left(\cdot, 1, 0\right)\right)\bigg),
\end{aligned}
\end{equation}
where the second inequality follows from \eqref{e:u2} and \eqref{E:scal S}. 
 Lemma \ref{L:upper2lemma} yields that for any $\delta>0$ there exists a constant $C_\delta>0$ such that for all $n\ge 1$,
\begin{equation}
     \label{eq:u2-expansion2}
 \begin{aligned}
     &\sum_{l_1+l_2= 2n} \mathbb{E} \prod_{j=1}^2 \widetilde{S}_{l_j}\left(g_{l_j}\left(\cdot, 1, 0\right)\right)\\
     &\leq C_{\delta} \frac{\left(1+\delta\right)^np^n}{\left(n!\right)^{2(b+r)-b \alpha / a-1}} \bigg[ \left(\frac{2}{2(b+r)-b \alpha / a}\right)^{2(b+r)-b \alpha / a }\left(\frac{\nu}{2}\right)^{-\frac{\alpha}{a}} \mathcal{M}_{a}^{\frac{2a-\alpha}{a}} \bigg]^n.
 \end{aligned}
 \end{equation}
Substituting  \eqref{eq:u2-expansion2} into \eqref{eq:u2-expansion} yields
 \begin{align}
     & \log \mathbb{E}\left[|u(t, x)|^p\right] 
\nonumber\\
   &\leq \frac{p}{2}\log C_{\delta}\sum_{n=0}^{\infty} \frac{\left(1+\delta\right)^np^nt^{[2(b+r)-b \alpha / a] n}}{\left(n!\right)^{2(b+r)-b \alpha / a-1}} \bigg[ \left(\frac{2}{2(b+r)-b \alpha / a}\right)^{2(b+r)-b \alpha / a }\left(\frac{\nu}{2}\right)^{-\frac{\alpha}{a}} \theta\mathcal{M}_a^{\frac{2a-\alpha}{a}} \bigg]^n.\nonumber
 \end{align}
To evaluate the limits as $t\to\infty$ and $p\to\infty$, we use the following
asymptotics of the Mittag–Leffler function (Lemma~A.3 in \cite{bcc22})
\begin{equation}\label{e:asym-ML}
    \lim_{h\to\infty} h^{-1/\kappa}\log\sum_{n=0}^{\infty}\frac{\vartheta^n h^n}{(n!)^{\kappa}}
= \kappa\,\vartheta^{1/\kappa}, \qquad \vartheta>0 .
\end{equation}

Consider first the limit $t\to\infty$ with $p$ fixed. Setting $h = t^{\,2(b+r)-b\alpha/a}$ and $\kappa = 2(b+r)-b\alpha/a-1$, we obtain
\begin{align}
&\lim _{t \rightarrow \infty} t^{-\beta} \log \mathbb{E}\left(|u(t, x)|^p\right)\nonumber\\
&\leq  \frac{p}{2}\lim _{t \rightarrow \infty} t^{-\beta} \log C_{\delta}\sum_{n=0}^{\infty} \frac{\left(1+\delta\right)^np^nt^{[2(b+r)-b \alpha / a] n}}{\left(n!\right)^{2(b+r)-b \alpha / a-1}} \bigg[ \left(\frac{2}{2(b+r)-b \alpha / a}\right)^{2(b+r)-b \alpha / a }\left(\frac{\nu}{2}\right)^{-\frac{\alpha}{a}} \theta\mathcal{M}_a^{\frac{2a-\alpha}{a}} \bigg]^n \nonumber\\
& \leq \frac{\left(1+\delta\right)^{\beta-1}p^\beta}{2}\left(\frac{2 a}{2 a(b+r)-b \alpha}\right)^\beta\left(2(b+r)-\frac{b \alpha}{a}-1\right) \left(\theta\left(\frac{\nu}{2}\right)^{-\frac{\alpha}{a}} \mathcal{M}_a^{\frac{2 a-\alpha}{a}}\right)^{\frac{a}{2 a(b+r)-b \alpha-a}}. \label{eq:t-limit}
\end{align}

	For the limit $p\to\infty$ with $t$ fixed, 
taking $h = p$ and the same $\kappa = 2(b+r)-b\alpha/a-1$, we get
\begin{align}
&\lim _{p \rightarrow \infty} p^{-\beta} \log \mathbb{E}\left(|u(t, x)|^p\right)\nonumber\\
&\leq  \frac{1}{2}\lim _{p\rightarrow \infty} p^{-\frac{1}{2(b+r)-b \alpha / a-1}} \log C_{\delta}\sum_{n=0}^{\infty} \frac{\left(1+\delta\right)^{n}p^nt^{[2(b+r)-b \alpha / a] n}}{\left(n!\right)^{2(b+r)-b \alpha / a-1}} \notag\\
&\qquad \qquad \qquad \times\bigg[ \left(\frac{2}{2(b+r)-b \alpha / a}\right)^{2(b+r)-b \alpha / a }\left(\frac{\nu}{2}\right)^{-\frac{\alpha}{a}} \theta\mathcal{M}_a^{\frac{2a-\alpha}{a}} \bigg]^n \nonumber\\
&\leq  \frac{\left(1+\delta\right)^{\beta-1}t^\beta}{2}\left(\frac{2 a}{2 a(b+r)-b \alpha}\right)^\beta\left(2(b+r)-\frac{b \alpha}{a}-1\right) \left(\theta \left(\frac{\nu}{2}\right)^{-\frac{\alpha}{a}} \mathcal{M}_a^{\frac{2 a-\alpha}{a}}\right)^{\frac{a}{2 a(b+r)-b \alpha-a}}.\label{eq:p-limit}
\end{align}
Letting $\delta\to0$ in \eqref{eq:t-limit} and \eqref{eq:p-limit} yields the desired estimates of Theorem~\ref{T:upperbound}.
\end{proof}

\begin{remark}\label{rem:hyper-contrac}
We believe that the proof of Theorem~\ref{T:upperbound} introduces a new approach to deriving real-moment asymptotics for Stratonovich solutions. 

In the Skorohod setting, \cite{Le2016Aremark} observed that the hypercontractivity of the Ornstein--Uhlenbeck semigroup allows integer-moment asymptotics to be extended to moments of every real order $p\ge 2$. The key mechanism is that, for $p>2$, hypercontractivity bounds the $p$-th moment of the solution by the second moment of a solution to an equation of the same type. This comparison relies on the Wick-product structure of Skorohod integrals and therefore does not extend to the Stratonovich regime. The only known result on real-moment asymptotics in the Stratonovich regime is \cite{Chen2018Temporal}, which treats the fractional stochastic heat equation using a Feynman--Kac formula specific to that equation and unavailable for general fractional diffusion equations in our setting.

In the present work, we construct an auxiliary Stratonovich solution $\widetilde{u}(t,x)$ (see~\eqref{e:u2}) whose Ornstein–Uhlenbeck transform coincides with $u(t,x)$. Hypercontractivity then allows us to bound the $p$-th moment of $u(t,x)$ by the second moment of $\widetilde{u}(t,x)$; see Proposition~\ref{prop:Hypercon}. We subsequently derive, in Proposition~\ref{L:upper2lemma}, an upper bound for the large-$n$ asymptotics of each Stratonovich chaos component of $\widetilde{u}(t,x)$. The proof uses the Laplace-transform technique developed in Section~\ref{sec:laplace}, which connects the fractional diffusion equation to the heat equation.

To the best of our knowledge, this is the first use of hypercontractivity directly at the level of the (Stratonovich) chaos expansion of the solution, without relying on the existence of a Feynman--Kac formula. The approach also applies to mild solutions of linear stochastic equations admitting series expansions, provided that their Green functions can be related to the fractional heat kernel through Laplace transforms. 
\end{remark}
By a similar argument as above, we can obtain the exact asymptotics of the first moment of $u(t,x).$

\begin{proof}[Proof of Theorem \ref{th:excat} (i)] \makeatletter\def\@currentlabelname{proof}\makeatother\label{proof:1}

From \eqref{e:expan-utx} ,\eqref{E:scalef} and \eqref{e:scal-G} it follows that
\begin{align}
	\mathbb{E}u(t,x)=	 \sum_{n=0}^{\infty} \theta^{n}\mathbb{E}S_{2n}\left(g_{2n}\left(\cdot, t, 0\right)\right)=\sum_{n=0}^{\infty} t^{[2(b+r)-b \alpha / a]n}\theta^{n}\mathbb{E} S_{2n}\left(g_{2n}\left(\cdot, 1, 0\right)\right). \label{eq:L1Sn1}
\end{align}
Using \eqref{E:scalef} and \eqref{e:scal-G} again, and then applying Theorem \ref{T:laplace Sn}, we obtain
\begin{align*}
	&\theta^{n}\mathbb{E} S_{2n}\left(g_{2n}\left(\cdot, 1, 0\right)\right) \nonumber\\
	&= \frac{\theta^{n}}{\Gamma(1+[2(b+r)-b \alpha / a] n)} \int_{0}^{\infty }e^{-t}\mathbb{E} S_{2n}\left(g_{2n}\left(\cdot,t, 0\right)\right)\ud t\nonumber\\
		&= \frac{\theta^{n} }{2^n n!}\left(\frac{2}{\nu}\right)^{2n+2}  \frac{1}{\Gamma(1+[2(b+r)-b \alpha / a] n)}  \int_0^{\infty} e^{-2t/\nu }
		\mathbf{E}_0\left[\int_0^{t} \int_0^{t} \gamma\left(X(s)-X(r)\right) \ud s \ud r\right]^n \ud t \nonumber\\
			&= \frac{\theta^{n} }{2^n n!}\left(\frac{2}{\nu}\right)^{\frac{\alpha}{a}} \frac{\Gamma(1+[(2a-\alpha) / a] n)}{\Gamma(1+[2(b+r)-b \alpha / a] n)} 
		\mathbf{E}_0\left[\int_0^{1} \int_0^{1} \gamma\left(X(s)-X(r)\right) \ud s \ud r\right]^n. 
\end{align*}

Lemma~\ref{LDP} together with Stirling's formula yields
\begin{align}
&\lim_{n\to\infty}\frac1n\log(n!)^{2(b+r)-\frac{b\alpha}{a}-1}	\theta^n\mathbb{E}S_{2n}\left(g_{2n}\left(\cdot,1,0\right)\right)\nonumber\\
	&=\log\left\{\frac{\theta}{2}\left(\frac{2}{\nu}\right)^{\frac{\alpha}{a}}\right\}
+\lim_{n\to\infty}\frac1n\log\left\{(n!)^{	2(b+r)-\frac{b\alpha}{a}	-\frac{2a-\alpha}{a}}\frac{\Gamma\left(1+\frac{2a-\alpha}{a}n\right)}{	\Gamma\left(1+\left[2(b+r)-\frac{b\alpha}{a}\right]n\right)}\right\}\nonumber\\
	&\quad+
	\lim_{n\to\infty}\frac1n
	\log(n!)^{-\frac{\alpha}{a}}\mathbf{E}_0\left[\int_0^1\int_0^1\gamma\left(X(s)-X(r)\right)	\ud s\ud r\right]^n	\nonumber\\
	&=\log\left\{	\frac{\theta}{2}\left(\frac{2}{\nu}\right)^{\frac{\alpha}{a}}\right\}
	+\frac{2a-\alpha}{a}\log\left(\frac{2a-\alpha}{a}\right)\nonumber\\
	&\quad-	\left(2(b+r)-\frac{b\alpha}{a}\right)\log\left(2(b+r)-\frac{b\alpha}{a}\right)
	+\log\left\{4\left(\frac{a\mathcal{M}_a}{2a-\alpha}\right)^{\frac{2a-\alpha}{a}}\right\}\nonumber\\
	&= \log  \left\{\left(\frac{1}{2(b+r)-b \alpha / a}\right)^{2(b+r)-b \alpha / a}(2\theta) \left(\frac{\nu}{2}\right)^{-\frac{\alpha}{a}} \mathcal{M}_a^{\frac{2a-\alpha}{a}}\right\}.
	\nonumber
\end{align}

More precisely, for any $\delta>0$ there exists a constant $C_\delta>0$ such that for all $n\ge1$,
\begin{align*}
\theta^{n}\mathbb{E}S_{2n}\left(g_{2n}\left(\cdot,1,0\right)\right)
\le\frac{ C_\delta}{(n!)^{2(b+r)-\frac{b\alpha}{a}-1}}
\left[(1+\delta)\left(\frac{1}{2(b+r)-b \alpha / a}\right)^{2(b+r)-b \alpha / a}\left(\frac{\nu}{2}\right)^{-\frac{\alpha}{a}} \theta 2 \mathcal{M}_a^{\frac{2a-\alpha}{a}}\right]^n,
\end{align*}
and
\begin{align*}
\theta^{n}\mathbb{E}S_{2n}\left(g_{2n}\left(\cdot,1,0\right)\right)
\ge\frac{ C_\delta^{-1}}{(n!)^{2(b+r)-\frac{b\alpha}{a}-1}}
\left[(1-\delta)\left(\frac{1}{2(b+r)-b \alpha / a}\right)^{2(b+r)-b \alpha / a}\left(\frac{\nu}{2}\right)^{-\frac{\alpha}{a}} \theta 2 \mathcal{M}_a^{\frac{2a-\alpha}{a}}\right]^n.
\end{align*}
Combining the two-sided estimate obtained above with \eqref{eq:L1Sn1}, and using the asymptotics of the Mittag--Leffler function \eqref{e:asym-ML}, we let $\delta\to0$ to obtain
\begin{equation*}
     \lim _{t \rightarrow \infty} t^{-\beta} \log \mathbb{E}\left[u(t, x)\right]
    = \frac{1}{2} \mathbf M(a,b,r,\theta, \nu, \alpha).   
\end{equation*}
This completes the proof.
\end{proof}

 \section{Lower bound}\label{se:lower}
 
In this section, we establish the lower bound for the asymptotic behavior of the solution to~\eqref{e:sfde} that matches the upper bound in Theorem~\ref{T:upperbound}. By refining the argument of \cite{ch2024} through the finite Wick expansion and exploiting the positivity of all the resulting contraction terms, we remove the loss from $p$ to $p-1$ in the previous lower-bound argument. Consequently, both the long-time asymptotics, as $t\to\infty$ with $p>1$ fixed, and the moment asymptotics, as $p\to\infty$ with $t>0$ fixed, have exactly the same rates and constants as those in upper bounds in Theorem~\ref{T:upperbound}. See Theorem~\ref{T:lowerbound1} and Remark~\ref{rem:lower} for more details.


Set $\cH_+:=\{f\in\cH:f\geq0\}$. For integers $m\geq0$ and
$0\leq\ell\leq\lfloor m/2\rfloor$, let $\Pi_{m,\ell}$ denote the
collection of all sets $\mathcal D$ consisting of $\ell$ disjoint
unordered pairs from $\{1,\ldots,m\}$.  

\begin{lemma}\label{L:finite-Wick-expansion}
Let $m$ be a non-negative integer and let $f\in\cH_+$.
Then
\begin{equation}\label{e:finite-Wick-expansion}
\begin{aligned}
 &e^{-\frac12\|f\|_{\cH}^2}
 \E\left[e^{W(f)}
 S_m\bigl(g_m(\cdot,t,0)\bigr)\right]\\
 &=
 \sum_{\ell=0}^{\lfloor m/2\rfloor}
 \sum_{\mathcal D\in\Pi_{m,\ell}}
 \int_{(\R^d)^m}g_m(x_1,\ldots,x_m,t,0)
 \prod_{(j,k)\in\mathcal D}\gamma(x_j-x_k)\\
 &\hspace{35mm}\times
 \prod_{j\notin\bigcup_{(k,l)\in\mathcal D}\{k,l\}}
 (\gamma*f)(x_j)
 \ud x_1\cdots\ud x_m.
\end{aligned}
\end{equation}
In particular, every term on the right-hand side is nonnegative. 
\end{lemma}

\begin{proof}
Replacing $\dot W$ by $\dot W_\e$ and expanding the exponential, we
have
\begin{align*}
 e^{-\frac12\|f\|_{\cH}^2}
 \E\left(e^{W(f)}\prod_{j=1}^m\dot W_\e(x_j)\right)
 &=e^{-\frac12\|f\|_{\cH}^2}
 \sum_{q=0}^{\infty}\frac1{q!}
 \E\left(W(f)^q\prod_{j=1}^m\dot W_\e(x_j)\right).
\end{align*}
Note  the expectations on the right-hand side of  the above equation are non-zero only if $m$ and $q$ share the same parity due to Wick's formula~\eqref{e:wick-formula}. 
Denoting $g_i=\dot W_\varepsilon(x_i)$ for $i=1, ..., m$ and $g_i=W(f)$ for $i=m+1, ..., m+q$,  we have
\begin{align*}
&\E\left(W(f)^q\prod_{j=1}^m\dot W_\e(x_j)\right)=\sum_{\mathcal D\in\Pi_{m+q}} \prod_{(k,l)\in \mathcal D} \E[g_k g_l]\\
&=\sum_{ \ell=0\vee(m-q)/2}^{\lfloor m/2\rfloor}\, \sum_{\mathcal D\in\Pi_{m, \ell }} \left(\prod_{(k,l)\in \mathcal D} \gamma_{2\varepsilon}(x_k-x_l) \right) \left(\frac{(2j)!}{2^j j!}\|f\|_{\mathcal H}^{2j} \right) \left(\frac{q!}{(2j)!} \prod_{k\notin\cup_{(i,l)\in \mathcal D}\{i,l\}} (\gamma_\varepsilon*f)(x_k)\right),
\end{align*}
where  in the second step we set $2j:=q-(m-2\ell)$ use the fact $$\E[\dot W_\varepsilon(x_k)\dot W_\varepsilon(x_l)]=\gamma_{2\varepsilon}(x_k-x_l), \quad \E[W(f) \dot W_\varepsilon(x_k)]=(\gamma_\varepsilon*f)(x_k).$$


Therefore, using the change of variable $q=m-2\ell +2j$ we get
\begin{align*}
 &e^{-\frac12\|f\|_{\cH}^2}
 \E\left[e^{W(f)}\prod_{j=1}^m\dot W_\e(x_j)\right]\\
 &=\sum_{\ell=0}^{\lfloor m/2\rfloor}
 \sum_{\mathcal D\in\Pi_{m,\ell}}
 \left(e^{-\frac12\|f\|_{\cH}^2}
 \sum_{j=0}^{\infty}
 \frac{\|f\|_{\cH}^{2j}}{2^j j!}\right)
 \prod_{(k,l)\in\mathcal D}\gamma_{2\e}(x_k-x_l)
 \prod_{k\notin\bigcup_{(i,l)\in\mathcal D}\{i,l\}}
 (\gamma_\e*f)(x_k)\\
 &=\sum_{\ell=0}^{\lfloor m/2\rfloor}
 \sum_{\mathcal D\in\Pi_{m,\ell}}
 \prod_{(k,l)\in\mathcal D}\gamma_{2\e}(x_k-x_l)
 \prod_{k\notin\bigcup_{(i,l)\in\mathcal D}\{i,l\}}
 (\gamma_\e*f)(x_k).
\end{align*}
 Consequently,
\begin{align*}
	&e^{-\frac12\|f\|_{\cH}^2}
	\E\left[e^{W(f)}\theta^{m/2}	S_{m,\e}\bigl(g_m(\cdot,t,0)\bigr)\right]\\
	&\quad=\theta^{m/2}
	\sum_{\ell=0}^{\lfloor m/2\rfloor}
	\sum_{\mathcal D\in\Pi_{m,\ell}}
	\int_{(\R^d)^m}g_m(x_1,\ldots,x_m,t,0)
	\prod_{(j,k)\in\mathcal D}\gamma_{2\e}(x_j-x_k)\\
	&\hspace{35mm}\times
	\prod_{j\notin\bigcup_{(k,l)\in\mathcal D}\{k,l\}}
	(\gamma_\e*f)(x_j)
	\ud x_1\cdots\ud x_m .
\end{align*}
By the same mollification argument as in Section~3, together with the
Laplace-transform representation in
Lemma~\ref{L:shifted-stable-transform} and the moment estimates in
Lemma~\ref{L:n momentbound}, one readily verifies that, as
$\e\downarrow0$, the two sides of the preceding identity converge to
the corresponding sides of~\eqref{e:finite-Wick-expansion}.  We omit
the routine details.  This completes the proof.
\end{proof}


Let $X=\{X(t):t\geq0\}$ be the symmetric $a$-stable process introduced
in Section \ref{se:technical}, independent of $W$, and let $\mathbf E_0$ denote its
expectation when it starts from the origin.  By
Lemma~\ref{L:Wdefined}, conditionally on $X$,
$\int_0^t\dot W(X(s))\ud s$ is normally distributed with mean zero and
variance
\begin{equation}
 \mathbb E\left[\left.
 \left(\int_0^t\dot W(X(s))\ud s\right)^2\right|X\right]
 =\int_0^t\int_0^t\gamma(X(s)-X(r))\ud s\ud r.
 \label{e:stable-noise-variance}
\end{equation}
For $f\in\cH$, the same mollification and the $\mathcal L^2$-convergence in Lemma~\ref{L:Wdefined} show that the pair
$\bigl(W(f),\int_0^t\dot W(X(s))\ud s\bigr)$ is conditionally jointly
Gaussian with
\begin{align}
 \mathbb E\left[\left.W(f)^2\right|X\right]
 &=\|f\|_{\cH}^2,\notag\\
 \mathbb E\left[\left.
 W(f)\int_0^t\dot W(X(s))\ud s\right|X\right]
 &=\lim_{\e\downarrow0}\int_0^t
 \left\langle f,p_\e(X(s)-\cdot)\right\rangle_{\cH}\ud s\notag\\
 &=\int_0^t(\gamma*f)(X(s))\ud s,
 \label{e:stable-cross-covariance}
\end{align}
where the limit is taken in $\mathcal{L}^2(\Omega, \mathcal{F}, \mathbb{P})$.
If $f\geq0$, then the last quantity is nonnegative, and applying Cauchy-Schwarz inequality to~\eqref{e:stable-cross-covariance} yields
\begin{equation}\label{e:stable-projection}
 \int_0^t(\gamma*f)(X(s))\ud s
 \leq \|f\|_{\cH}
 \left(\int_0^t\int_0^t
 \gamma(X(s)-X(r))\ud s\ud r\right)^{1/2}.
\end{equation}

\begin{lemma}\label{L:shifted-stable-transform}
For every $m\geq0$, every $f\in\cH_+$, and every
$\lambda>0$,
\begin{equation}
    \label{e:shifted-stable-transform}
\begin{aligned}
 &\int_0^\infty e^{-\lambda s}
 e^{-\frac12\|f\|_{\cH}^2}
 \E\left[e^{W(f)} S_m\bigl(g_m(\cdot,s,0)\bigr)\right]\ud s\\
 &=
 \frac{\lambda^{b-mr-1}}{m!}
 \left(\frac2\nu\right)^{m+1}\\
 &\quad \times
 \int_0^\infty e^{-2\lambda^bt/\nu}\,
 \mathbf E_0\E_N\Bigg[
 \Bigg\{
 N\left(\int_0^t\int_0^t
 \gamma(X(s)-X(r))\ud s\ud r\right)^{1/2}
 +\int_0^t(\gamma*f)(X(s))\ud s
 \Bigg\}^{m}\Bigg]\ud t,
 \end{aligned}
\end{equation}
where $N$ is a standard Gaussian random variable independent of $X$.
\end{lemma}

\begin{proof}
By Theorem~\ref{T:laplace Sn}, we get
\begin{equation}\label{e:int-eWSm}
\begin{aligned}
 &\int_0^\infty e^{-\lambda s}e^{-\frac12\|f\|_{\cH}^2}
 \E\left[e^{W(f)} S_m\bigl(g_m(\cdot,s,0)\bigr)\right]\ud s\\
 &=\frac{\lambda^{b-mr-1}}{m!}
 \left(\frac2\nu\right)^{m+1}
 \int_0^\infty e^{-2\lambda^bt/\nu}\mathbf E_0\E\left[
 e^{W(f)-\frac12\|f\|_{\cH}^2}
 \left(\int_0^t\dot W(X(s))\ud s\right)^m
 \right]\ud t.
\end{aligned}
\end{equation}
Here the expectations and the integration can be interchanged by the
$\mathcal L^2$-estimate in Theorem~\ref{T:laplace Sn} and
Cauchy--Schwarz inequality.  For each fixed path of $X$,
\eqref{e:stable-noise-variance}--\eqref{e:stable-cross-covariance} give,
for every $z\in\R$,
\begin{align*}
 &e^{-\frac12\|f\|_{\cH}^2}
 \E\left[\exp\left\{W(f)+z\int_0^t\dot W(X(s))\ud s\right\}
 \right]\\
 &\quad=\exp\Bigg\{
 z\int_0^t(\gamma*f)(X(s))\ud s
 +\frac{z^2}{2}\int_0^t\int_0^t
 \gamma(X(s)-X(r))\ud s\ud r\Bigg\}\\
 &\quad=\E_N\exp\Bigg\{z\Bigg[
 N\left(\int_0^t\int_0^t
 \gamma(X(s)-X(r))\ud s\ud r\right)^{1/2}
 +\int_0^t(\gamma*f)(X(s))\ud s\Bigg]\Bigg\}.
\end{align*}
Differentiating both sides $m$ times at $z=0$ gives
\begin{align*}
 &e^{-\frac12\|f\|_{\cH}^2}
 \E\left[e^{W(f)}
 \left(\int_0^t\dot W(X(s))\ud s\right)^m\right]\\
 &\quad=
 \E_N\Bigg[
 \Bigg\{N\left(\int_0^t\int_0^t
 \gamma(X(s)-X(r))\ud s\ud r\right)^{1/2}
 +\int_0^t(\gamma*f)(X(s))\ud s\Bigg\}^{m}\Bigg].
\end{align*}
Substituting this into~\eqref{e:int-eWSm} yields the desired~\eqref{e:shifted-stable-transform}. The proof is complete.
\end{proof}

Now we are ready to prove the main result in this section: 

\begin{theorem}\label{T:lowerbound1}
Let $u(t,x)$ be the mild Stratonovich solution to~\eqref{e:sfde}.  For
every real number $p>1$,
\begin{equation}\nonumber
 \liminf_{t\to\infty}t^{-\beta}
 \log\E\left[|u(t,x)|^p\right]
 \geq
 \frac{p^\beta}{2}\,
 \mathbf M(a,b,r,\theta,\nu,\alpha).
\end{equation}
Moreover, for every fixed $t>0$,
\begin{equation}\nonumber
 \liminf_{p\to\infty}p^{-\beta}
 \log\E\left[|u(t,x)|^p\right]
 \geq
 \frac{t^\beta}{2}\,
 \mathbf M(a,b,r,\theta,\nu,\alpha),
\end{equation}
where $\mathbf M(a,b,r,\theta,\nu,\alpha)$ is defined
in~\eqref{e:bM}.
\end{theorem}

\begin{proof}
For convenience, set
\begin{equation*}
	\chi:=2(b+r)-\frac{b\alpha}{a},
\end{equation*}
and recall from \eqref{e:de-beta} that $\beta=\dfrac{\chi}{\chi-1}$.

	By spatial stationarity, it suffices to consider the case $x=0$.  For $p>1$
	and $t>0$, set
	\begin{equation}\nonumber
		t_p:=p^{\beta-1}t^\beta.
	\end{equation}
	Let $k,c>0$ be fixed and set
	\begin{equation}\label{e:lower-n-t}		n_t:=2\left\lfloor\frac{ct_p}{2}\right\rfloor,
		\qquad
		R:=\frac{p-1}{\sqrt p}\,k\sqrt{t_p}.
	\end{equation}
	For every $f\in\cH_+$ with $\|f\|_{\cH}=R$,
	\begin{equation*}
		\left\|\exp\left\{W(f)-
		\frac{p}{2(p-1)}R^2\right\}\right\|_{\frac{p}{p-1}}=1.
	\end{equation*}
	Hence, by H\"older's inequality and~\eqref{e:expan-utx},
	\begin{align*}
		\|u(t,0)\|_p
        &\ge \E\left[u(t,0) \exp\left\{W(f)-
		\frac{p}{2(p-1)}R^2\right\}\right]\\
		&\geq \exp\left\{-\frac{R^2}{2(p-1)}\right\}
		\sum_{j=0}^{\infty}e^{-\frac{1}{2}R^2}
		\E\left[e^{W(f)}\theta^{\frac{j}{2}}
		S_j\bigl(g_j(\cdot,t,0)\bigr)\right].
	\end{align*}
	Since $g_j(\cdot,t,0)\geq0$, Lemma~\ref{L:finite-Wick-expansion}
	shows that every summand is nonnegative.  Retaining the term
	$j=n_t$ and taking the supremum over $f$ implies
	\begin{align}
		\|u(t,0)\|_p
		&\geq \exp\left\{-\frac{R^2}{2(p-1)}\right\}
		\sup_{f\in\cH_+,\,\|f\|_{\cH}=R}
		e^{-\frac{R^2}{2}}\E\left[e^{W(f)}\theta^{\frac{n_t}{2}}
		S_{n_t}\bigl(g_{n_t}(\cdot,t,0)\bigr)\right].
		\label{e:lower-holder-coefficient}
	\end{align}
	
	We next estimate the supremum in~\eqref{e:lower-holder-coefficient}.
	For every integer $n\geq0$ and
	every $R>0$, using the scaling property \eqref{E:scal S},
	\begin{align}
		\sup_{f\in\cH_+,\,\|f\|_{\cH}=R}
		\E\left[e^{W(f)}\theta^{\frac{n}{2}}
		S_n\bigl(g_n(\cdot,t,0)\bigr)\right]
		=t^{\frac{\chi n}{2}}
		\sup_{f\in\cH_+,\,\|f\|_{\cH}=R}
		\E\left[e^{W(f)}\theta^{\frac{n}{2}}
		S_n\bigl(g_n(\cdot,1,0)\bigr)\right].
		\label{e:complete-sphere-scaling}
	\end{align}
	Applying~\eqref{e:complete-sphere-scaling}, we obtain
	\begin{equation}
	    \label{e:lower-sup-Laplace}
    \begin{aligned}
		&\sup_{f\in\cH_+,\,\|f\|_{\cH}=R}
		e^{-\frac{R^2}{2}}\E\left[e^{W(f)}\theta^{\frac{m}{2}}
		S_m\bigl(g_m(\cdot,t,0)\bigr)\right]\\
		&\quad=t^{\frac{\chi m}{2}}
		\sup_{f\in\cH_+,\,\|f\|_{\cH}=R}
		e^{-\frac{R^2}{2}}\E\left[e^{W(f)}\theta^{\frac{m}{2}}
		S_m\bigl(g_m(\cdot,1,0)\bigr)\right]\\
		&\quad=\frac{t^{\frac{\chi m}{2}}}{\Gamma(\frac{\chi m}{2}+1)}
		\int_0^\infty e^{-s}
		\sup_{f\in\cH_+,\,\|f\|_{\cH}=R}
		e^{-\frac{R^2}{2}}\E\left[e^{W(f)}\theta^{\frac{m}{2}}
		S_m\bigl(g_m(\cdot,s,0)\bigr)\right]\ud s\\
		&\quad\geq
		\frac{t^{\frac{\chi m}{2}}}{\Gamma(\frac{\chi m}{2}+1)}
		\sup_{f\in\cH_+,\,\|f\|_{\cH}=R}
		\int_0^\infty e^{-s}e^{-\frac{R^2}{2}}
		\E\left[e^{W(f)}\theta^{\frac{m}{2}}
		S_m\bigl(g_m(\cdot,s,0)\bigr)\right]\ud s.
	\end{aligned}
    \end{equation}
	Let
	$h\in\cH_+$ with $\|h\|_{\cH}=1$. Define
	\begin{equation*}
		f_p:=Rh
		=\frac{p-1}{\sqrt p}\,k\sqrt{t_p}\,h.
	\end{equation*}
	Thus $f_p\in \mathcal H_+$ with $\|f_p\|_{\cH}=R,$ and hence is an admissible choice in~\eqref{e:lower-sup-Laplace} for lower bound.
	
	Let $N$ be the standard Gaussian random variable as in
	Lemma~\ref{L:shifted-stable-transform}, and set
	\begin{equation*}
	    E:=\left\{k\sqrt{\frac{t_p}{p}}\leq N  \right\}.
	\end{equation*}
	Then, we get on $E$, the quantity in \eqref{e:shifted-stable-transform}
	\begin{align*}
		&N\left(\int_0^t\int_0^t
		\gamma(X(s)-X(r))\ud s\ud r\right)^{\frac{1}{2}}
		+\int_0^t(\gamma*f_p)(X(s))\ud s\\
		&\geq (N+R)\int_0^t(\gamma*h)(X(s))\ud s
        \ge
		\sqrt p\,k\sqrt{t_p}
		\int_0^t(\gamma*h)(X(s))\ud s,
	\end{align*}
    where the first inequaity follows from~\eqref{e:stable-projection} and the second one is due to \eqref{e:lower-n-t} and the definition of $E$.

	Noting that $n_t$ is even, the integrand in the expectation with respect to
	$N$ in~\eqref{e:shifted-stable-transform} is nonnegative,   and hence we
	may restrict this expectation on $E$ to get a lower bound. Taking $\lambda=1$, $m=n_t$ and $f=f_p$ in
	Lemma~\ref{L:shifted-stable-transform} gives
	\begin{equation}
	    \label{e:lower-before-spectral}
    \begin{aligned}
		&\int_0^\infty e^{-s}
		e^{-\frac{1}{2}\|f_p\|_{\cH}^2}
		\E\left[e^{W(f_p)}\theta^{\frac{n_t}{2}}
		S_{n_t}\bigl(g_{n_t}(\cdot,s,0)\bigr)\right]\ud s\\
		&\geq\P(E) \frac{\theta^{\frac{n_t}{2}}}{n_t!}
		(p k^2t_p)^{\frac{n_t}{2}}\left(\frac{2}{\nu}\right)^{n_t+1}
		\int_0^\infty e^{-\frac{2t}{\nu}}
		\mathbf E_0\left[
		\left(\int_0^t(\gamma*h)(X(s))\ud s\right)^{n_t}\right]\ud t.
	\end{aligned}
    \end{equation}
	
	It remains to choose $h$ and estimate the last integral.  For
	$h\in\cH$, set
	\begin{equation*}
		h_\nu(x):=\left(\frac{\nu}{2}\right)^{\frac{2d-\alpha}{2a}}
		h\left(\left(\frac{\nu}{2}\right)^{\frac{1}{a}}x\right).
	\end{equation*}
	Then, we have 
	\begin{align*}
		&\frac{1}{n!}\left(\frac{2}{\nu}\right)^{n+1}
		\int_0^\infty e^{-\frac{2t}{\nu}}
		\mathbf E_0\left[
	\left(\int_0^t(\gamma*h)(X(s))\ud s\right)^n\right]\ud t
		\notag\\
        &=\left(\frac{2}{\nu}\right)^{n+1}
		\int_0^\infty e^{-\frac{2t}{\nu}}\int_{0<s_1<\cdots<s_n<t}
		\mathbf E_0\left[ \prod_{i=1}^n
	(\gamma*h)(X(s_i))\right]\ud \mathbf{s} \ud t\notag\\
		&=\int_{(\R^d)^n}
		\prod_{k=1}^n(\mathcal{F}h)(\xi_k)
		\prod_{k=1}^n
		\frac{1}{1+\frac{\nu}{2}|\xi_k+\cdots+\xi_n|^a}
		\boldsymbol\mu(\ud\boldsymbol\xi)
		\notag\\
		&=\left(\frac{\nu}{2}\right)^{-\frac{\alpha n}{2a}}
		\int_{(\R^d)^n}\prod_{k=1}^n
		(\mathcal{F}h_\nu)(\xi_k)
		\prod_{k=1}^n
		\frac{1}{1+|\xi_k+\cdots+\xi_n|^a}
		\boldsymbol\mu(\ud\boldsymbol\xi).
	\end{align*}
	For every nonnegative and nonnegative definite function $\varphi$, set
	\begin{align*}
		\rho(\varphi):=
		\sup_{\|h\|_{L^2(\R^d)}=1}
		\int_{\R^d}\varphi(\xi)
		\left[\int_{\R^d}
		\frac{h(\xi+\eta)h(\eta)}
		{\sqrt{1+|\xi+\eta|^a}\sqrt{1+|\eta|^a}}\ud\eta\right]
		\mu(\ud\xi).
	\end{align*}
	By \cite[Section~3]{Bass2009},
	\begin{align*}
		&\liminf_{n\to\infty}\frac{1}{n}\log
		\int_{(\R^d)^n}\prod_{k=1}^n(\mathcal{F}h_\nu)(\xi_k)
		\prod_{k=1}^n
		\frac{1}{1+|\xi_k+\cdots+\xi_n|^a}
		\boldsymbol\mu(\ud\boldsymbol\xi)
		\geq\log\rho(\mathcal{F}h_\nu).
	\end{align*}
	Moreover, by \cite[Theorem~3.5]{Chen2023Interpolating} or
	\cite[Theorem~1.5]{Bass2009},
	\begin{equation*}
	\sup_{h\in\cH_+,\|h\|_{\cH}=1}
		\rho(\mathcal{F}h)
		=\mathcal{M}_a^{\frac{2a-\alpha}{2a}}.
	\end{equation*}
	Therefore, for every $\e\in(0,1)$, there exists a constant $C_\e>0$ such that
	\begin{equation}
	    \label{e:lower-scaled-spectral}
    \begin{aligned}
		&\frac{1}{n!}\left(\frac{2}{\nu}\right)^{n+1}
		\int_0^\infty e^{-\frac{2t}{\nu}}
		\mathbf E_0\left[
		\left(\int_0^t(\gamma*f)(X(s))\ud s\right)^n\right]\ud t\\
		&\geq C_\e^{-1}
		\left(\frac{\nu}{2}\right)^{-\frac{\alpha n}{2a}}
		\left((1-\e)
		\mathcal{M}_a^{\frac{2a-\alpha}{2a}}\right)^n,
		\quad\text{for all } n\geq1.
	\end{aligned}
    \end{equation}
	Taking $n=n_t$ in~\eqref{e:lower-scaled-spectral} and substituting it
	into~\eqref{e:lower-before-spectral}, we obtain
	\begin{align}
		&\int_0^\infty e^{-s}
		e^{-\frac{1}{2}\|f_p\|_{\cH}^2}
		\E\left[e^{W(f_p)}\theta^{\frac{n_t}{2}}
		S_{n_t}\bigl(g_{n_t}(\cdot,s,0)\bigr)\right]\ud s \notag\\
        &\geq C_\e^{-1}\P(E)
		\left\{p k^2t_p(1-\e)^2\theta
		\left(\frac{\nu}{2}\right)^{-\frac{\alpha}{a}}
		\mathcal{M}_a^{\frac{2a-\alpha}{a}}\right\}^{\frac{n_t}{2}}.
		\label{e:lower-Laplace-coefficient}
	\end{align}
	
	Combining \eqref{e:lower-holder-coefficient}, \eqref{e:lower-sup-Laplace} and \eqref{e:lower-Laplace-coefficient} with $m=n_t$, we obtain
\begin{align}
	\|u(t,0)\|_p
	&\geq C_\e^{-1}\P(E)
	\exp\left\{-\frac{p-1}{2p}k^2t_p\right\}
	\frac{t^{\frac{\chi n_t}{2}}}{\Gamma(\frac{\chi n_t}{2}+1)}
	\left\{p k^2t_p(1-\e)^2\theta
	\left(\frac{\nu}{2}\right)^{-\frac{\alpha}{a}}
	\mathcal{M}_a^{\frac{2a-\alpha}{a}}\right\}^{\frac{n_t}{2}}.
	\label{e:lower-pre-asymptotic}
\end{align}
For the standard normal density, we have
\begin{equation*}
	\P(E)\geq\frac{1}{\sqrt{2\pi}}
	\exp\left\{-\frac{1}{2}
	\left(k\sqrt{\frac{t_p}{p}}+1\right)^2\right\}.
\end{equation*}
Consequently, as $t_p\to\infty$,
\begin{equation}\nonumber
	\frac{1}{t_p}\left\{\log\P(E)
	-\frac{p-1}{2p}k^2t_p\right\}
	\geq-\frac{k^2}{2}+o(1).
\end{equation}

Stirling's formula gives
\begin{equation*}
	\frac{1}{t_p}\log\Gamma\!\left(\frac{\chi n_t}{2}+1\right)
	=\frac{n_t}{t_p}\left(\frac{\chi}{2}\log n_t+\frac{\chi}{2}\log\frac{\chi}{2}-\frac{\chi}{2}\right)+o(1), \quad \text{as } t_p\to\infty.
\end{equation*}
Substituting this estimate and using $\log n_t=\log t_p+\log\frac{n_t}{t_p}$, we can rewrite the right-hand side of \eqref{e:lower-pre-asymptotic} as
\begin{align*}
	&-\frac{k^2}{2}+o(1)
	+\frac{n_t}{2t_p}\Bigg[\log\left\{k^2(1-\e)^2\theta
	\left(\frac{\nu}{2}\right)^{-\frac{\alpha}{a}}
	\mathcal{M}_a^{\frac{2a-\alpha}{a}}\right\}\\
	&\qquad +\log p-(\chi-1)\log t_p+\chi\log t
	-\chi\log\frac{n_t}{t_p}
	-\chi\log\frac{\chi}{2}+\chi\Bigg].
\end{align*}
Since $t_p=p^{\beta-1}t^\beta$ and $\beta=\frac{\chi}{\chi-1}$, we have
\begin{equation*}
	\log p-(\chi-1)\log t_p+\chi\log t=0.
\end{equation*}
Therefore, letting $t_p\to\infty$, we obtain
\begin{align*}
	\liminf_{t_p\to\infty}
	\frac{1}{t_p}\log\|u(t,0)\|_p
	&\ge -\frac{k^2}{2}
	+\frac{c}{2}\Bigg[-\chi\log\frac{c\chi}{2}+\chi+\log\left\{k^2(1-\e)^2\theta
	\left(\frac{\nu}{2}\right)^{-\frac{\alpha}{a}}
	\mathcal{M}_a^{\frac{2a-\alpha}{a}}\right\}\Bigg].
\end{align*}

	We first optimize the right-hand side with respect to $c$. Substituting the optimal choice 
	\begin{equation*}
		c^* = \frac{2}{\chi} \left\{ k^2(1-\e)^2\theta \left(\frac{\nu}{2}\right)^{-\frac{\alpha}{a}} \mathcal{M}_a^{\frac{2a-\alpha}{a}} \right\}^{\frac{1}{\chi}}
	\end{equation*}
	into the preceding lower bound gives
	\begin{equation*}
		\liminf_{t_p\to\infty} \frac{1}{t_p}\log\|u(t,0)\|_p \ge -\frac{k^2}{2} + \left\{ k^2(1-\e)^2\theta \left(\frac{\nu}{2}\right)^{-\frac{\alpha}{a}} \mathcal{M}_a^{\frac{2a-\alpha}{a}} \right\}^{\frac{1}{\chi}}.
	\end{equation*}
	We next substitute the optimal choice of $k$ which satisfies
	\begin{equation*}
		(k^*)^2 = \left(\frac{2}{\chi}\right)^{\frac{\chi}{\chi-1}} \left\{ (1-\e)^2\theta \left(\frac{\nu}{2}\right)^{-\frac{\alpha}{a}} \mathcal{M}_a^{\frac{2a-\alpha}{a}} \right\}^{\frac{1}{\chi-1}}.
	\end{equation*}
	and get
	\begin{equation*}
		\liminf_{t_p\to\infty} \frac{1}{t_p}\log\|u(t,0)\|_p \ge \frac{\chi-1}{2} \left(\frac{2}{\chi}\right)^{\frac{\chi}{\chi-1}} \left\{ (1-\e)^2\theta \left(\frac{\nu}{2}\right)^{-\frac{\alpha}{a}} \mathcal{M}_a^{\frac{2a-\alpha}{a}} \right\}^{\frac{1}{\chi-1}}.
	\end{equation*}
	
	Letting $\e\downarrow0$, we obtain
	\begin{equation*}
		\liminf_{t_p\to\infty} \frac{1}{t_p}\log\|u(t,0)\|_p
		\ge \frac{\chi-1}{2} \left(\frac{2}{\chi}\right)^{\beta}
		\theta^{\frac{1}{\chi-1}}
		\left(\frac{\nu}{2}\right)^{-\frac{\alpha}{a(\chi-1)}}
		\mathcal{M}_a^{\frac{2a-\alpha}{a(\chi-1)}}.
	\end{equation*}
	Taking $t\to\infty$ with $p$ fixed, or $p\to\infty$ with $t$ fixed, gives the two estimates stated in Theorem~\ref{T:lowerbound1}. This completes the proof. \end{proof}

\begin{remark}\label{rem:lower} 
To derive a lower bound for the asymptotics of the $p$-th moment with $p>1$, a standard approach is to apply H\"older's inequality in the form $\|u(t,0)\|_p\ge \E[u(t,0)X]$, where $X$ is any random variable satisfying $\|X\|_q=1$ and is typically chosen as $X=\exp\left\{W(f)-\frac{p}{2(p-1)}\|f\|_{\mathcal H}^2\right\}$. Using this approach, a variational inequality was established in \cite[Proposition~4.1]{Chen2018Temporal} for fractional SHEs whose solutions admit Feynman--Kac representations and was then used to obtain the lower bound for the large-$t$ asymptotics. For the stochastic wave equation whose solution is given by the series expansion $u(t,x)=\sum_n S_n(g_n)$, a variational inequality was proved in \cite[Lemma~5.4]{ch2024} and subsequently used to obtain the lower bound for the large-$p$ asymptotics.

Since our FDE \eqref{e:sfde} admits a series expansion solution, adapting the argument of \cite[Lemma~5.4]{ch2024} to the present setting also yields the optimal lower bound for the large-$p$ asymptotics. However, it does not yield the optimal lower bound for the large-$t$ asymptotics, since the exponent $p-1$ on the right-hand side of eq.~(5.29) in \cite{ch2024} is not optimal. In light of \cite[Proposition~4.1]{Chen2018Temporal} with $\rho=0$, we conjecture that this exponent can be improved from $p-1$ to $p$.

In our proof of Theorem~\ref{T:lowerbound1}, when estimating $\E[Xu(t,0)]=\sum_n\E[XS_n(g_n)]$, we employ the Laplace transform technique developed in Lemma~\ref{L:shifted-stable-transform} to derive a lower bound for $\E[XS_n(g_n)]$. It turns out that the term $N\left(\int_0^t\int_0^t\gamma(X(s)-X(r))\ud s\ud r\right)^{1/2}$ on the right-hand side of \eqref{e:shifted-stable-transform} plays an essential role in obtaining the optimal lower bound. This term, however, was discarded in the proof of \cite[Lemma~5.4]{ch2024}.
\end{remark}

\subsection*{Acknowledgment} We wish to thank Xia Chen for helpful discussions. J.\ Song is partially supported by NSFC (No.\ 12471142) and the Fundamental Research Funds for the Central Universities.

\appendix
\section{Some technical lemmas}\label{ap:A}

The following lemma is borrowed form \cite[Lemma A.3 (i)]{Guo2025Sample}.
\begin{lemma}\label{L:laplace G}
Recall that $\cF G(t, \cdot)(\xi)$ is given by \eqref{E:FZ}. For all $\lambda>0$, the following identity holds
\begin{equation*}
    \int_0^\infty e^{-\lambda t} \cF G(t,\cdot)(\xi)\ud t = \frac{\lambda^{-r}}{\lambda^b+ \frac{\nu}{2} |\xi|^a }.
\end{equation*}
\end{lemma}

The following lemma is  a revised version of \cite[Lemma 2.2.7]{Chenbook}.

\begin{lemma}\label{le:chenbook-lem227}
    For any continuous functions $\varphi_1(t),\dots,\varphi_n(t)$ satisfying
    \begin{equation*}
        \int_0^\infty e^{-t} |\varphi_k(t)|\ud t <\infty, k=1,\dots,n,
    \end{equation*}
    and for any $\lambda>0$, we have
    \begin{equation*}
        \int_0^\infty\ud t \, e^{-\lambda t} \int_{[0, t]_{<}^n} \ud s_1\cdots \ud s_n \prod_{k=1}^n \varphi(s_k - s_{k-1}) = \lambda^{-1} \prod_{k=1}^n \int_0^{\infty} e^{-\lambda t} \varphi_k(t)  \ud t,
    \end{equation*}
    with the convention $s_0 = 0$ and $[0, t]_{<}^n:=\left\{\left(s_1, \cdots, s_n\right) \in[0, t]^n : 0<s_1<s_2<\cdots<s_n<t\right\}$.
\end{lemma}
\begin{proof}
    Using the changes of variables 
    \begin{equation*}
        t_k=s_k- s_{k-1}, \quad \text{ for } k=1,\dots,n+1,
    \end{equation*}
    with $s_{n+1}=t$, we have 
    \begin{align*}
        &\int_0^\infty \ud t \, e^{-\lambda t} \int_{[0, t]_{<}^n} \ud s_1\cdots \ud s_n \prod_{k=1}^n \varphi(s_k - s_{k-1})\\
        &=  \int_0^\infty\cdots\int_0^\infty e^{-\lambda t_{n+1}} \bigg( \prod_{k=1}^n   e^{-\lambda t_k} \varphi(t_k) \bigg) \ud t_1\cdots \ud t_{n+1}
        = \lambda^{-1} \prod_{k=1}^n \int_0^{\infty} e^{-\lambda t} \varphi_k(t)  \ud t.
    \end{align*}
    This completes the proof.
\end{proof}

Recall that $X(t), X_1(t), X_2(t), \cdots$ are independent  $d$-dimensional symmetric $a$-stable processes and $\mathbf{E}_x$ is the expectation with respect to the $a$-stable processes with starting point $x$.
The following lemma is taken from (3.7) and Proposition 3.5 of \cite{song2017} with $\beta_0=0$ and $\Psi(\xi)=|\xi|^a$, which we restate here.

\begin{lemma} \label{L:n momentbound} 
	Under Assumption \ref{A:Main}, there exists a positive constant $A_0$  such that for all $N>0$,
	\begin{equation} \label{nbound1}
		\mathbf{E}_0\left[\int_0^t \int_0^t \gamma(X(s)-X(r)) \ud s \ud r\right]^n \leq (2n-1)!! \: t^{n}\sum_{k=0}^n\binom{n}{k} \frac{ t^{k}}{k!} m_N^k\left[A_0 \varepsilon_N\right]^{n-k}
	\end{equation}
and

\begin{equation} \nonumber
		 \mathbf{E}_0\left[\int_0^t \int_0^t \gamma\left(X_1(s)-X_2(r)\right) \ud s \ud r\right]^n\leq n! \: t^{n} \sum_{k=0}^n \binom{n}{k} \frac{ t^{k}}{k!} m_N^k\left[A_0 \varepsilon_N\right]^{n-k},
	\end{equation}
	where $\varepsilon_N$ and $m_N$ are given by
	$$
	\varepsilon_N=\int_{|\xi| \geq N} \frac{1}{|\xi|^a} \mu(\ud \xi), \quad \text { and } \quad m_N=\mu(|\xi| \leq N),
	$$
	 for fixed large $N$.
\end{lemma}

\begin{lemma} \label{L:Wdefined}
	Under Assumption \ref{A:Main}, the limit
$$
\int_0^t \dot{W}(X(s)) \ud s \triangleq \lim _{\varepsilon \rightarrow 0^{+}} \int_0^t \dot{W}_{\varepsilon}(X(s)) \ud s \quad \text{ in } \mathcal{L}^2\left(\Omega, \mathcal{F}, \mathbb{P}_x \otimes \mathbb{P}\right)
$$
exists for every $t \geq 0$. Further, conditioned on the $a$-stable process $X(t)$, the process
	$$
	\int_0^t \dot{W}\left(X(s)\right) \ud s, \quad t \geq 0
	$$
	is mean-zero Gaussian with the (conditional) variance
	$$
	\mathbb{E}\left[\int_0^t \dot{W}\left(X(s)\right) \ud s\right]^2=\int_0^{t} \int_0^{t} \gamma\left(X(s)-X(r)\right) \ud s \ud r, \quad t \geq 0.
	$$
\end{lemma}
\begin{proof}
By Lemma \ref{L:n momentbound}, the second moment of the mollified integral is bounded uniformly in $\varepsilon$. The existence of the $\mathcal{L}^2$-limit then follows from a standard mollification procedure as in \cite{Chen2014Quenched}. We refer the reader to Lemma~A.1 therein for details and omit the routine argument here.
\end{proof}

\begin{lemma}\label{le:continu-density}
    Let $\{\mu_n\}_{n\geq1}$ be finite measures on $[0,\infty)^2$of
the form $\mu_n(\ud \mathbf{t})=f_n(t_1,t_2)\ud t_1\ud t_2$,
where $f_n:[0,\infty)^2\to[0,\infty)$ such that $\int_{0}^\infty\int_{0}^\infty f_n(t_1,t_2)\ud t_1\ud t_2\le C$ for all $n\ge1$.
Suppose that 
\begin{enumerate}[(i)]
    \item $\mu_n$ weakly converges to $\mu$ as $n\to\infty$.
    \item For every $t_1,t_2\in[0,\infty)$,
\begin{equation}\label{eq:pointwise-equicontinuity}
\lim_{\delta_1,\delta_2\to0}
\sup_{n\geq1}
\left|f_n(t_1+\delta_1,t_2+\delta_2)-f_n(t_1,t_2)\right|=0,
\end{equation}
where the limit is two-sided for $t>0$ and right-sided for $t=0$.
\end{enumerate}
Then there exists a continuous function
$f:[0,\infty)^2\to[0,\infty)$ such that $\mu(\ud\mathbf{t})=f(t_1,t_2)\ud t_1\ud t_2$.
In fact, $f_n\to f$
uniformly on every compact subset of $[0,\infty)^2$.
\end{lemma}

\begin{proof}
    Fix $t_1,t_2\geq0$. Applying \eqref{eq:pointwise-equicontinuity}, we can choose $r_{\mathbf{t}}>0$ such that
\[
f_n(s_1,s_2)\geq f_n(t_1,t_2)-1,
\qquad (s_1,s_2)\in I_{\mathbf{t}},\quad n\geq1.
\]
where
\[
I_{\mathbf{t}}:=(t_1-r_{\mathbf{t}},t_1+r_{\mathbf{t}})\times (t_2-r_{\mathbf{t}},t_2+r_{\mathbf{t}}) \cap [0,\infty)^2.
\]
Since $\mu_n$ are finite measures, we obtain
\begin{align*}
    &|I_t|\bigl(f_n(t_1,t_2)-1\bigr)
    =\int_{I_t}\bigl(f_n(t_1,t_2)-1\bigr)\,\ud s_1\ud s_2\\
&\le \int_{I_t}f_n(s_1,s_2)\,\ud s_1\ud s_2 
\le \int_{0}^\infty f_n(s_1,s_2)\,\ud s_1\ud s_2\le C,
\end{align*}
where $|I_{\mathbf{t}}|$ denotes the Lebesgue measure of $I_{\mathbf{t}}$. Therefore,
\[
\sup_{n\geq1}f_n(t_1,t_2)
\leq 1+\frac{C}{|I_{\mathbf{t}}|}
<\infty.
\]
Hence $\{f_n\}$ is pointwise bounded.

Let $K\subset[0,\infty)^2$ be compact. The pointwise equicontinuity condition implies that the family is uniformly equicontinuous on $K$ by standard compactness argument. Together with the pointwise bound, this gives
\[
\sup_{n\ge1}\sup_{\mathbf{t}\in K} f_n(\mathbf{t}) < \infty
\quad\text{and}\quad
\lim_{\eta\downarrow0}\sup_n\sup_{\substack{\mathbf{s},\mathbf{t}\in K\\|\mathbf{s}-\mathbf{t}|<\eta}} |f_n(\mathbf{s})-f_n(\mathbf{t})|=0.
\]
Thus $\{f_n\}$ is uniformly bounded and uniformly equicontinuous on each compact $K$.

By the Arzelà–Ascoli theorem and a diagonal argument, there exist a subsequence $\{f_{n_k}\}$ and a continuous function $f:[0,\infty)^2\to[0,\infty)$ such that $f_{n_k} \to f$ uniformly on every compact subset of $[0,\infty)$.

For any continuous function $\varphi$ with $\operatorname{supp}\varphi\subset[0,R]^2$ and $R>0$ . Then, using the uniform convergence on $[0,R]^2$,
\[
\int_0^\infty\int_0^\infty \varphi(t_1,t_2) f_{n_k}(t_1,t_2)\,\ud t_1\ud t_2 \longrightarrow \int_0^\infty\int_0^\infty \varphi(t_1,t_2) f(t_1,t_2)\,\ud t_1\ud t_2.
\]
Moreover, $\mu_{n_k}$ weakly converges to $\mu$ implies
\[
\int_0^\infty\int_0^\infty \varphi(t_1,t_2) f_{n_k}(t_1,t_2)\,\ud t_1\ud t_2 = \int_{[0,\infty)^2}\varphi\,\ud \mu_{n_k} \longrightarrow \int_{[0,\infty)^2}\varphi\,\ud\mu.
\]
Therefore,
\[
\int_{[0,\infty)^2}\varphi\,\ud\mu = \int_0^\infty\int_0^\infty \varphi(t_1,t_2) f(t_1,t_2)\,\ud t_1\ud t_2,
\quad\text{for all }\varphi\in C_c\left([0,\infty)^2 \right),
\]
so $\mu(\ud \mathbf{t})=f(t_1,t_2)\,\ud t_1\ud t_2$.

Suppose $f_n$ did not converge locally uniformly to $f$. Then there exist $R>0$, $\varepsilon>0$, and a subsequence $\{f_{n_j}\}$ such that
\begin{equation}\label{e:separation}
    \sup_{\mathbf{t}\in[0,R]^2} |f_{n_j}(t_1,t_2)-f(t_1,t_2)| \ge \varepsilon.
\end{equation}
By the same compactness argument, $\{f_{n_j}\}$ has a further subsequence converging locally uniformly to some continuous $g$. The condition that $\mu_{n_j}$ weakly converges to $\mu$ implies $\mu(\ud \mathbf{t})=g(t_1,t_2)\,\ud t_1\ud t_2$, hence $g=f$. This contradicts \eqref{e:separation}. Thus $f_n \to f$ uniformly on every compact subset of $[0,\infty)^2$. The proof is complete.
\end{proof}

\begin{lemma} \label{l:characteristic function}
Suppose Assumption \ref{A:Main} holds. 
Let $W$ and $\widetilde{W}$ be independent copies of the centered Gaussian noise with covariance \eqref{e:noise-covariance}, 
and let $X_1, X_2$ be independent symmetric $\alpha$-stable processes on $\mathbb{R}^d$ starting from $0$, independent of $W$ and $\widetilde{W}$. 
Then for any $n\in\mathbb{N}$, $t_1,t_2, a,b>0$, we have
\begin{align}
&\sum_{l_1+l_2= 2n}\frac{1}{l_1!l_2!} \notag\\
&\quad \times
\mathbb{E}\left[\mathbb{E}_{\widetilde{W}}\left[\int_0^{t_1} [a\dot{W}(X_1(s))+\iota b\dot{\widetilde{W}}(X_1(s))] \ud s\right]^{l_1} \cdot \mathbb{E}_{\widetilde{W}}\left[\int_0^{t_2} [a\dot{W}(X_2(s))+\iota b\dot{\widetilde{W}}(X_2(s))] \ud s\right]^{l_2}\right] \nonumber\\
&=\frac{1}{2^n n!}\left(a^2\sum_{j,k=1}^{2}\int_0^{t_j} \int_0^{t_k} \gamma(X_j(s)-X_k(r)) \ud s\ud r - b^2\sum_{i=1}^{2}\int_0^{t_i} \int_0^{t_i} \gamma(X_i(s)-X_i(r)) \ud s \ud r\right)^n, \nonumber
\end{align}
where the stochastic integrals with respect to $\dot{W}$ and $\dot{\widetilde{W}}$ are understood in the same sense as in Lemma~\ref{L:Wdefined}.
\end{lemma}
\begin{proof}
We prove the identity by expanding the quantity below into a Taylor series and comparing the coefficients. For any $\theta>0$,
$$
\begin{aligned}
&\mathbb{E}\left[\mathbb{E}_{\widetilde{W}}\exp \left\{\theta\int_0^{t_1} [a\dot{W}(X_1(s))+\iota b\dot{\widetilde{W}}(X_1(s))] \ud s\right\} \cdot \mathbb{E}_{\widetilde{W}}\exp \left\{ \theta\int_0^{t_2} [a\dot{W}(X_2(s))+\iota b\dot{\widetilde{W}}(X_2(s))] \ud s\right\}\right]\\
&=\mathbb{E}\Bigg[ \sum_{n=0}^{\infty}\frac{1}{n!}\theta^n\mathbb{E}_{\widetilde{W}}\left[\int_0^{t_1} [a\dot{W}(X_1(s))+\iota b\dot{\widetilde{W}}(X_1(s))] \ud s\right]^{n} \\
&\qquad\qquad\qquad\qquad\qquad \times\sum_{m=0}^{\infty}\frac{1}{m!}\theta^m\mathbb{E}_{\widetilde{W}}\left[\int_0^{t_2} [a\dot{W}(X_2(s))+\iota b\dot{\widetilde{W}}(X_2(s))] \ud s\right]^{m} \Bigg]\\
&=\sum_{n=0}^{\infty}\sum_{l_1+l_2= 2n}\theta^{2n}\frac{1}{l_1!l_2!}
\mathbb{E}\left[\mathbb{E}_{\widetilde{W}}\left[\int_0^{t_1} [\dot{W}(X_1(s))+\iota\rho\dot{\widetilde{W}}(X_1(s))] \ud s\right]^{l_1} \right.\\
&\qquad\qquad\qquad\qquad\qquad \qquad \left.\times \mathbb{E}_{\widetilde{W}}\left[\int_0^{t_2} [\dot{W}(X_2(s))+\iota\rho\dot{\widetilde{W}}(X_2(s))] \ud s\right]^{l_2}\right].
\end{aligned}
$$

On the other hand, taking the expectations with respect to $\widetilde{W}$ and $W$ first and then expanding into a Taylor series, we obtain
$$
\begin{aligned}
&\mathbb{E}\left[\mathbb{E}_{\widetilde{W}}\exp \left\{\theta\int_0^{t_1} [a\dot{W}(X_1(s))+\iota b\dot{\widetilde{W}}(X_1(s))] \ud s\right\} \cdot \mathbb{E}_{\widetilde{W}}\exp \left\{\theta\int_0^{t_2} [a\dot{W}(X_2(s))+\iota b\dot{\widetilde{W}}(X_2(s))] \ud s\right\}\right]\\
&= \exp\left\{\frac{\theta^2}{2} \Bigg[a^2\sum_{j,k=1}^{2}\int_0^{t_j} \int_0^{t_k} \gamma(X_j(s)-X_k(r)) \ud s \ud r-b^2\sum_{i=1}^{2}\int_0^{t_i} \int_0^{t_i} \gamma(X_i(s)-X_i(r)) \ud s\ud r\Bigg] \right\}\\
&= \sum_{n=0}^{\infty}\frac{\theta^{2n}}{2^n n!} \left[a^2 \sum_{j,k=1}^{2}\int_0^{t_j} \int_0^{t_k} \gamma(X_j(s)-X_k(r)) \ud s \ud r-b^2\sum_{i=1}^{2}\int_0^{t_i} \int_0^{t_i} \gamma(X_i(s)-X_i(r)) \ud s \ud r \right]^n.
\end{aligned}
$$

Comparing the coefficients of $\theta^{2n}$ in both power series expansions yields the desired identity. This completes the proof.
\end{proof}

\bibliographystyle{plain}
 \bibliography{ref}

@book{kst2006,
  title={Theory and Applications of Fractional Differential Equations},
  author={Kilbas,  Anatoly A. and Srivastava, Hari M. and Trujillo, Juan J.},
  volume={204},
  year={2006},
  publisher={Elsevier Science Limited}
}

@book{feller1971,
  title={An introduction to probability theory and its applications},
  author={Feller, William and others},
  year={1971},
  publisher={Wiley New York}
}

@article{lswz26,
  title={{On a fractional stochastic heat equation arising from the disordered pinning model}},
  author={Li, Zi'an and Song, Jian and Wei, Ran and Zhang, Hang},
  journal={arXiv preprint arXiv:2603.01823},
  year={2026}
}

@article{bs19,
  title={{Second order Lyapunov exponents for parabolic and hyperbolic Anderson models}},
  author={Balan, Raluca M and Song, Jian},
  journal={Bernoulli},
  volume={25},
  number={4A},
  pages={3069--3089},
  year={2019}
}

@article{bcc22,
  title={Exact asymptotics of the stochastic wave equation with time-independent noise},
  author={Balan, Raluca M and Chen, Le and Chen, Xia},
  journal={Annales de l'Institut Henri Poincare (B) Probabilites et statistiques},
  volume={58},
  number={3},
  pages={1590--1620},
  year={2022},
  organization={Institut Henri Poincar{\'e}}
}

@article{Bass2009,
  title={Large deviations for Riesz potentials of additive processes},
  author={Bass, R. and Chen, X. and Rosen, M.},
  journal={Annales de l'Institut Henri Poincar\'{e} Probabilit\'{e}s et Statistiques},
  volume={45},
  pages={626--666},
  year={2009}
}

@article{bc2014,
  title={A note on intermittency for the fractional heat equation},
  author={Balan, Raluca M and Conus, Daniel},
  journal={Statistics $\&$ Probability Letters},
  volume={95},
  pages={6--14},
  year={2014},
  publisher={Elsevier}
}

@article{Chen2014Quenched,
  author  = {Chen, X.},
  title   = {{Quenched asymptotics for Brownian motion in generalized Gaussian potential}},
  journal = {The Annals of Probability},
  volume  = {42},
  pages   = {576--622},
  year    = {2014}
}

@article {Chen2015Precise,
    AUTHOR = {Chen, Xia},
     TITLE = {Precise intermittency for the parabolic {A}nderson equation
              with an {$(1+1)$}-dimensional time-space white noise},
   JOURNAL = {Ann. Inst. Henri Poincar\'e{} Probab. Stat.},
  FJOURNAL = {Annales de l'Institut Henri Poincar\'e{} Probabilit\'es et
              Statistiques},
    VOLUME = {51},
      YEAR = {2015},
    NUMBER = {4},
     PAGES = {1486--1499},
      ISSN = {0246-0203,1778-7017},
   MRCLASS = {60F10 (60H15 60J65 81U10)},
  MRNUMBER = {3414455},
MRREVIEWER = {Leila\ Setayeshgar},
       DOI = {10.1214/15-AIHP673},
       URL = {https://doi.org/10.1214/15-AIHP673},
}

@article {Chen2017Moment,
    AUTHOR = {Chen, Xia},
     TITLE = {Moment asymptotics for parabolic {A}nderson equation with
              fractional time-space noise: in {S}korokhod regime},
   JOURNAL = {Ann. Inst. Henri Poincar\'e{} Probab. Stat.},
  FJOURNAL = {Annales de l'Institut Henri Poincar\'e{} Probabilit\'es et
              Statistiques},
    VOLUME = {53},
      YEAR = {2017},
    NUMBER = {2},
     PAGES = {819--841},
      ISSN = {0246-0203,1778-7017},
   MRCLASS = {60F10 (60H15 60H40 60J65 81U10)},
  MRNUMBER = {3634276},
MRREVIEWER = {Mathew\ Joseph},
       DOI = {10.1214/15-AIHP738},
       URL = {https://doi.org/10.1214/15-AIHP738},
}

@article{Chen2017Nonlinear,
    AUTHOR = {Chen, Le},
     TITLE = {Nonlinear stochastic time-fractional diffusion equations on
              {$\mathbb{R}$}: moments, {H}\"older regularity and intermittency},
   JOURNAL = {Trans. Amer. Math. Soc.},
  FJOURNAL = {Transactions of the American Mathematical Society},
    VOLUME = {369},
      YEAR = {2017},
    NUMBER = {12},
     PAGES = {8497--8535},
      ISSN = {0002-9947,1088-6850},
   MRCLASS = {60H15 (35R11 35R60 60G60)},
  MRNUMBER = {3710633},
MRREVIEWER = {Feng-Yu\ Wang},
       DOI = {10.1090/tran/6951},
       URL = {https://doi.org/10.1090/tran/6951},
}

@article {Chen2017Spatial,
    AUTHOR = {Chen, Xia and Hu, Yaozhong and Nualart, David and Tindel,
              Samy},
     TITLE = {Spatial asymptotics for the parabolic {A}nderson model driven
              by a {G}aussian rough noise},
   JOURNAL = {Electron. J. Probab.},
  FJOURNAL = {Electronic Journal of Probability},
    VOLUME = {22},
      YEAR = {2017},
     PAGES = {Paper No. 65, 38},
      ISSN = {1083-6489},
   MRCLASS = {60G15 (60H07 60H15)},
  MRNUMBER = {3690290},
MRREVIEWER = {Peter\ Karl\ Friz},
       DOI = {10.1214/17-EJP83},
       URL = {https://doi.org/10.1214/17-EJP83},
}

@article {Chen2018Temporal,
    AUTHOR = {Chen, Xia and Hu, Yaozhong and Song, Jian and Song, Xiaoming},
     TITLE = {Temporal asymptotics for fractional parabolic {A}nderson
              model},
   JOURNAL = {Electron. J. Probab.},
  FJOURNAL = {Electronic Journal of Probability},
    VOLUME = {23},
      YEAR = {2018},
     PAGES = {Paper No. 14, 39},
      ISSN = {1083-6489},
   MRCLASS = {60H15 (60F10 60G15 60G52)},
  MRNUMBER = {3771751},
MRREVIEWER = {Robert\ C.\ Dalang},
       DOI = {10.1214/18-EJP139},
       URL = {https://doi.org/10.1214/18-EJP139},
}

@article {Chen2019Nonlinear,
    AUTHOR = {Chen, Le and Hu, Yaozhong and Nualart, David},
     TITLE = {Nonlinear stochastic time-fractional slow and fast diffusion
              equations on {$\mathbb R^d$}},
   JOURNAL = {Stochastic Process. Appl.},
  FJOURNAL = {Stochastic Processes and their Applications},
    VOLUME = {129},
      YEAR = {2019},
    NUMBER = {12},
     PAGES = {5073--5112},
      ISSN = {0304-4149,1879-209X},
   MRCLASS = {60H15 (35R11 35R60 60G60)},
  MRNUMBER = {4025700},
MRREVIEWER = {Latifa\ Debbi},
       DOI = {10.1016/j.spa.2019.01.003},
       URL = {https://doi.org/10.1016/j.spa.2019.01.003},
}

@article {Chen2019Parabolic,
    AUTHOR = {Chen, Xia},
     TITLE = {Parabolic {A}nderson model with rough or critical {G}aussian
              noise},
   JOURNAL = {Ann. Inst. Henri Poincar\'e{} Probab. Stat.},
  FJOURNAL = {Annales de l'Institut Henri Poincar\'e{} Probabilit\'es et
              Statistiques},
    VOLUME = {55},
      YEAR = {2019},
    NUMBER = {2},
     PAGES = {941--976},
      ISSN = {0246-0203,1778-7017},
   MRCLASS = {60H15 (60G22 60H07 60H40 60J65)},
  MRNUMBER = {3949959},
MRREVIEWER = {Georgiy\ M.\ Shevchenko},
       DOI = {10.1214/18-aihp904},
       URL = {https://doi.org/10.1214/18-aihp904},
}

@article {Chen2020Parabolic,
    AUTHOR = {Chen, Xia},
     TITLE = {Parabolic {A}nderson model with a fractional {G}aussian noise
              that is rough in time},
   JOURNAL = {Ann. Inst. Henri Poincar\'e{} Probab. Stat.},
  FJOURNAL = {Annales de l'Institut Henri Poincar\'e{} Probabilit\'es et
              Statistiques},
    VOLUME = {56},
      YEAR = {2020},
    NUMBER = {2},
     PAGES = {792--825},
      ISSN = {0246-0203,1778-7017},
   MRCLASS = {60F10 (60H15 60H40 60J65)},
  MRNUMBER = {4076766},
MRREVIEWER = {Kunwoo\ Kim},
       DOI = {10.1214/19-AIHP983},
       URL = {https://doi.org/10.1214/19-AIHP983},
}

@article{chen2025,
  title={{Time-dependency in hyperbolic Anderson model: Stratonovich regime}},
  author={Chen, Xia},
  journal={arXiv preprint arXiv:2510.01412},
  year={2025}
}

@article{cdst2021,
  title={Solving the hyperbolic {A}nderson model 1: {S}korohod setting},
  author={Chen, Xia and Deya, Aur{\'e}lien and Song, Jian and Tindel, Samy},
  journal={Annales de l’Institut Henri Poincaré - Probabilités et Statistiques},
  volume={61},
  number={3},
  pages={1794--1814},
  year={2025}
}

@article {Chen2023Interpolating,
    AUTHOR = {Chen, Le and Eisenberg, Nicholas},
     TITLE = {Interpolating the stochastic heat and wave equations with
              time-independent noise: solvability and exact asymptotics},
   JOURNAL = {Stoch. Partial Differ. Equ. Anal. Comput.},
  FJOURNAL = {Stochastic Partial Differential Equations. Analysis and
              Computations},
    VOLUME = {11},
      YEAR = {2023},
    NUMBER = {3},
     PAGES = {1203--1253},
      ISSN = {2194-0401,2194-041X},
   MRCLASS = {60H15 (35R11 37H15 60H07)},
  MRNUMBER = {4624137},
       DOI = {10.1007/s40072-022-00258-6},
       URL = {https://doi.org/10.1007/s40072-022-00258-6},
}

@article{ch2024,
  title={{Hyperbolic Anderson equations with general time-independent Gaussian noise: Stratonovich regime}},
  author={Chen, Xia and Hu, Yaozhong},
  journal={The Annals of Probability},
  volume={53},
  number={6},
  pages={2144--2195},
  year={2025},
  publisher={Institute of Mathematical Statistics}
}

@article {Chen2024Moments,
    AUTHOR = {Chen, Le and Guo, Yuhui and Song, Jian},
     TITLE = {Moments and asymptotics for a class of {SPDE}s with space-time
              white noise},
   JOURNAL = {Trans. Amer. Math. Soc.},
  FJOURNAL = {Transactions of the American Mathematical Society},
    VOLUME = {377},
      YEAR = {2024},
    NUMBER = {6},
     PAGES = {4255--4301},
      ISSN = {0002-9947,1088-6850},
   MRCLASS = {60H15 (26A33 37H15 60G60 60H07)},
  MRNUMBER = {4748620},
MRREVIEWER = {Lucio\ Galeati},
       DOI = {10.1090/tran/9138},
       URL = {https://doi.org/10.1090/tran/9138},
}

@incollection {Conus2013Intermittency,
    AUTHOR = {Conus, Daniel and Joseph, Mathew and Khoshnevisan, Davar and
              Shiu, Shang-Yuan},
     TITLE = {Intermittency and chaos for a nonlinear stochastic wave
              equation in dimension 1},
 BOOKTITLE = {Malliavin calculus and stochastic analysis},
    SERIES = {Springer Proc. Math. Stat.},
    VOLUME = {34},
     PAGES = {251--279},
 PUBLISHER = {Springer, New York},
      YEAR = {2013},
      ISBN = {978-1-4614-5906-4; 978-1-4614-5905-7},
   MRCLASS = {60H15 (60H05 60H20)},
  MRNUMBER = {3070447},
MRREVIEWER = {Stefan\ Tappe},
       DOI = {10.1007/978-1-4614-5906-4\_11},
       URL = {https://doi.org/10.1007/978-1-4614-5906-4_11},
}

@article{dalang1999,
  title={Extending the martingale measure stochastic integral with applications to spatially homogeneous spde's},
  author={Dalang, Robert},
Journal={Electronic Journal of Probability},
  year={1999}
}

@article {Dalang2009Intermittency,
    AUTHOR = {Dalang, Robert C. and Mueller, Carl},
     TITLE = {Intermittency properties in a hyperbolic {A}nderson problem},
   JOURNAL = {Ann. Inst. Henri Poincar\'e{} Probab. Stat.},
  FJOURNAL = {Annales de l'Institut Henri Poincar\'e{} Probabilit\'es et
              Statistiques},
    VOLUME = {45},
      YEAR = {2009},
    NUMBER = {4},
     PAGES = {1150--1164},
      ISSN = {0246-0203,1778-7017},
   MRCLASS = {60H15 (35L05 35R60)},
  MRNUMBER = {2572169},
MRREVIEWER = {Athanasios\ Yannacopoulos},
       DOI = {10.1214/08-AIHP199},
       URL = {https://doi.org/10.1214/08-AIHP199},
}

@article{Guo2024Stochastic,
    AUTHOR = {Guo, Yuhui and Song, Jian and Song, Xiaoming},
     TITLE = {Stochastic fractional diffusion equations with {G}aussian
              noise rough in space},
   JOURNAL = {Bernoulli},
  FJOURNAL = {Bernoulli. Official Journal of the Bernoulli Society for
              Mathematical Statistics and Probability},
    VOLUME = {30},
      YEAR = {2024},
    NUMBER = {3},
     PAGES = {1774--1799},
      ISSN = {1350-7265,1573-9759},
   MRCLASS = {60H15 (26A33 35G10 35R11 35R60 60G22 60G60 60H07)},
  MRNUMBER = {4746588},
MRREVIEWER = {Derui\ Sheng},
       DOI = {10.3150/23-bej1652},
       URL = {https://doi.org/10.3150/23-bej1652},
}

@article {Guo2025Sample,
    AUTHOR = {Guo, Yuhui and Song, Jian and Wang, Ran and Xiao, Yimin},
     TITLE = {Sample path properties and small ball probabilities for
              stochastic fractional diffusion equations},
   JOURNAL = {J. Differential Equations},
  FJOURNAL = {Journal of Differential Equations},
    VOLUME = {446},
      YEAR = {2025},
     PAGES = {Paper No. 113604, 56},
      ISSN = {0022-0396,1090-2732},
   MRCLASS = {60H15 (33E12 60G15 60G17 60G22)},
  MRNUMBER = {4930468},
       DOI = {10.1016/j.jde.2025.113604},
       URL = {https://doi.org/10.1016/j.jde.2025.113604},
}

@book {Hu2017Analysis,
    AUTHOR = {Hu, Yaozhong},
     TITLE = {Analysis on {G}aussian spaces},
 PUBLISHER = {World Scientific Publishing Co. Pte. Ltd., Hackensack, NJ},
      YEAR = {2017},
     PAGES = {xi+470},
      ISBN = {978-981-3142-17-6},
   MRCLASS = {60G15 (28C20 46-02 46E27 60B11 60H07 60Hxx)},
  MRNUMBER = {3585910},
MRREVIEWER = {Jan\ van Neerven},
}

@article {Hu2024Matching,
    AUTHOR = {Hu, Yaozhong and Wang, Xiong},
     TITLE = {Matching upper and lower moment bounds for a large class of
              stochastic {PDE}s driven by general space-time {G}aussian
              noises},
   JOURNAL = {Stoch. Partial Differ. Equ. Anal. Comput.},
  FJOURNAL = {Stochastics and Partial Differential Equations. Analysis and
              Computations},
    VOLUME = {12},
      YEAR = {2024},
    NUMBER = {1},
     PAGES = {1--52},
      ISSN = {2194-0401,2194-041X},
   MRCLASS = {60H15 (26A33 30E05 35R11 35R60 37H15 60H07)},
  MRNUMBER = {4709538},
MRREVIEWER = {David\ Nualart},
       DOI = {10.1007/s40072-022-00278-2},
       URL = {https://doi.org/10.1007/s40072-022-00278-2},
}

@article {Huang2017Large,
    AUTHOR = {Huang, Jingyu and L\^e, Khoa and Nualart, David},
     TITLE = {Large time asymptotics for the parabolic {A}nderson model
              driven by spatially correlated noise},
   JOURNAL = {Ann. Inst. Henri Poincar\'e{} Probab. Stat.},
  FJOURNAL = {Annales de l'Institut Henri Poincar\'e{} Probabilit\'es et
              Statistiques},
    VOLUME = {53},
      YEAR = {2017},
    NUMBER = {3},
     PAGES = {1305--1340},
      ISSN = {0246-0203,1778-7017},
   MRCLASS = {60G15 (60F10 60H07 60H15 65M75)},
  MRNUMBER = {3689969},
MRREVIEWER = {Xia\ Chen},
       DOI = {10.1214/16-AIHP756},
       URL = {https://doi.org/10.1214/16-AIHP756},
}

@article {Huang2017Large2,
    AUTHOR = {Huang, Jingyu and L\^e, Khoa and Nualart, David},
     TITLE = {Large time asymptotics for the parabolic {A}nderson model
              driven by space and time correlated noise},
   JOURNAL = {Stoch. Partial Differ. Equ. Anal. Comput.},
  FJOURNAL = {Stochastics and Partial Differential Equations. Analysis and
              Computations},
    VOLUME = {5},
      YEAR = {2017},
    NUMBER = {4},
     PAGES = {614--651},
      ISSN = {2194-0401,2194-041X},
   MRCLASS = {60G15 (35B40 35K15 35R60 60F10 60H07 60H15 65M75)},
  MRNUMBER = {3736656},
MRREVIEWER = {Xia\ Chen},
       DOI = {10.1007/s40072-017-0099-0},
       URL = {https://doi.org/10.1007/s40072-017-0099-0},
}

@book {Kilbas2004H,
    AUTHOR = {Kilbas, Anatoly A. and Saigo, Megumi},
     TITLE = {{$H$}-transforms},
    SERIES = {Analytical Methods and Special Functions},
    VOLUME = {9},
      NOTE = {Theory and applications},
 PUBLISHER = {Chapman \& Hall/CRC, Boca Raton, FL},
      YEAR = {2004},
     PAGES = {xii+389},
      ISBN = {0-415-29916-0},
   MRCLASS = {44-02 (33C60 44A15 47G10)},
  MRNUMBER = {2041257},
MRREVIEWER = {A.\ Corduneanu},
       DOI = {10.1201/9780203487372},
       URL = {https://doi.org/10.1201/9780203487372},
}

@article {Le2016Aremark,
    AUTHOR = {L\^e, Khoa},
     TITLE = {A remark on a result of {X}ia {C}hen},
   JOURNAL = {Statist. Probab. Lett.},
  FJOURNAL = {Statistics \& Probability Letters},
    VOLUME = {118},
      YEAR = {2016},
     PAGES = {124--126},
      ISSN = {0167-7152,1879-2103},
   MRCLASS = {60H15 (60J65)},
  MRNUMBER = {3531492},
MRREVIEWER = {Le\ Chen},
       DOI = {10.1016/j.spl.2016.06.004},
       URL = {https://doi.org/10.1016/j.spl.2016.06.004},
}

@book {Nualart2006TheMalliavin,
    AUTHOR = {Nualart, David},
     TITLE = {The {M}alliavin calculus and related topics},
    SERIES = {Probability and its Applications (New York)},
   EDITION = {Second},
 PUBLISHER = {Springer-Verlag, Berlin},
      YEAR = {2006},
     PAGES = {xiv+382},
      ISBN = {978-3-540-28328-7; 3-540-28328-5},
   MRCLASS = {60-02 (60H07 60H30)},
  MRNUMBER = {2200233},
MRREVIEWER = {Daniel\ Ocone},
}

@book {Podlubny1999Fractional,
    AUTHOR = {Podlubny, Igor},
     TITLE = {Fractional differential equations},
    SERIES = {Mathematics in Science and Engineering},
    VOLUME = {198},
      NOTE = {An introduction to fractional derivatives, fractional
              differential equations, to methods of their solution and some
              of their applications},
 PUBLISHER = {Academic Press, Inc., San Diego, CA},
      YEAR = {1999},
     PAGES = {xxiv+340},
      ISBN = {0-12-558840-2},
   MRCLASS = {26A33 (34K05)},
  MRNUMBER = {1658022},
MRREVIEWER = {Anatoly\ Kilbas},
}

@article{song2017,
  title={On a class of stochastic partial differential equations},
  author={Song, Jian},
  journal={Stochastic Processes and their Applications},
  volume={127},
  number={1},
  pages={37--79},
  year={2017},
  publisher={Elsevier}
}

@article{Chenbook,
  title={Random Walk Intersections: Large Deviations and Related Topics},
  author={ Chen, Xia },
  journal={Mathematical Surveys  Monographs Ams},
  volume={157},
  number={4},
  pages={823-825},
  year={2010}
}

@incollection {HM1988,
    AUTHOR = {Hu, Y. Z. and Meyer, P.-A.},
     TITLE = {Sur les int\'egrales multiples de {S}tratonovitch},
 BOOKTITLE = {S\'eminaire de {P}robabilit\'es, {XXII}},
    SERIES = {Lecture Notes in Math.},
    VOLUME = {1321},
     PAGES = {72--81},
 PUBLISHER = {Springer, Berlin},
      YEAR = {1988},
      ISBN = {3-540-19351-0},
   MRCLASS = {60H05 (28C20 58D20 60B05 60J65)},
  MRNUMBER = {960509},
MRREVIEWER = {Sheng\ Wu\ He},
       DOI = {10.1007/BFb0084119},
       URL = {https://doi.org/10.1007/BFb0084119},
}

@book {Mainardi2010Fractional,
    AUTHOR = {Mainardi, Francesco},
     TITLE = {Fractional calculus and waves in linear viscoelasticity},
      NOTE = {An introduction to mathematical models},
 PUBLISHER = {Imperial College Press, London},
      YEAR = {2010},
     PAGES = {xx+347},
      ISBN = {978-1-84816-329-4; 1-84816-329-0},
   MRCLASS = {74-02 (26A33 74Cxx 74D05)},
  MRNUMBER = {2676137},
       DOI = {10.1142/9781848163300},
       URL = {https://doi.org/10.1142/9781848163300},
}

@book {MR2006,
    AUTHOR = {Marcus, Michael B. and Rosen, Jay},
     TITLE = {Markov processes, {G}aussian processes, and local times},
    SERIES = {Cambridge Studies in Advanced Mathematics},
    VOLUME = {100},
 PUBLISHER = {Cambridge University Press, Cambridge},
      YEAR = {2006},
     PAGES = {x+620},
      ISBN = {978-0-521-86300-1; 0-521-86300-7},
   MRCLASS = {60-02 (60G15 60J25 60J55)},
  MRNUMBER = {2250510},
MRREVIEWER = {Nathalie\ Eisenbaum},
       DOI = {10.1017/CBO9780511617997},
       URL = {https://doi.org/10.1017/CBO9780511617997},
}

@article {Zel1987Intermittency,
    AUTHOR = {Zel'dovich, Ya. B. and Molchanov, S. A. and Ruzmaĭkin, A. A. and Sokolov, D. D.},
     TITLE = {Intermittency in random media},
JOURNAL = {Soviet Phys. Uspekhi},
FJOURNAL = {Soviet Physics. Uspekhi},
VOLUME = {30},
YEAR = {1987},
NUMBER = {5},
     PAGES = {353--369},
ISSN = {0038-5670},
   MRCLASS = {82A42 (80A30 92A40)},
  MRNUMBER = {921018},
       DOI = {10.1070/PU1987v030n05ABEH002867},
       URL = {https://doi.org/10.1070/PU1987v030n05ABEH002867},
}

@article {chsx2015,
    AUTHOR = {Chen, Xia and Hu, Yaozhong and Song, Jian and Xing, Fei},
     TITLE = {Exponential asymptotics for time-space {H}amiltonians},
   JOURNAL = {Ann. Inst. Henri Poincar\'e{} Probab. Stat.},
  FJOURNAL = {Annales de l'Institut Henri Poincar\'e{} Probabilit\'es et
              Statistiques},
    VOLUME = {51},
      YEAR = {2015},
    NUMBER = {4},
     PAGES = {1529--1561},
      ISSN = {0246-0203,1778-7017},
   MRCLASS = {60J65 (60F10 60G55 60K37 60K40)},
  MRNUMBER = {3414457},
MRREVIEWER = {Francis\ Comets},
       DOI = {10.1214/13-AIHP588},
       URL = {https://doi.org/10.1214/13-AIHP588},
}

 \end{document}